\documentclass[a4paper, 11pt]{article}

\def\RR{\mathbb R}
\def\NN{\mathbb N}
\def\HH{\mathbb H}

\usepackage[english]{babel}
\usepackage{csquotes}
\usepackage{amsmath, amsfonts,amsthm}
\usepackage{graphicx}
\usepackage{fullpage}
\usepackage{color}
\usepackage[affil-it]{authblk}
\usepackage{paralist}
\usepackage{scrextend}

\newtheorem{theorem}{Theorem}
\newtheorem{lemma}{Lemma}

\newtheorem{remark}{Remark}
\newtheorem{propos}{Proposition}
\newcommand\supp{\mathop{\rm supp}}
\DeclareMathOperator*{\argmin}{\arg\!\min}
\newcommand{\sgn}{\operatorname{sgn}}
\newcommand{\wlim}{\operatorname{w-lim}}
\newcommand{\diag}{\operatorname{diag}}
\title{Minimization of multi-penalty functionals by alternating iterative thresholding and optimal parameter choices}
\date{}
\author{Valeriya Naumova}
\affil{Simula Research Laboratory, \\Martin Linges vei 17, 1364 Fornebu, Norway}

\author{Steffen Peter}
\affil{Faculty of Mathematics, Technical University Munich,\\ Boltzmannstrasse 3, 85748
Garching, Germany}

\begin{document}

\maketitle
\begin{abstract}
Inspired by several recent  developments in regularization theory, optimization, and signal processing,  we present and analyze a numerical approach to multi-penalty regularization in spaces of sparsely represented functions. The sparsity prior is motivated by the largely expected geometrical/structured features of high-dimensional data, which may not be well-represented in the framework of typically more isotropic Hilbert spaces. 
In this paper, we are particularly interested in regularizers which are able to correctly model and separate the multiple components of additively mixed signals. This situation is rather common as pure signals may be corrupted by additive noise. 
To this end, we consider a regularization functional composed by a data-fidelity term, where signal and noise are additively mixed, a non-smooth and  non-convex sparsity promoting term, and a penalty term to model the noise.
We propose and analyze the convergence of an iterative alternating algorithm based on simple iterative thresholding steps to perform the minimization of the functional. 
By means of this algorithm, we explore the effect of choosing different regularization parameters and penalization norms in terms of the quality of recovering the pure signal and separating it from additive noise. For a given fixed noise level numerical experiments confirm a significant improvement in performance compared to standard one-parameter regularization methods. By using high-dimensional data analysis methods such as Principal Component Analysis, we are able to show the correct geometrical clustering of regularized solutions around the expected solution. Eventually, for the compressive sensing problems considered in our experiments we provide a guideline for a choice of regularization norms and parameters.

\end{abstract}
\section{Introduction}

In several interesting real-life problems we do not dispose directly of the quantity of interest, but only data from indirect observations are given. Additionally or alternatively to possible noisy data, the original signal to be recovered may be affected by its own 
noise. In this case, the reconstruction problem can be understood as an inverse problem, where the solution $x^\dag$ 
consists of two components of different nature, the relevant signal and its noise, to be separated.

Let us now formulate mathematically the situation so concisely described.  Let $\mathcal K$ and $\mathcal H$ be (separable) Hilbert spaces and $T: \mathcal K \rightarrow \mathcal H$ be a bounded linear operator.
For the moment we do not specify $T$ further. 
To begin, we consider a model problem of the type
\begin{equation}
	\label{unmixing_pproblem}
	y = T (u^\dag+v^\dag),
\end{equation}
where $u^\dag, v^\dag$ are the two components of the solution $x^\dag$ which we wish to identify and to separate.  In general, this unmixing problem
has clearly an infinite number of solutions. In fact, let us define the operator
$$
S:\mathcal K \times \mathcal K \to \mathcal H, \quad 
S \left ( 
\begin{array}{l} u \\ 
v 
\end{array} 
\right )
:= T (u + v).
$$
Its kernel is given by 
$$
\ker S= \left \{ \left ( 
\begin{array}{l} u \\ 
v 
\end{array} 
\right ) \in  \mathcal K \times \mathcal K : v= - u + \xi, \quad \xi \in \ker T \right \}.
$$
If $T$ had closed range then $S$ would have closed range and the operator
$$
S/\sim:(\mathcal K \times \mathcal K)/\ker S \to \mathcal H, \quad S \left ( \left [\left ( 
\begin{array}{l} u \\ 
v 
\end{array} 
\right ) \right ]_\sim \right ) \mapsto T(u+v),
$$
would be boundedly invertible on the new restricted quotient space $(\mathcal K \times \mathcal K)/\ker S$
of the equivalence classes given by
$$
\left ( 
\begin{array}{l} u \\ 
v 
\end{array} 
\right ) \sim \left ( 
\begin{array}{l} u' \\ 
v' 
\end{array} 
\right )
\mbox{ if and only if } (v -v')+(u-u') \in \ker T.
$$
Unfortunately, even in this well-posed setting, each of these equivalence classes is huge, and very different representatives can be picked as solutions.
In order to facilitate the choice, one may want to impose additional conditions on the solutions according to the expected structure. As mentioned above,
we may wish to distinguish a relevant component $u^\dag$ of the solution from a noise component $v^\dag$. Hence, in this paper we focus on the situation
where $u^\dag$ can be actually represented as a sparse vector considered as coordinates with respect to a certain orthogonal basis  in $\mathcal K$, and $v^\dag$ has bounded coefficients
up to a certain noise level $\eta>0$ with respect to the {\it same} basis. For the sake of simplicity, we shall identify below vectors in $\mathcal K$
with their Fourier coefficients in $\ell_2$ with respect to the fixed orthonormal basis. 
Let us stress that considering different reference bases is also a very
interesting setting when it comes to separation of components, but it will not be considered for the moment within the scope of this paper.\\
As a very simple and instructive example of the situation described so far, let us assume $\mathcal K = \mathcal H = \mathbb R^2$ and $T=I$
is the identity operator. Under the assumptions on the structure of the interesting solution $y=x^\dag = u^\dag + v^\dag$, without loss
of generality we write $u^\dag = ( u_1^\dag, 0)$ for $R= u_1^\dag >0$ and $\max\{|v_1^\dag|, |v_2^\dag|\} =\eta = |y_2|>0$. We consider now the 
following constrained problem: depending on the choice of $R>0$, find $u, v \in \mathbb R^2$ such that
\begin{equation*}
 \mathcal P(R) \quad \quad \quad u \in B_{\ell_p}(R), v \in B_{\ell_q}(|y_2|) \mbox{ subject to } u + v = y,
\end{equation*}
where $q=\infty$ and $0<p<1$.

\begin{figure}[htp]
\begin{center}
\includegraphics[width=2.5in]{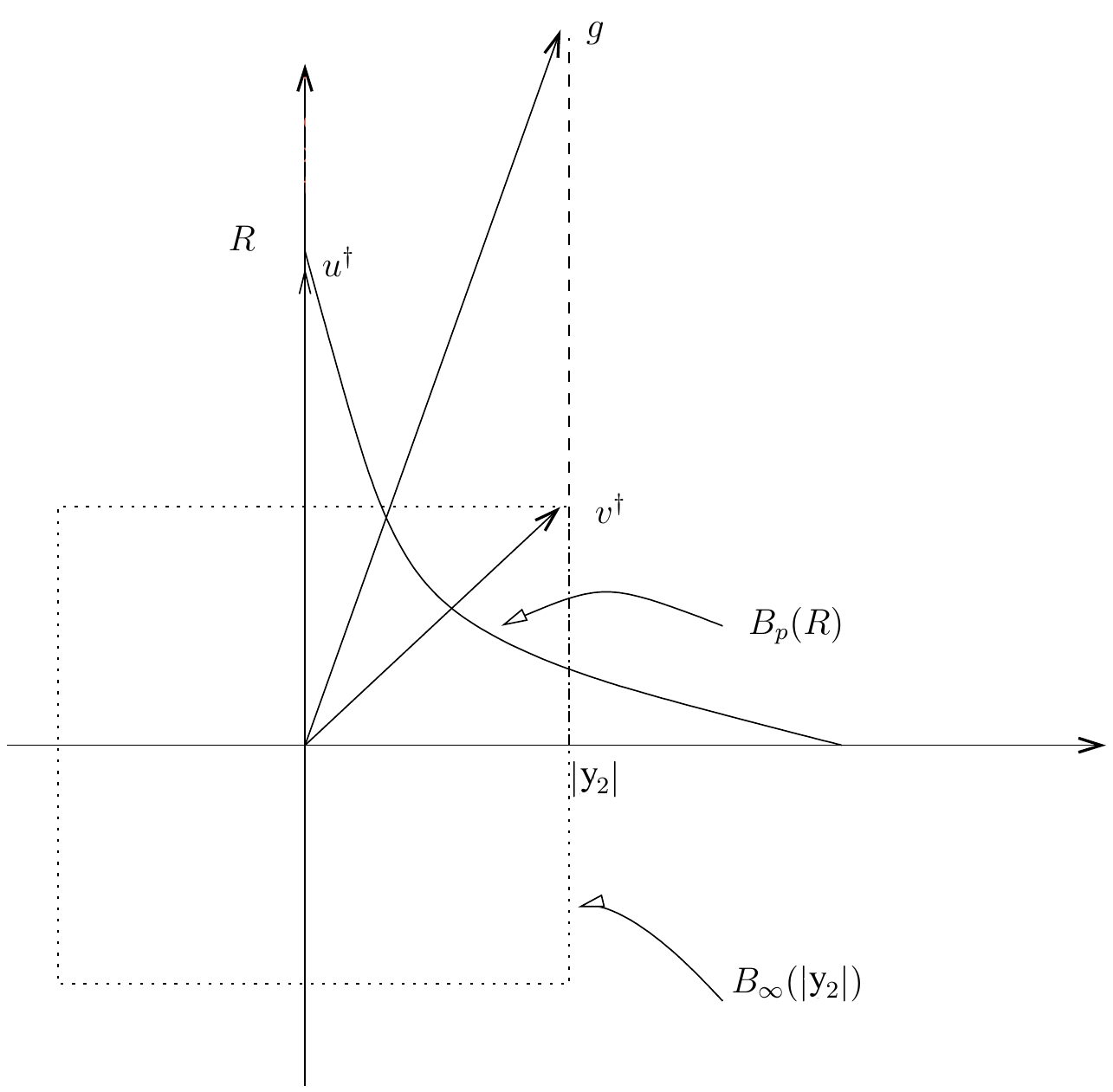}
\end{center}
\caption{Geometrical interpretation of the problem in 2D.}
\label{fig:biribi}
\end{figure}

Simple geometrical arguments, as illustrated in Figure \ref{fig:biribi}, yield to the existence of a special 
radius $R^*=R^*(\eta,p)>0$ for which only three situations can occur:
\begin{itemize}
\item If $R<R^*$ then problem $\mathcal P(R)$ has no solutions;
\item If $R>R^*$ then there are infinitely many solutions of $\mathcal P(R)$ and the larger $R$ is, the larger is the set of
solutions (in measure theoretical sense), including many possible non-sparse solutions in terms of the
$u$ component;
\item If $R=R^*$ there is only one solution for the problem $\mathcal P(R),$ whose $u^\dag$ components are actually sparse.
\end{itemize}
Hence, once the noise level $\eta$ on the solution is fixed, the parameter $R>0$ can be actually seen as a regularization parameter of the problem, which is smoothly going
from the situation where no solution exists, to the situation where there are many solutions, going through the
well-posed situation where there is actually only one solution. 
In order to promote uniqueness, one may also reformulate $\mathcal P(R)$ in terms of the following optimization problem, depending on $R>0$ and an additional parameter $\beta>0$:
\begin{equation*}
 \mathcal P^{\operatorname{opt}}(R,\beta) \quad \quad \quad \argmin_{%
      \substack{
u \in  B_{\ell_p}(R),\\
 v \in  B_{\ell_q}(|y_2|)
      }
    }  \|u\|_{\ell_p}^p +  \beta \|v\|_{\ell_q}^q  \mbox{ subject to } u + v = y.
\end{equation*}
 (Here and later we make an abuse of notation by assuming the convention that $\| \cdot \|_{\ell_q}^q = \| \cdot \|_{\ell_q}$ as soon as $q = \infty$.) 
The finite dimensional constrained problem $\mathcal P(R)$ or its constrained optimization version $\mathcal P^{\operatorname{opt}}(R,\beta)$ can be also recast in Lagrangian form as follows:

\begin{equation*}
 \mathcal P(\alpha,\beta) \quad \quad \quad \argmin_{u,v} \|u + v - y\|_{\ell_2}^2 + \alpha \| u\|_{\ell_p}^p + \beta \|v\|_{\ell_q}^q.
\end{equation*}
  
Due to the equivalence of the problem $\mathcal P(R)$ with  a problem of the type $\mathcal P(\alpha,\beta)$ for suitable
$\alpha=\alpha(R)>0$, $\beta=\beta(R)>0,$ we infer the existence of a parameter choice $(\alpha^*,\beta^*)$ for which
$\mathcal P(\alpha^*,\beta^*)$ has actually a unique solution $(u^\dag,v^\dag)$ such that $y = u^\dag + v^\dag$. 
For other choices there might be infinitely many solutions $(u,v)$ for which $\| u + v -y \|_{\ell_2}^2 \geq 0$. 
While the solution in $\mathbb R^2$ of the problem $\mathcal P(R)$ follows by simple geometrical arguments, in higher dimension
the form $\mathcal P(\alpha,\beta)$ may allow us to explore solutions via a rather simple algorithm based on alternating minimizations:
We shall consider the following iteration, starting from $u^{(0)}=0=v^{(0)}$,
\begin{eqnarray*}
u^{(n+1)} &=& \argmin_{u} \|u + v^{(n)} - y\|_{\ell_2}^2 + \alpha \| u\|_{\ell_p}^p,\\
v^{(n+1)} &=& \argmin_{v} \|u^{(n+1)} + v - y\|_{\ell_2}^2 + \beta \| v\|_{\ell_q}^q.
\end{eqnarray*}
As we shall see in details in this paper, both these two steps are explicitly solved by means of simple thresholding operations, making this
algorithm extremely fast and easy to implement. As we will show in Theorem~\ref{theorem:minimizers} of this article, the  algorithm above converges to a solution of $\mathcal P(\alpha,\beta)$ in the case of $p=1$ and at least to a local minimal solution in the case of $0<p<1$. To get an impression about the operating principle of this alternating algorithm, in the following, we present the results of  representative 2D experiments. To this end, we fix $y = \left(0.3, 1.35\right)^T$, and consider $0 \leq p < 2$ in order to promote sparsity in $u^{\dagger}$,  $ q \geq 2$ in order to obtain a non-sparse $v^{\dagger}$. \\

First, consider the case $p=1$. Due to the strict convexity of $\mathcal P(\alpha,\beta)$ for $p=1$ and $q \geq 2$, the computed minimizer is unique.  In Figure~\ref{fig:plotp1fixed} we visually estimate the \emph{regions of solutions} for $u^{\dagger}$ and $v^{\dagger}$, which we define as
\[\mathcal{R}_{p,q}^u:=\left\{u^{\dagger}|(u^{\dagger},v^{\dagger}) \text{ is the solution of } \mathcal P(\alpha,\beta), \text{ for } \alpha,\beta>0\right\}, \]
\[\mathcal{R}_{p,q}^v:=\left\{v^{\dagger}|(u^{\dagger},v^{\dagger}) \text{ is the solution of } \mathcal P(\alpha,\beta), \text{ for } \alpha,\beta>0\right\},\]
by $\times$- and $\ast$-markers respectively. Notice that this plot does not contain a visualization of the information of which $u^{\dagger}$ belongs to which  $v^{\dagger}$. In the three plots for $q\in\left\{2,4,\infty\right\}$, the above sets are discretized by showing the solutions for all possible pairs of $\alpha,\beta\in\{0.1\cdot i|i=1,\ldots,20\}$.

\begin{figure}[ht]
\includegraphics[width=0.32\textwidth]{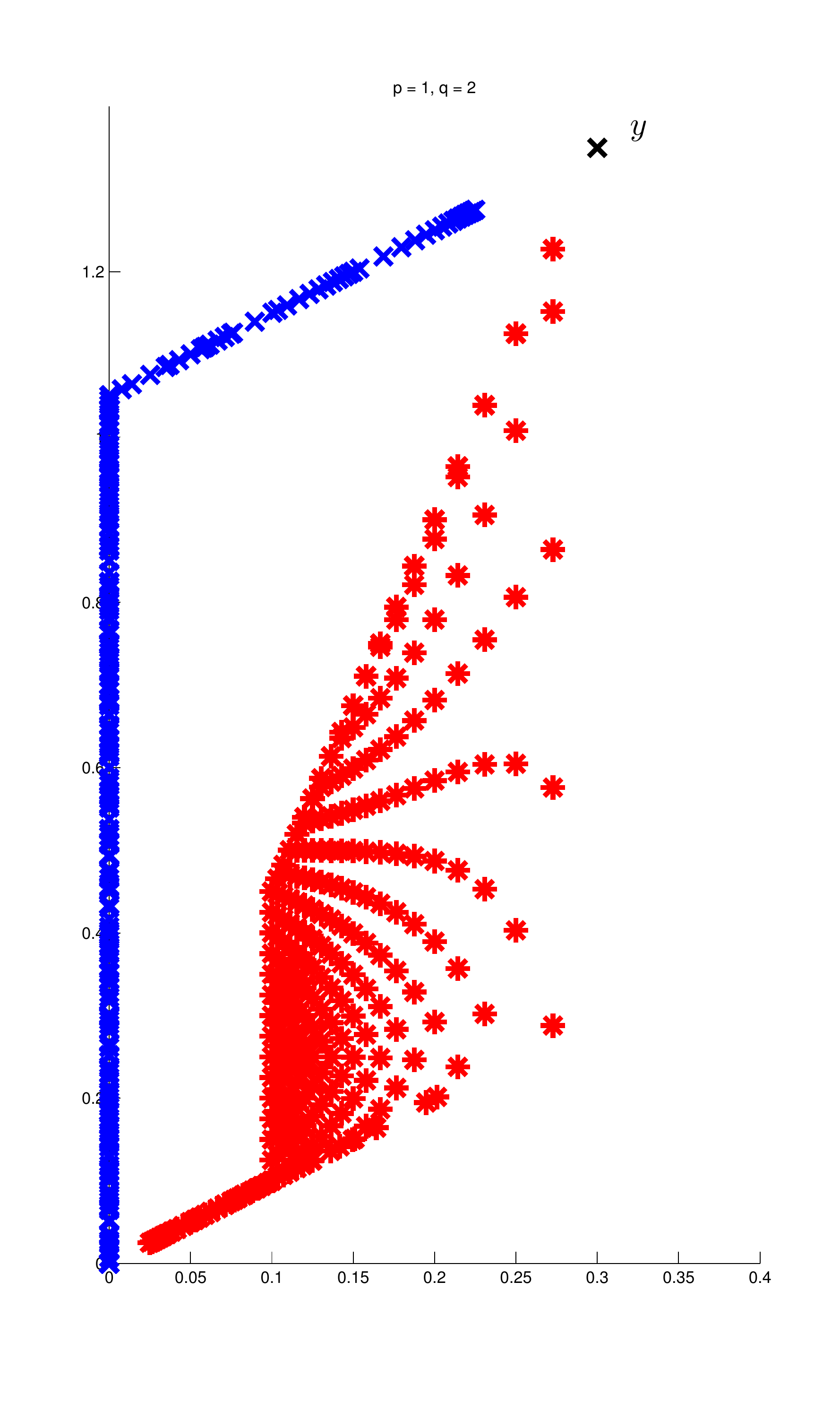}
\includegraphics[width=0.32\textwidth]{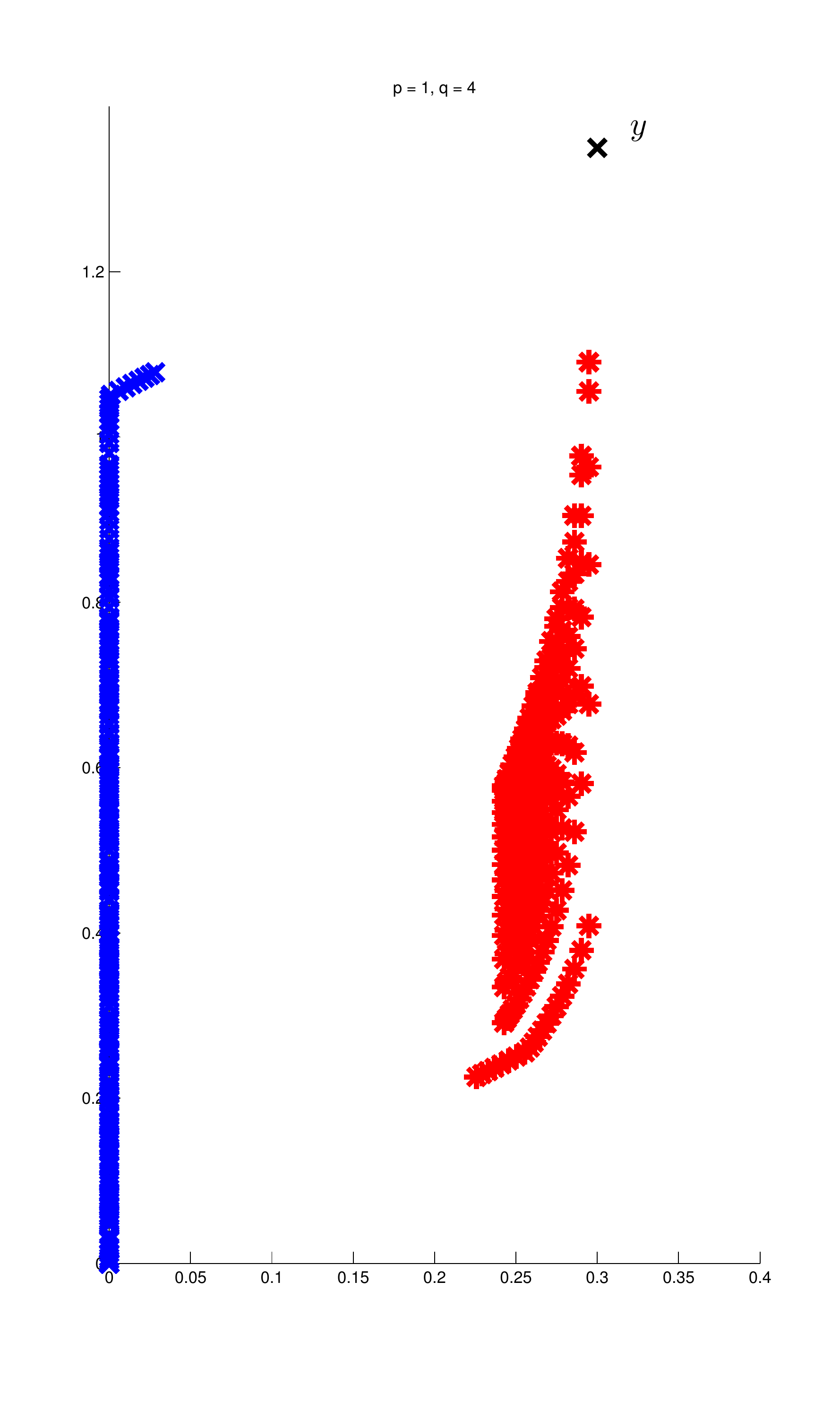}
\includegraphics[width=0.32\textwidth]{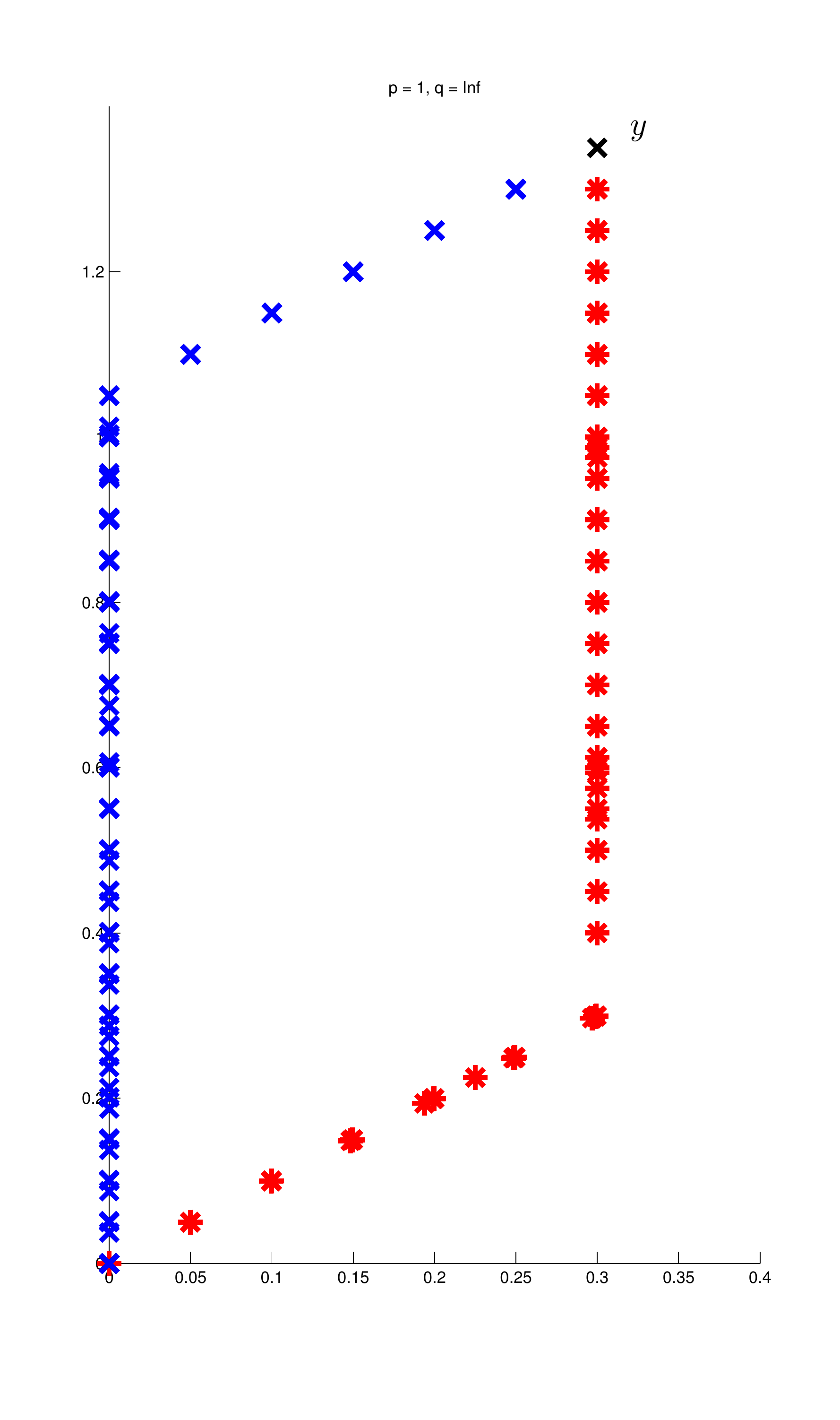}
\caption{Estimated regions of solution for $p=1$ and $q\in\left\{2,4,\infty\right\}$.}
\label{fig:plotp1fixed}
\end{figure}

We immediately observe that the algorithm is computing solutions $u^{\dagger},v^{\dagger}$ in a certain region which is bounded by a parallelogram. In particular, independently of the choice of $q$, the solutions $u^{\dagger}$ are distributed only on the upper and left-hand side of this parallelogram, while the solutions $v^{\dagger}$ may be also distributed in its interior. Depending on the choice of $q$, the respective regions seem to cover the lower right-hand  \enquote{triangular} part of the parallelogram, having a straight ($q=2$), or concave ($q>2$) boundary. In the case of $q=\infty$, all solutions can be found on the right-hand and lower side of the parallelogram, which represents the limit of the concave case. \\

To explain the above results, we have to give a detailed look at a single iteration of the algorithm. As mentioned above, the algorithm is guaranteed to converge to the global minimizer independently on the starting vector. 
Therefore, for simplicity and transparency we choose $u^{(0)}=v^{(0)} =0.$  The case of $p=1$ and $q=\infty$ reveals the most ``structured" results in terms of the region of solutions, namely piecewise linear paths. Thus, we consider this parameter pair as a reference for the following explanations. In Figure~\ref{fig:exp_algo}  we explicitly show the first three iterations of the algorithm, as well as the totality of 13 iterations, setting $\alpha = 0.4$ and $\beta=0.5$. To get a better understanding of what the algorithm  is doing, we introduce the notion of  \emph{solution path}, which we define as the set of minimizers, depending on $\alpha$ and $\beta$ respectively, 
\[U_p^{n+1} := \left\{\bar{u} | \bar{u} =  \argmin_{u} \|u + v^{(n)} - y\|_{\ell_2}^2 + \alpha \| u\|_{\ell_p}^p, \alpha > 0 \right\}, \]
\[V_q^{n+1} := \left\{\bar{v} | \bar{v} =  \argmin_{v} \|v + u^{(n+1)} - y\|_{\ell_2}^2 + \beta \| v\|_{\ell_q}^q, \beta > 0 \right\}. \]
As we shall show in Section~\ref{sec:newthresholdingoperators}, these sets can be described, explicitly in the case of $p=1$ and $q=\infty$, by simple thresholding operators $u^{(n+1)} = \mathbb S^1_{\alpha}(y-v^{(n)})$, and $v^{(n+1)}=\mathbb S^{\infty}_{\beta}(y-u^{(n+1)})$, where $\left(\mathbb S^1_{\alpha}(y-v^{(n)})\right)_i := \max\left\{1-\frac{\alpha}{2|y_i-v_i^{(n)}|},0\right\}(y_i-v_i^{(n)})$, $i=1,2$, and 
\[\mathbb S^{\infty}_{\beta}(y-u^{(n+1)}) := \begin{cases} \left(\begin{matrix}0 \\ 0\end{matrix}\right), & |y_1-u^{(n+1)}_1| + |y_2-u^{(n+1)}_2| < \beta / 2, \\ \left(\begin{matrix}\sgn(y_1-u^{(n+1)}_1)(|y_1-u^{(n+1)}_1|-\beta / 2 ) \\ y_2-u^{(n+1)}_2\end{matrix}\right), & |y_2-u^{(n+1)}_2| < |y_1-u^{(n+1)}_1| - \beta / 2 ,
\\ \left(\begin{matrix} y_1-u^{(n+1)}_1 \\ \sgn(y_2-u^{(n+1)}_2)(|y_2-u^{(n+1)}_2|-\beta / 2 ) \end{matrix}\right), & |y_1-u^{(n+1)}_1| < |y_2-u^{(n+1)}_2| - \beta / 2, \\ \frac{|y_1-u^{(n+1)}_1|+|y_2-u^{(n+1)}_2|-\beta / 2 }{2}\left(\begin{matrix}\sgn(y_1-u^{(n+1)}_1), \\ \sgn(y_2-u^{(n+1)}_2)\end{matrix}\right), & \text{else.} \end{cases}  \]

In Figure~\ref{fig:exp_algo} the solution paths $U_1^{n+1}$ and $V_{\infty}^{n+1}$ are presented as dashed and dotted lines respectively. We observe the particular shape of a piecewise linear one-dimensional path. Naturally, $u^{(n+1)}\in U_1^{n+1}$ and $v^{(n+1)}\in V_{\infty}^{n+1}$. In our particular setting, we can observe geometrically, and also verify  by means of the above given thresholding functions, that $u^{(n)}\in U_1^1$ for all $n\in \mathbb{N}$. The detailed calculation can be found in Appendix~\ref{app:A}. It implies that also the limit has to be in the same set, and, therefore, the set of limit points is included in $U_1^1$, which is represented by a piecewise linear path between $0$ and $y$. \\

\begin{figure}
\includegraphics[width=0.49\textwidth]{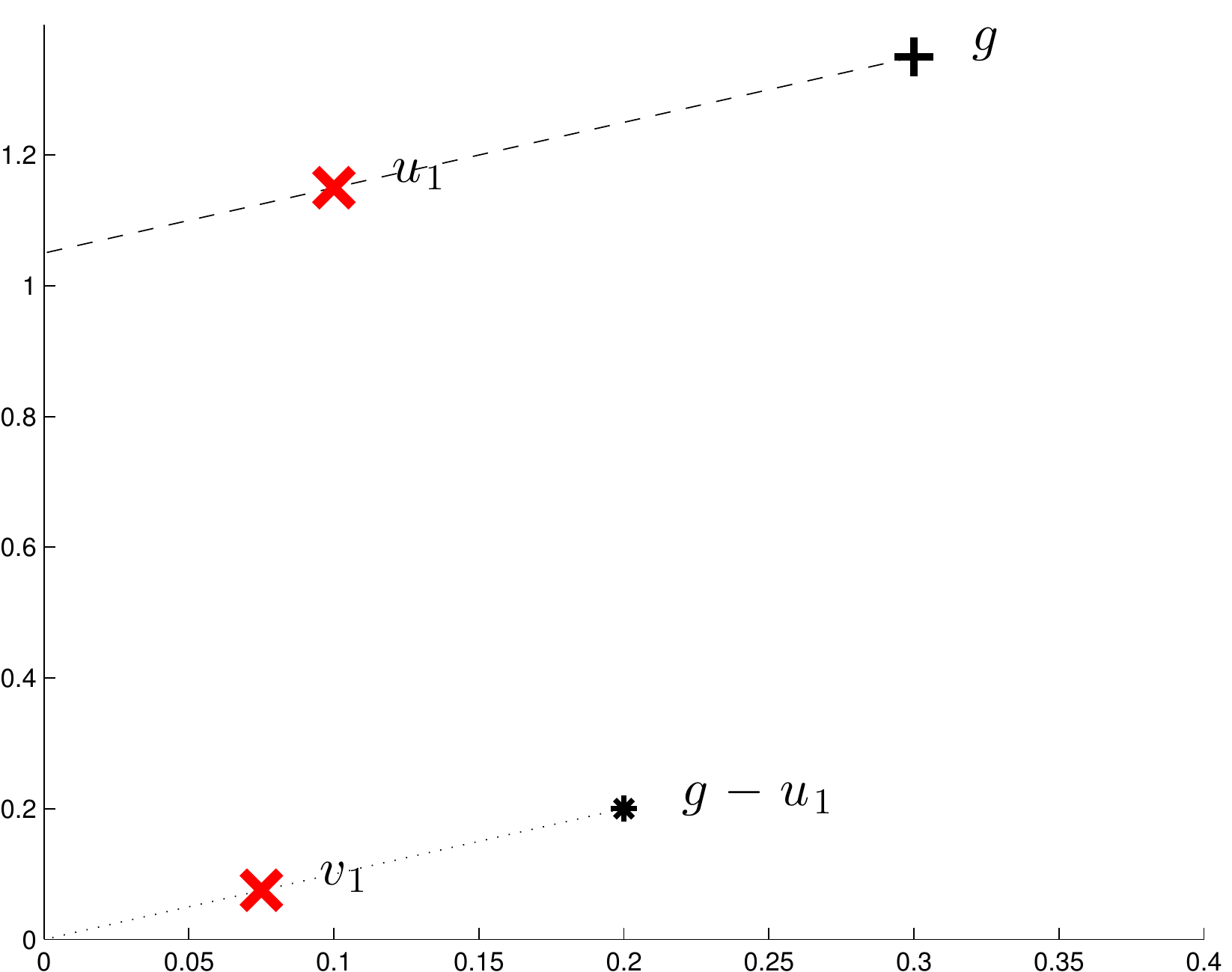}
\includegraphics[width=0.49\textwidth]{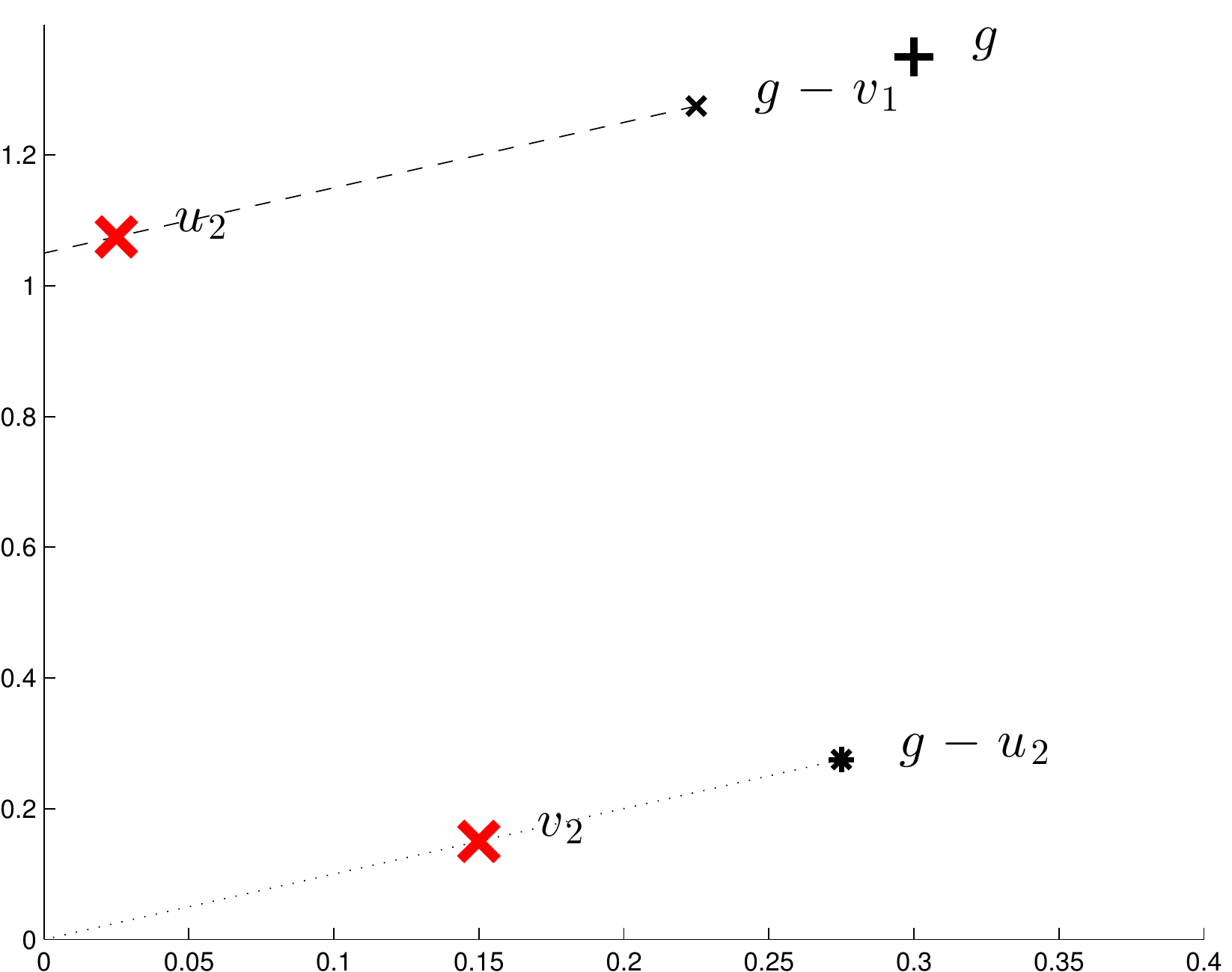}
\includegraphics[width=0.49\textwidth]{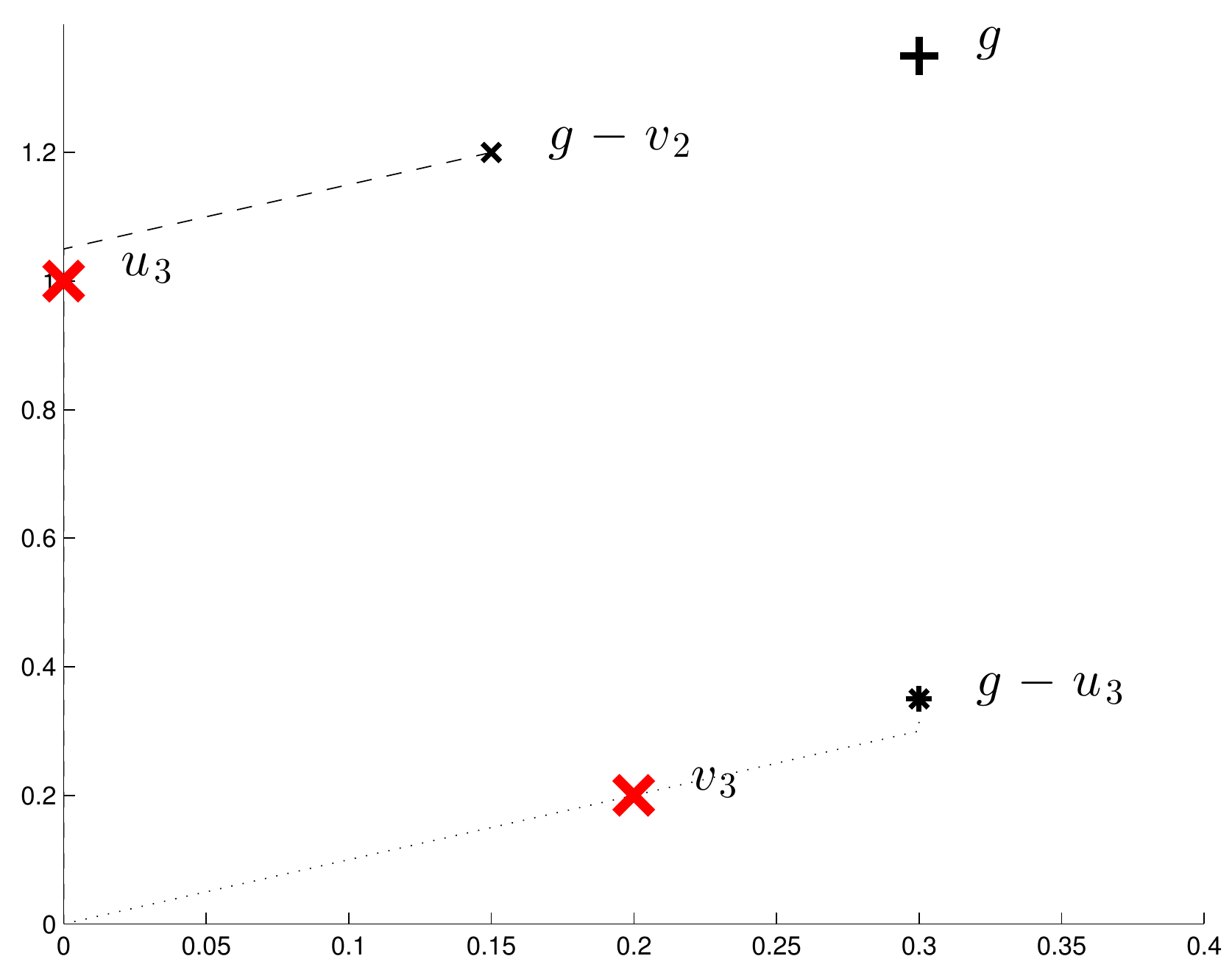}
\includegraphics[width=0.49\textwidth]{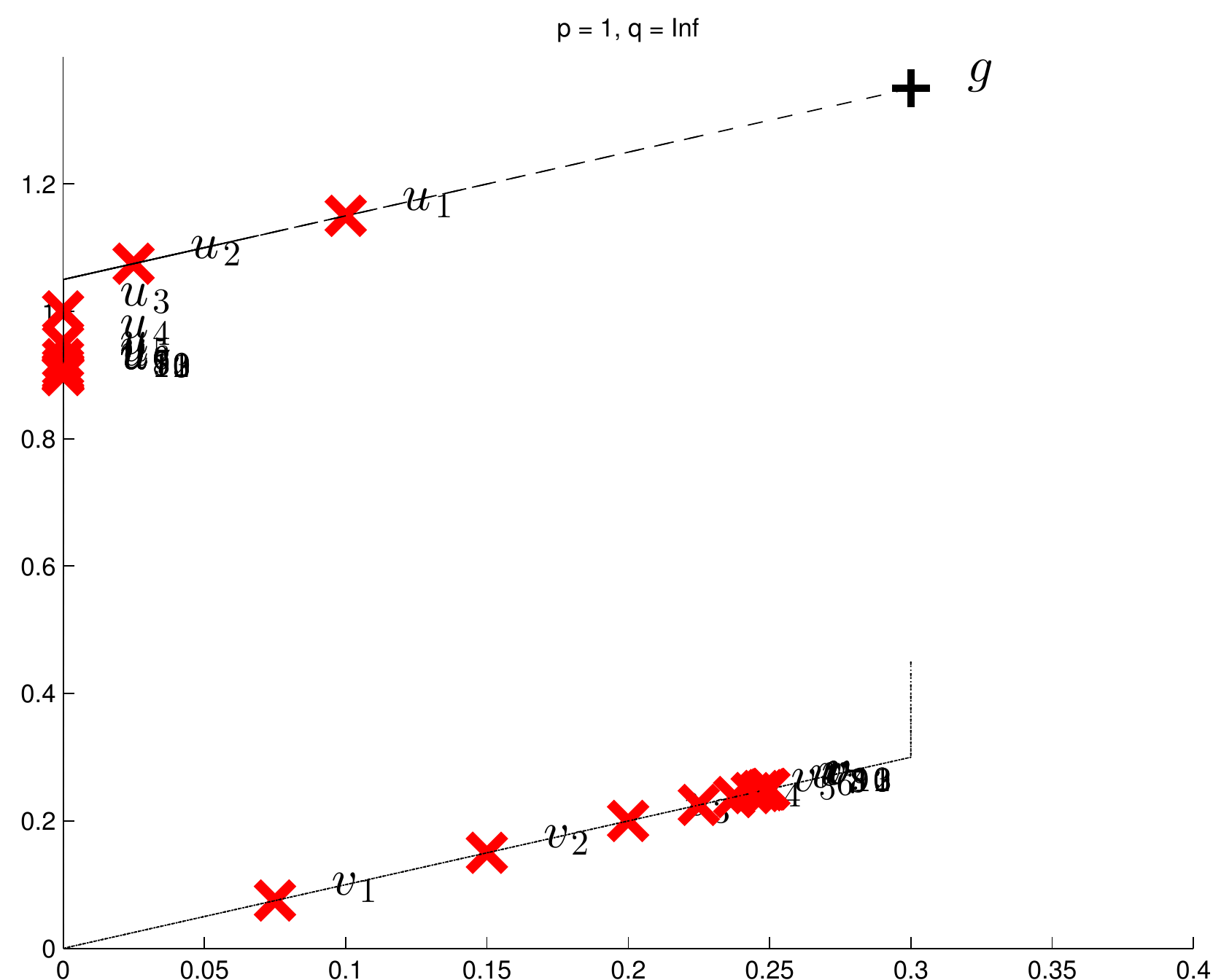}
\caption{Behavior of the algorithm for $p=1$, $q=\infty$, $\alpha=0.4$, $\beta=0.5$. The solution path for $u$ and $v$ is represented by the dashed and dotted line respectively.}
\label{fig:exp_algo}
\end{figure}

While we have a well-shaped convex problem in the case of $p=1$, the situation becomes more complicated for $p < 1$ since multiple minimizers may appear and the global minimizer has to be determined. In Figure~\ref{fig:plotp05fixed} again we visually estimate the regions of solutions $(u^{\dagger},v^{\dagger})$ with $\times$- and $\ast$-markers respectively. In the three plots for $p=0.5$ and $q\in\left\{2,4,\infty\right\}$, the regions of solutions are discretized by showing the solutions for all possible pairs of $\alpha,\beta\in\{0.1\cdot i|i=1,\ldots,40\}$. Compared to the results shown in Figure~\ref{fig:plotp1fixed}, on the one hand, the parallelogram is expanded and on the other hand, two gaps seem to be present in the solution region of $u^{\dagger}$. Such behavior is due to the non-convexity of the problem. As an extreme case, in Figure~\ref{fig:plotp0fixed} we present the same experiments  only putting $p=0$. As one can easily see, the obtained results confirm the above observations. Note that in this limit case setting, the gaps become so large, that the solution area of $u^{\dagger}$ is restricted to 3 vectors only.

\begin{figure}[ht]
\includegraphics[width=0.32\textwidth]{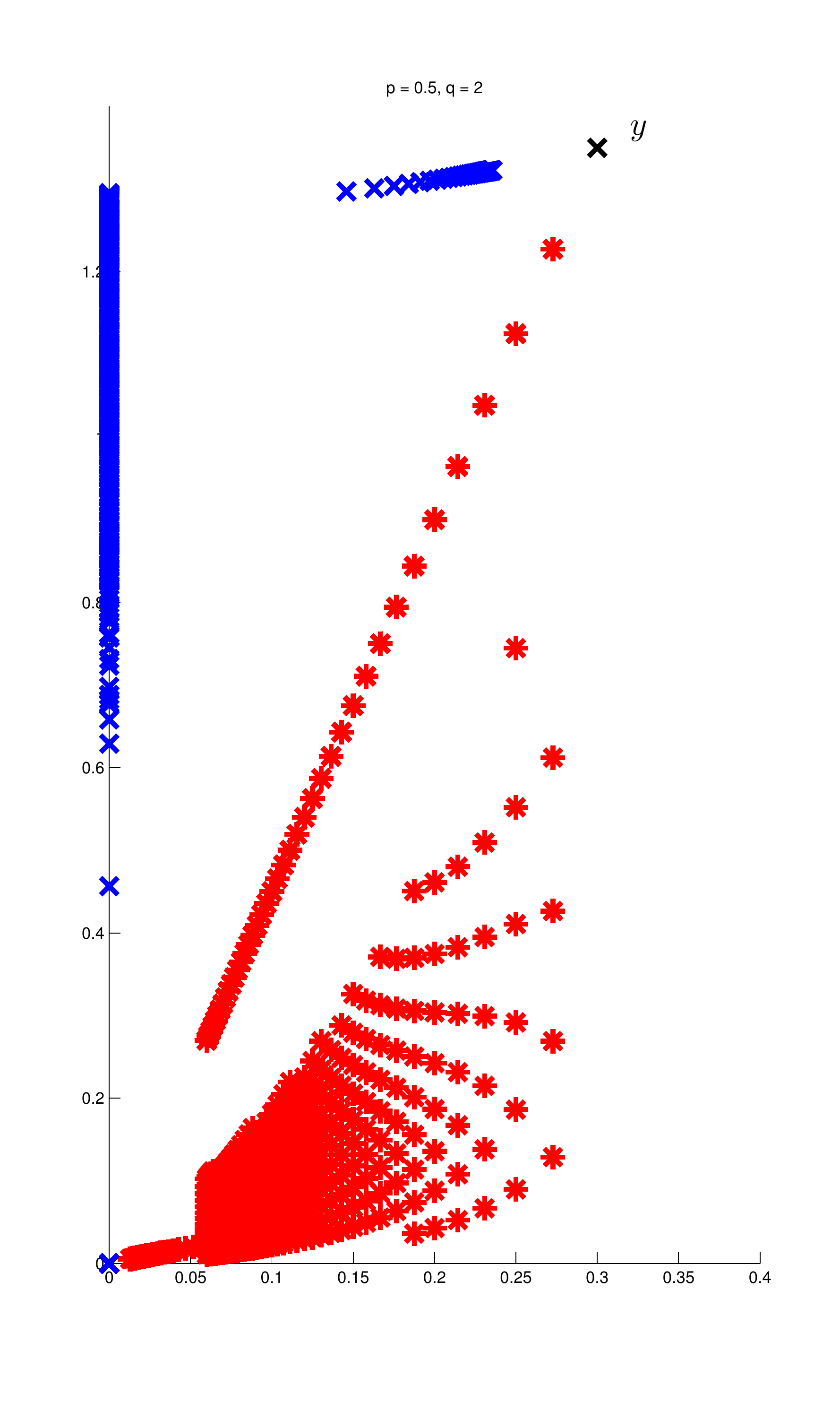}
\includegraphics[width=0.32\textwidth]{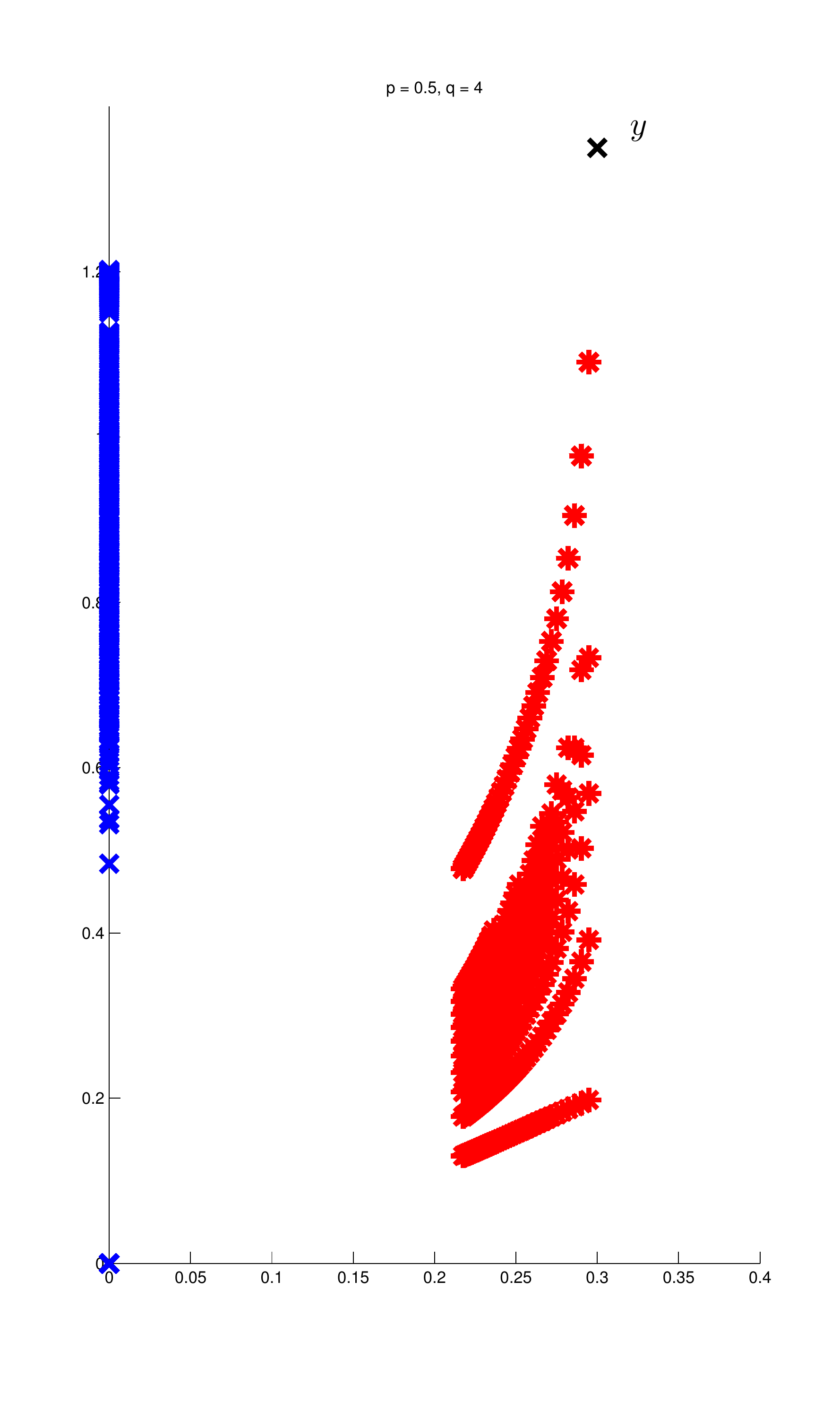}
\includegraphics[width=0.32\textwidth]{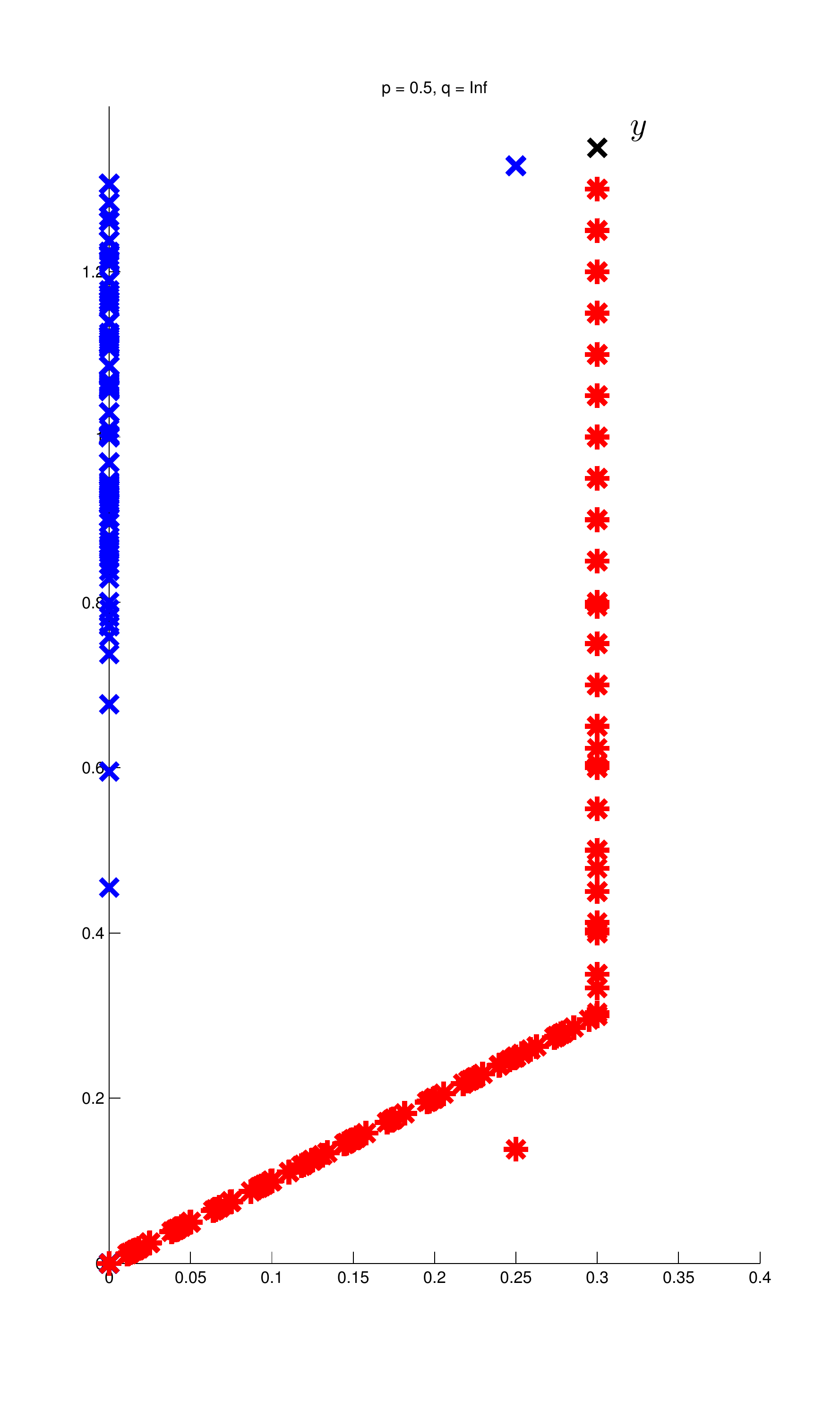}
\caption{Estimated regions of solution for $p=0.5$, and $q\in\left\{2,4,\infty\right\}$.}
\label{fig:plotp05fixed}
\end{figure}

\begin{figure}[ht]
\includegraphics[width=0.32\textwidth]{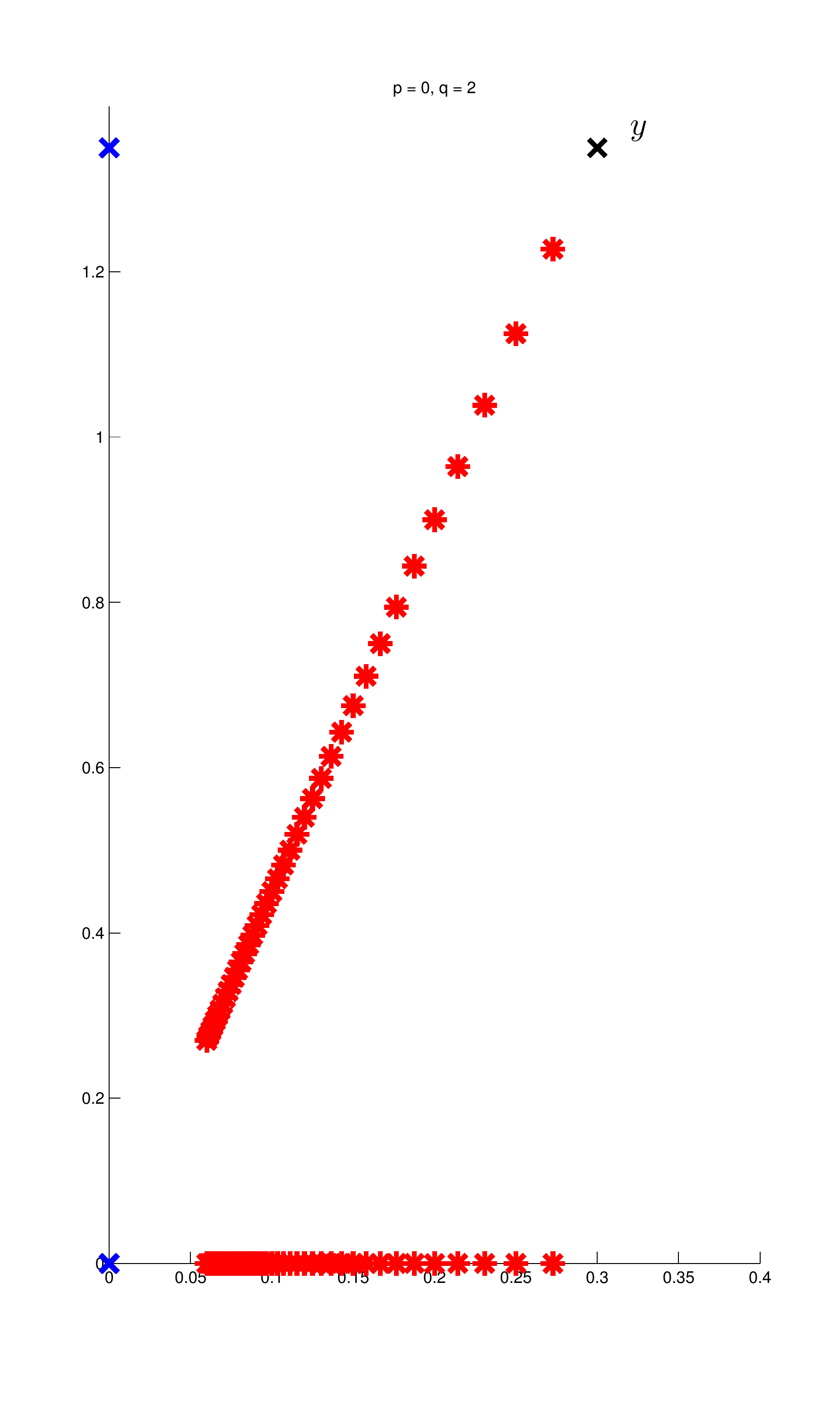}
\includegraphics[width=0.32\textwidth]{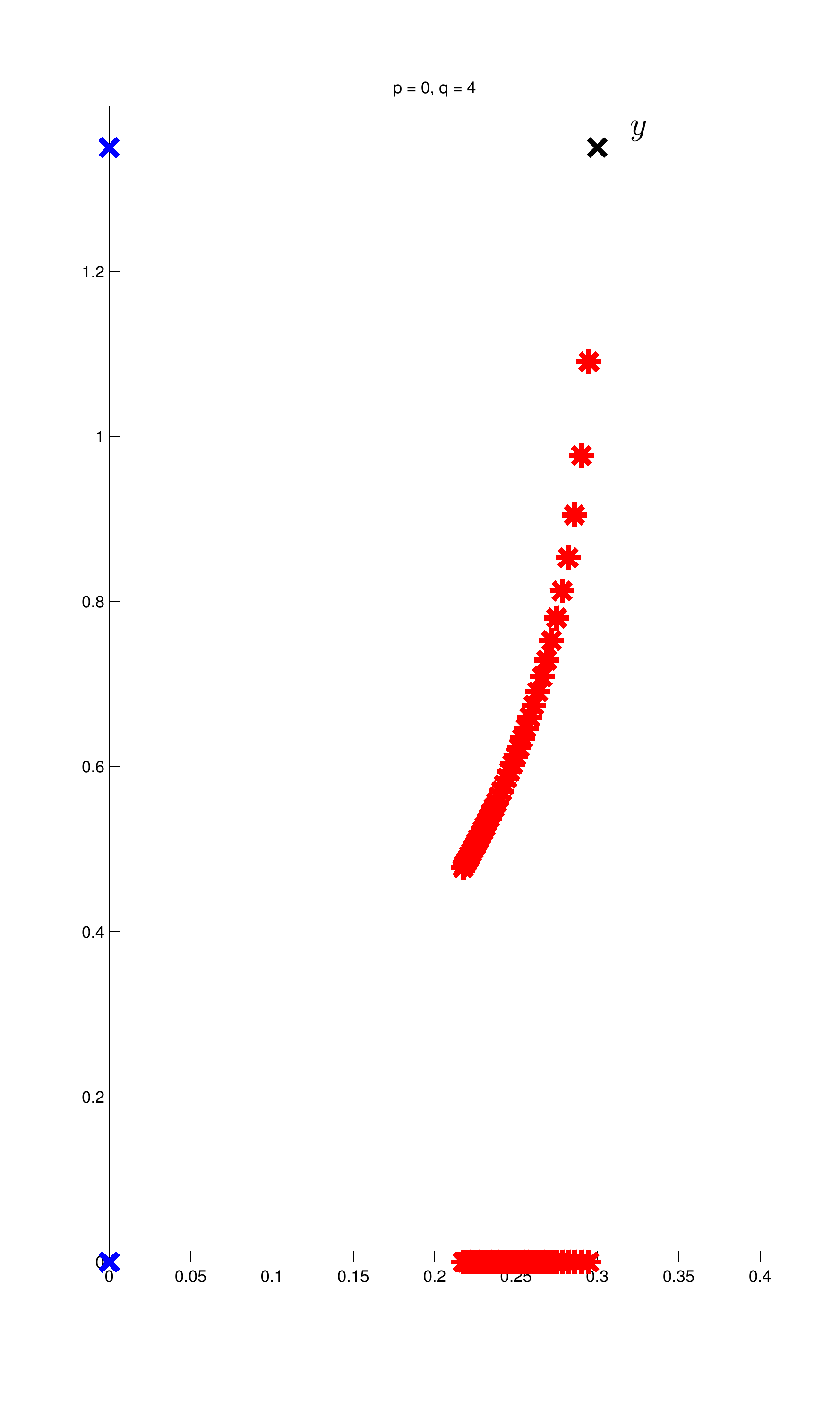}
\includegraphics[width=0.32\textwidth]{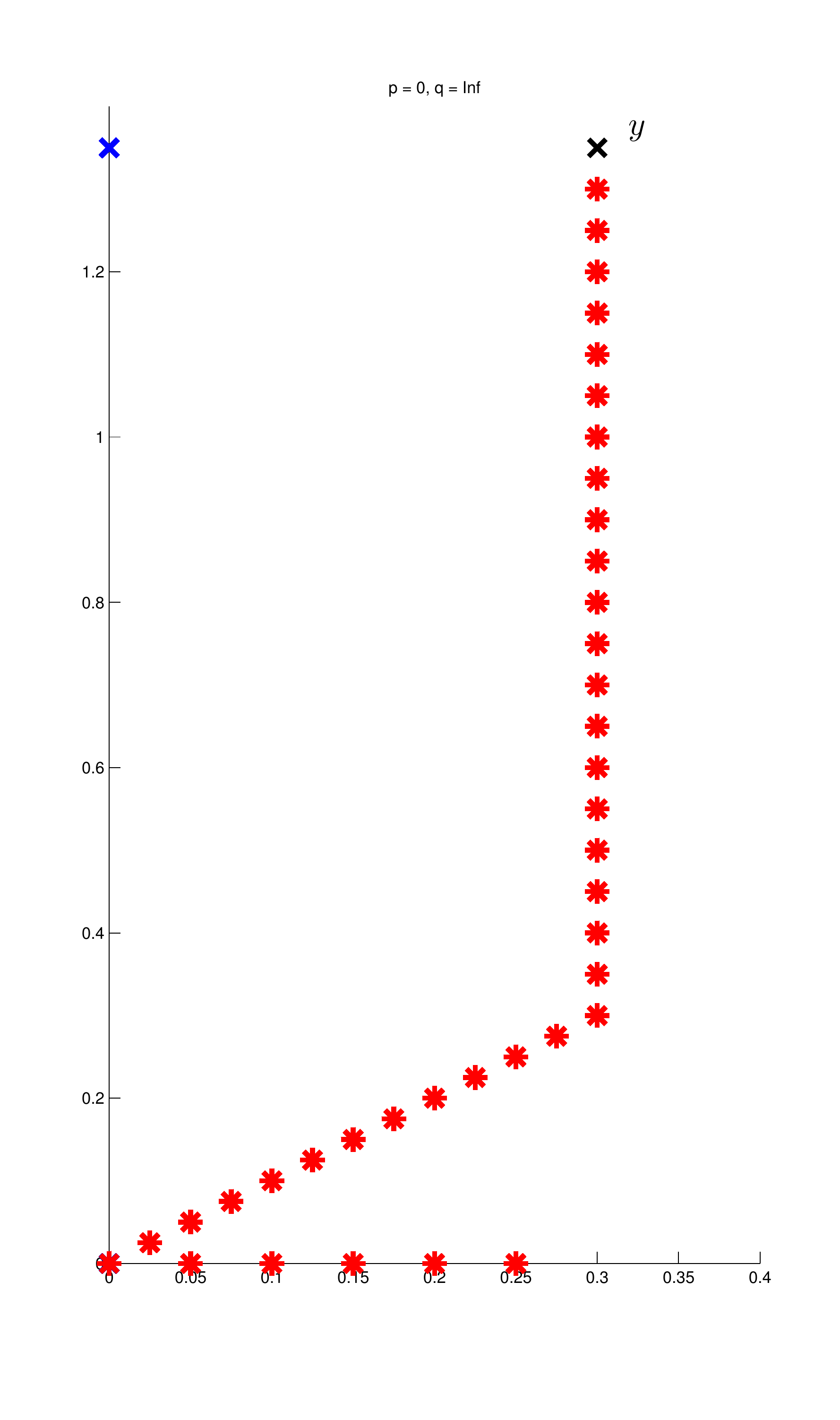}
\caption{Estimated regions of solution for $p=0$, and $q\in\left\{2,4,\infty\right\}$.}
\label{fig:plotp0fixed}
\end{figure}

Owing to these first simple results, we obtain the following three preliminary observations: 
\begin{enumerate}
\item The algorithm promotes a variety of solutions, which form a very particular structure;
\item With decreasing $p$, and increasing $q$, the clustering of the solutions is stronger;
\item The set of possible solutions is bounded by a compact set and, thus, many possible solutions can never be obtained for any choice of $q>2$, $\alpha>0$, and $\beta>0$.
\end{enumerate}

Inspired by this simple geometrical example of a 2D unmixing problem, 
in this paper we  deal with several aspects of optimizations of the type $\mathcal P(\alpha, \beta),$ recasting the unmixing problem (\ref{unmixing_pproblem}) into the classical inverse problems framework, where $T$ may have non-closed range and the observed data is additionally corrupted by noise, obtained by folding additive noise on the signal through  the measurement operator $T$, i.e.,
\[
	y = T u^\dag+ \xi,
\]
where $\xi = T v^\dagger$ and $\| v^\dagger \|_{\ell_2} \leq \eta, ~\eta \in (0,1).$ Due to non-closedness of $\mathcal R (T),$ the solution $u^\dag$  does not depend anymore  continuously on the data and can be reconstructed in a stable way from $y$ only by means of a regularization method \cite{Engl}.

On the basis of these considerations, we assume that the components $u$ and $v$ of the solution are sequences belonging to suitable spaces $\ell_p$ and $\ell_2=\ell_q \cap \ell_2$ respectively, for $0\leq p< 2$ and $2 \leq q <\infty$.
We are interested in the numerical minimization in $\ell_p \times \ell_2$ of the general form of the functionals
\begin{equation}
	\label{eq:funct_general}
	J_{p,q}(u,v) : = \| T(u+v)-y \|^2_{\mathcal H} +  \alpha \| u\|_{\ell_p}^{p} + \left (\beta  \| v\|_{\ell_q}^{q} + \varepsilon \| v\|_{\ell_2}^2 \right ),
\end{equation}
where $\alpha, \beta, \varepsilon \in \mathbb R_+$, and $p,q$ may all be considered regularization parameters of the problem. The parameter $\varepsilon>0 $ ensures the $\ell_2-$coercivity of $J_{p,q}(u,\cdot)$ also with respect to the component $v$. We shall also take advantages of this additional term in the proof of Lemma \ref{lemma:embarace}. 

In this paper we explore the following issues:
\begin{itemize}
\item We propose an iterative alternating algorithm to perform the minimization of $J_{p,q}$ by means of simple iterative thresholding steps; due to the potential non-convexity
of the functional for $0<p<1$ the analysis of this iteration requires a very careful adaptation of several techniques which are collected in different previous papers
of several authors \cite{BD08, BL09, DDD04, Fornasier07} on a single-parameter regularization with a sparsity-promoting $\ell_p-$penalty, $0 \leq p < 2.$  
\item  Thanks to this algorithm, we can explore by means of high-dimensional data analysis methods such as Principle Component Analysis (PCA), the geometry of the computed  solutions for different parameters $\alpha, \beta$ and $p,q.$ Additionally, we carefully estimate their effect in terms of quality of recovery
given $y$ as the noisy data associated via $T$ to sparse solutions affected by bounded noise. Such an empirical analysis shall allow us  to classify  the best recovery parameters for certain compressed sensing problems considered in our numerical experiments.
\end{itemize}
The formulation of multi-penalty functionals of the type \eqref{eq:funct_general} is not at all new, as we shall recall below several known results associated to it. However, the
systematic investigation of the two relevant issues indicated above, in particular, the employment of high-dimensional data analysis methods to classify parameters  and relative solutions are, to our knowledge, the new contributions given by this paper. 
\\



Perhaps as one of the earliest examples of multi-penalty optimization of the type (\ref{eq:funct_general}) in imaging, we may refer to the one seminal work of Meyer \cite{Meyer}, where
the combination of two different function spaces $H^{-1}$ and $BV$, the Sobolev space of distributions with negative exponent $-1$
and the space of bounded variation functions respectively, has been used in image reconstruction towards a proper recovery and separation
of texture and cartoon image components; we refer also to the follow up papers by Vese and Osher \cite{VO40,VO04}.   
Also in the framework of multi-penalty sparse regularization, one needs to look at the early work on signal separation  by means of $\ell_1$-$\ell_1$ norm minimizations in the seminal papers
of Donoho et al. on the incoherency of  Fourier basis and the canonical basis \cite{MR997928, MR1872845}. We mention also
the recent work \cite{MR2994548}, where the authors consider again $\ell_1$-$\ell_1$ penalization
with respect to curvelets and wavelets to achieve separation of curves and point-like features in images. Daubechies and Teschke built on the works
\cite{Meyer,VO40,VO04} providing a sparsity based formulation of the simultaneous image decomposition, deblurring, and denoising \cite{MR2147059},
by using multi-penalty $\ell_1$- and weighted-$\ell_2$-norm minimization.  The work  by Bredies and Holler \cite{BH14} analyses the regularization properties of the total generalized variation functional 
for general symmetric tensor fields and provides convergence for multiple parameters for a special form of the regularization term.
In more recent work \cite{HK14}, the infimal convolution of total variation type functionals for image sequence regularization has been considered, where an image sequence is defined as a function on a three dimensional space time domain. The motivation for such kind of functionals is to allow  suitably combined spatial and temporal regularity constraints.  
We emphasize also the two recent conceptual papers \cite{NP13, fonapeXX}, where the potential of multi-penalty regularization to outperform single-penalty regularization has been discussed and theoretically proven in the Hilbert space settings.

The results presented in this paper are very much inspired not only by the above-mentioned theoretical developments in inverse problems  and imaging communities but also by certain problems in compressed sensing \cite{FoRa13}. In particular, in the recent  paper   \cite{arfopeXX}  a two-step method to separate noise from sparse signals acquired by compressed sensing measurements has been proposed. 
However, the computational cost of the second phase of the procedure presented in  \cite{arfopeXX}, being a non-smooth and non-convex optimization, is too demanding to be performed on problems with realistic dimensionalities. 
At the same  time, in view of the results in  \cite{NP13}, the concept of the procedure described above could be profitably combined into multi-penalty regularization of the type (\ref{eq:funct_general}) with suitable choice of the spaces and parameters and lead to a simple and fast procedure. 
Of course, the range of applicability of the presented approach is not limited to  problems in compressed sensing. Image reconstruction, adaptive optics, high-dimensional learning, and several other problems are fields where we can expect that multi-penalty regularization can be fruitfully used. We expect our paper to be a useful guideline to those scientists in these fields for a proper use of multi-penalty regularization, whenever their problem requires the separation of sparse and non-sparse components.

It is worthwhile mentioning that in recent years  both regularization with non-convex constraints (see \cite{BL09, HinterWu, ItoKunisch, Zarzer, ZarzerRamlau} and references therein) and multi-penalty regularization (\cite{LP_NumMath, NP13}, just to mention a few) have become the focus of interest  in several research communities. While in most of the literature these two directions are considered separately, there have also been some efforts to understand regularization and convergence behavior for multiple parameters and functionals, especially for image analysis \cite{BH14, SL13}.
However, to the best of our knowledge, the present paper is the first one providing an explicit direct mechanism for  minimization of the multi-penalty functional with non-convex and non-smooth terms, and highlighting its improved accuracy power with respect to more traditional one-parameter regularizations.

The paper is organized as follows. In Section 2 we recall the main concepts related to surrogate functionals. Later we borrow these concepts for the sake of minimization of the functional $J_{p,q}$.
In particular, we split the minimization problem  into two subproblems and we propose an associated sequential algorithm based on iterative thresholding. 
The main contributions of the paper are presented in Sections 3 and 4. In Section 3 by greatly generalizing and adapting the arguments in \cite{BD08, BL09, DDD04, Fornasier07} we show that the successive minimizers of the surrogate functionals defined in Section 2 indeed converge to stationary points, which under reasonable additional conditions are local minimizers of the functional (\ref{eq:funct_general}). We first establish weak convergence, but conclude the section by proving that the convergence also holds in norm.
In Section 4 we open the discussion on the advantages and disadvantages of different parameter  $\alpha, \beta$ and $p,q$ choices. Inspired by the 2D illustrations in the introduction, we perform similar experiments but this time for high-dimensional problems and provide an appropriate representation of high-dimensional data by employing statistical techniques for dimensionality reduction, in particular, PCA.  
Moreover, a  performance comparison of the multi-penalty regularization with its one-parameter counterpart is presented and discussed. As a byproduct of the results, in this section 
we provide a guideline for the \enquote{best} parameter choice. 

\section{An iterative algorithm: new thresholding operators}
\subsection{Notation}
We begin this section with a short reminder of the standard notations used in this paper.
For some countable index set $\Lambda$ we denote by $\ell_p = \ell_p(\Lambda),~ï¿½0 \leq p \leq \infty,$ the space of real sequences $u = (u_\lambda)_{\lambda \in \Lambda}$ with (quasi-)norm
\[
	\| u \|_p : = \| u \|_{\ell_p(\Lambda)} : = \left(  \sum_{\lambda \in \Lambda} |u_\lambda |^p \right)^{1/p},~ï¿½0 < p < \infty,
\]
and $\| u \|_{\infty} : = \sup_{\lambda \in \Lambda} |u_\lambda|$ and $\| u \|_{0} : =\# \{ \lambda \in \Lambda~ï¿½|ï¿½~ï¿½u_\lambda \neq 0\}$ as usual. 
 
Until different notice, here we assume $0\leq p <2$ and $2 \leq q \leq\infty.$
Since for $q= \infty$ the penalty term $\| \cdot \|_{\ell_q}^q$ in (\ref{eq:funct_general}) \enquote{equals} infinity, as already done before  we  adopt the convention that $\| \cdot \|_{\ell_q}^q = \| \cdot \|_{\ell_q}$ for $q=\infty.$

More specific notations will be defined in the paper, where they turn out to be useful. 

\subsection{Preliminary lemma}

We want to minimize $J_{p,q}$ by the suitable instances of the following  alternating algorithm:  pick up initial $u^{(0)}, v^{(0)}$ and iterate
\begin{equation}
	\label{carrot1}
 			\left\{
				\begin{array}{l}
					u^{(n+1)} \approx \argmin_u J_{p,q}(u,v^{(n)}) \\
					v^{(n+1)} \approx \argmin_v J_{p,q}(u^{(n+1)},v),
				\end{array}
			\right.\\
\end{equation}
where \enquote{$\approx$} stands for the approximation symbol, because in practice we never perform the exact minimization.
Instead of optimising $J_{p,q}$ directly, let us introduce auxiliary functionals $J_u^s,~J_v^s,$ called the {\it surrogate functionals} of $J_{p,q}$: for some additional parameter $a$ let
\begin{eqnarray}
	J_u^s (u, v; a) & : = & J_{p,q}(u,v) + \| u- a \|^2_2 - \|Tu-Ta\|^2_{\mathcal H}, 	\label{surrogate_u} \\
	J_v^s (u, v; a) & : = & J_{p,q}(u,v) + \| v- a \|^2_2 - \|Tv-Ta\|^2_{\mathcal H}. \label{surrogate_v} 
\end{eqnarray}

In the following we assume that $\| T \| <1.$ This condition can always be achieved by suitable rescaling of $T$ and $y.$ Observe that 
\begin{eqnarray}
	 \| u- a \|^2_2 - \|Tu-Ta\|^2_{\mathcal H} & \geq & C \|u-a\|^2_2, \label{surrogate_ineq_u} \\
	\| v- a \|^2_2 - \|Tv-Ta\|^2_{\mathcal H} & \geq & C \|v-a\|^2_2, \label{surrogate__ineq_v} 
\end{eqnarray}
for $C = (1 - \|T\|)^2.$ Hence, 
\begin{eqnarray}
	J_{p,q}(u,v) & = & J_u^s (u, v; u) \leq J_u^s (u, v; a), 	\label{surrogate_func_u} \\
	J_{p,q}(u,v)  & = & J_v^s (u, v; v) \leq J_v^s (u, v; a) . \label{surrogate_func_v} 
\end{eqnarray}
everywhere, with equality if and only if $u=a$ or $v=a.$ Moreover, the functionals decouple the variables $u_\lambda$ and $v_\lambda$ so that the above minimization procedure reduces to component-wise minimization (see Sections 2.3.1, 2.3.2 below).

Alternating minimization of  (\ref{carrot1}) can be performed by minimizing the corresponding surrogate functionals (\ref{surrogate_u})-(\ref{surrogate_v}). This  leads to the following  sequential algorithm: pick up initial $u^{(0)}$, $v^{(0)}$, and iterate
\begin{equation}
	\label{algorithm_surrogate_func}
\left\{
	\begin{array}{l}
 			\left\{
				\begin{array}{l}
					u^{(n)} = u^{(n,L)} = u^{(n+1,0)} \\
					u^{(n+1,l+1)} = \argmin_{u \in \ell_2(\Lambda)} J_u^s (u, v^{(n)}; u^{(n+1,l)}), \quad l=0, \ldots, L-1 \\
				\end{array}
			\right.\\
			\left\{
				\begin{array}{l}
					v^{(n)} = v^{(n,M)} = v^{(n+1,0)} \\
					v^{(n+1,l+1)} = \argmin_{v \in \ell_2(\Lambda)} J_v^s (u^{(n+1,L)}, v; v^{(n+1,l)}), \quad l=0, \ldots, M-1 \\
				\end{array}
			\right.
	\end{array}.
\right.
\end{equation}


The following lemma provides a tool to prove the weak convergence of the algorithm. It is 
standard when using surrogate functionals (see \cite{BD08, DDD04}), and concerns general real-valued surrogate functionals.
It holds independently of the specific form of the functional $J_{p,q},$ but does rely on the restriction that $\| T \| <1.$

\begin{lemma}
	\label{lemma1_general_surrigate}
	If $J_u^s (u, v; a)$ and $J_v^s (u, v; a)$ are given as in (\ref{surrogate_u}), (\ref{surrogate_v}), and the sequences $(u^{(n)})$, $(v^{(n)})$ are generated by the algorithm in
	(\ref{algorithm_surrogate_func}),  then the sequences $J_{p,q}(u^{(n)}, v^{(n)}),~J_u^s (u^{(n+1)}, v^{(n)}; u^{(n)})$ and $J_v^s (u^{(n+1)}, v^{(n+1)}; v^{(n)})$ are non-increasing as long as $\| T \|<1.$ Moreover, 
	\begin{equation*}
	\| u^{(n+1)} - u^{(n)} \|_2 \rightarrow 0, \quad\quad\quad \| v^{(n+1)} - v^{(n)} \|_2 \rightarrow 0,
\end{equation*}
for $n \rightarrow \infty$. 
\end{lemma}

\begin{proof}
%

	Using  (\ref{surrogate_func_u}) we have
	\[
		J_{p,q}(u^{(n)},v^{(n)}) = J_u^s(u^{(n)},v^{(n)};u^{(n)}) = J_u^s(u^{(n,L)},v^{(n,M)};u^{(n+1,0)}).
	\]
	
	Since at this point the proof is similar for both $u$ and $v,$ for the sake of brevity, we consider the case of $J_u^s$ in detail only. By definition of $u^{(n+1,1)}$ and its minimal properties in (\ref{algorithm_surrogate_func}) we have
	\[
		J_u^s(u^{(n,L)},v^{(n,M)};u^{(n+1,0)}) \geq J_u^s(u^{(n+1,1)},v^{(n,M)};u^{(n+1,0)}).
	\]
	An application of (\ref{surrogate_func_u}) gives
	\[
		J_u^s(u^{(n+1,1)},v^{(n,M)};u^{(n+1,0)}) \geq J_u^s(u^{(n+1,1)},v^{(n,M)};u^{(n+1,1)}).
	\]	
	Putting in line these inequalities we get
	\[
	J_{p,q}(u^{(n)},v^{(n)}) \geq J_u^s(u^{(n+1,1)},v^{(n,M)};u^{(n+1,1)}).
	\]
	In particular, from (\ref{surrogate_ineq_u}) we obtain
	\[
	J_{p,q}(u^{(n)},v^{(n)}) - J_u^s(u^{(n+1,1)},v^{(n,M)};u^{(n+1,1)}) \geq C \| u^{(n+1,1)} - u^{(n+1,0)} \|^2_2.
	\]
	By successive iterations of this argument we get
	\begin{eqnarray}
		J_{p,q}(u^{(n)},v^{(n)}) & \geq &J_u^s(u^{(n+1,1)},v^{(n,M)};u^{(n+1,1)}) \nonumber \\
					  & \geq & J_u^s(u^{(n+1,L)},v^{(n,M)};u^{(n+1,L)}) \nonumber \\
					  & = & J_{p,q}(u^{(n+1,L)},v^{(n,M)}), \label{eq:decr_u}
	\end{eqnarray}
	and
	\begin{equation}
		\label{eq:diff_sum_c_u}
	J_{p,q}(u^{(n,L)},v^{(n,M)}) - J_{p,q}(u^{(n+1,L)},v^{(n,M)}) \geq C \sum_{l=0}^{L-1} \| u^{(n+1,l+1)} - u^{(n+1,l)} \|^2_2.
	\end{equation}
	By definition of $v^{(n+1,1)}$ and its minimal properties 
	\[
	J_v^s(u^{(n+1,L)},v^{(n,M)};v^{(n+1,0)}) \geq J_v^s(v^{(n+1,L)},v^{(n,M)};v^{(n+1,1)}).
	\]
	By similar arguments as above we find 
	\begin{equation}
		\label{eq:decr_v}
	J_{p,q}(u^{(n+1,L)},v^{(n,M)}) \geq J_v^s(u^{(n+1,L)},v^{(n+1,M)};v^{(n+1,M)}) = J_{p,q}(u^{(n+1,L)},v^{(n+1,M)}),
	\end{equation}
	and
	\begin{equation}
		\label{eq:diff_sum_c_v}
	J_{p,q}(u^{(n+1,L)},v^{(n,M)}) - J_{p,q}(u^{(n+1,L)},v^{(n+1,M)}) \geq C \sum_{l=0}^{M-1} \| v^{(n+1,l+1)} - v^{(n+1,l)} \|^2_2.
	\end{equation}

From the above discussion it follows that $J_{p,q}(u^{(n)}, v^{(n)}) \geq 0$ is a non-increasing sequence, therefore it converges.
From (\ref{eq:diff_sum_c_u}) and (\ref{eq:diff_sum_c_v}) and the latter convergence we deduce
\begin{eqnarray*}
&\sum_{l=0}^{L-1} \| u^{(n+1,l+1)} - u^{(n+1,l)} \|^2_2 & \rightarrow 0, \\
&\sum_{l=0}^{M-1} \| v^{(n+1,l+1)} - v^{(n+1,l)} \|^2_2 & \rightarrow 0.
\end{eqnarray*}

In particular, by triangle inequality and the standard inequality $(a+b)^2 \leq 2(a^2+b^2)$ for $a,b>0,$ we also  have
\begin{eqnarray*}
	\| u^{(n+1,L)} - u^{(n+1, 0)} \|^2_2 & = & \left \| \sum_{l=0}^{L-1} (u^{(n+1,l+1)} - u^{(n+1,l)}) \right \|^2_2 \leq \left(\sum_{l=0}^{L-1} \| u^{(n+1,l+1)} - u^{(n+1,l)} \|_2\right)^2 \\
	& \leq & C_L \sum_{l=0}^{L-1} \| u^{(n+1,l+1)} - u^{(n+1,l)} \|^2_2 \rightarrow 0.
\end{eqnarray*}
Here $C_L$ is some constant depending on $L$. Analogously we can show that
\begin{eqnarray*}
	\| v^{(n+1,M)} - v^{(n+1, 0)} \|^2_2 \leq C_M \sum_{l=0}^{M-1} \| v^{(n+1,m+1)} - v^{(n+1,m)} \|^2_2 \rightarrow 0.
\end{eqnarray*}
Therefore, we finally obtain
\begin{eqnarray*}
\| u^{(n+1,L)} - u^{(n+1, 0)} \|^2_2 = \| u^{(n+1)} - u^{(n)} \|^2_2 & \rightarrow 0, \\
\| v^{(n+1,M)} - v^{(n+1,0)} \|^2_2 = \| v^{(n+1)} - v^{(n)} \|^2_2 & \rightarrow 0.
\end{eqnarray*}

\end{proof}

\subsection{An iterative algorithm: new thresholding operators}
\label{sec:newthresholdingoperators}
The main virtue/advantage of the alternating iterative thresholding algorithm (\ref{algorithm_surrogate_func}) is the given  explicit formulas for computation of the successive $v^{(n)}$ and $u^{(n)}.$
In the following subsections, we discuss how the minimizers of $J_u^s (u, v; a)$ and $J_v^s (u, v; a)$ can be efficiently computed.

We first observe a useful property of the surrogate functionals.
Expanding the squared terms on the right-hand side of the expression (\ref{surrogate_u}), we get
\begin{eqnarray*}
		J_u^s (u, v; a) & = & \| u - T^* (y-Ta-Tv) -a \|^2_2 +  \alpha \|ï¿½u \|_{p}^p 
		+ \Phi_1\\
					& = & \sum_{\lambda \in \Lambda} [(u_\lambda - [(a - T^*Ta - T^*Tv + T^*y)]_\lambda)^2 + \alpha | u_\lambda|^p +  \Phi_1,
\end{eqnarray*}
and similarly for the expression (\ref{surrogate_v}) and $2\leq q <\infty$
\begin{eqnarray*}
		J_v^s (u, v; a) & = & \| v - T^* (y-Ta-Tv) -a \|^2_2  + \beta \| v \|_q^q + \varepsilon \| vï¿½\|_2^2  + \Phi_2 \\
					& = & \sum_{\lambda \in \Lambda} [(v_\lambda - [(a - T^*Ta - T^*Tu + T^*y)]_\lambda)^2  + \beta | v_\lambda|^q + \varepsilon |v_\lambda|^2] +  \Phi_2,
\end{eqnarray*}
where the terms $\Phi_1 = \Phi_1(a, y, v)$ and $\Phi_2 = \Phi_2(a, y, u)$ depend only on $a,y,v$, and $a,y,u$ respectively. 
Due to the cancellation of the terms involving $\| Tu \|^2_2$ and $\| Tv \|^2_2,$ the variables $u_\lambda,~v_\lambda$ in $J_u^s$ and  $J_v^s$ respectively are decoupled.  Therefore, the minimizers of $J_u^s(u,v;a),~J_v^s(u,v;a)$ for $a$ and $v$ or $u$ fixed respectively, can be computed {\it component-wise} according to
\begin{eqnarray}
	 u^*_\lambda & = & \argmin_{t\in \RR} [(t - [(a - T^*Ta - T^*Tv + T^*y)]_\lambda)^2 + \alpha | t|^p], ~\lambda \in \Lambda, \label{eq:min_surr_cw_u}\\
	 v^*_\lambda &  = &  \argmin_{t\in \RR} [(t - [(a - T^*Ta - T^*T u^* + T^*y)]_\lambda)^2 + \beta | t |^q + \varepsilon |t|^2]. \label{eq:min_surr_cw_v}
\end{eqnarray}

In the case $p=0,~p =0.5,~p=1$ and $q=2$ one can solve (\ref{eq:min_surr_cw_u}), (\ref{eq:min_surr_cw_v}) explicitly; the treatment of the case $q=\infty$ is explained in Remark~\ref{rem:infty};
for the general case $0 < p < 2,~2 < q < \infty$ we derive an implementable and efficient method to compute $u^{(n)}, v^{(n)}$ from previous iterations.

\subsubsection{Minimizers of $J_v^s (u, v; a)$ for $a,~u$ fixed}
	\label{subsec:min_n}

We first discuss the minimization of the functional $J_v^s (u, v; a)$ for a generic $a, u.$ For $2 \leq q< \infty$ the summand in $J_v^s (u, v; a)$ is differentiable in $v_\lambda$, and the minimization reduces to solving the variational equation
\[
	2 (1 + \varepsilon) v_\lambda + \beta q \sgn(v_\lambda) |v_\lambda|^{q-1}  =2 [a + T^*(y - Tu - Ta)]_\lambda.
\]
Setting $\tilde v_\lambda : = (1+\varepsilon) v_\lambda$ and recalling that $| \cdot |$ is $1-$homogenous, we may rewrite the above equality as
\[
	\tilde v_\lambda + \frac{\beta q}2 \frac{\sgn(\tilde v_\lambda) |\tilde v_\lambda|^{q-1}}{(1+\varepsilon)^{q-1}}  =[a + T^*(y - Tu - Ta)]_\lambda.
\]
Since for any choice of $\beta\geq 0$ and any $q>1,$ the real function
\[
	F_{\beta, \varepsilon}^q (x) = x + \frac{\beta q}{2(1+\varepsilon)^{q-1}} \sgn(x) |x|^{q-1}
\] 
is a one-to-one map from $\RR$ to itself, we thus find that the minimizer of $J_v^s (u, v; a)$ satisfies
\begin{equation}
	\label{eq:minim_v}
	 v^*_\lambda = v_\lambda = (1+\varepsilon)^{-1} S_{\beta, \varepsilon}^q (a_\lambda + [T^*(y - Tu - Ta)]_\lambda),
\end{equation}
where $S_{\beta, \varepsilon}^q$ is defined by
\[
	S_{\beta, \varepsilon}^q= (F_{\beta, \varepsilon}^q)^{-1} \quad \mbox{ for } q\geq2.
\]
\begin{remark}
\label{rem:infty}
In the particular case $q= 2$ the explicit form of  the thresholding function $S_{\beta, \varepsilon}^2$ can be easily derived as a proper scaling and we refer the interested reader to \cite{DDD04}. For $q=\infty$ the definition of the thresholding function as 
\[
	\mathbb S_{\beta, \varepsilon}^\infty (x) = \argmin_v \| v- x \|^2_{2} + \beta \|v\|_{\infty} + \varepsilon \|v \|^2_2,
\]
for vectors $v, x \in \RR^M$  was determined explicitly using the polar projection method \cite{fora08}. 
Since in our numerical experiments we consider finite-dimensional sequences and the case $q= \infty$, we recall  here  $\mathbb S_{\beta, \varepsilon}^\infty$ explicitly for the case $\varepsilon=0$ (in finite-dimensions the additional $\ell_2$-term $\varepsilon \|v \|^2_2$ is not necessary to have $\ell_2-$coercivity).

Let $x\in\RR^M$ and $\beta>0.$ Order the entries of $x$ by magnitude such that $|x_{i_1}|\geq |x_{i_2}|\geq \ldots \geq |x_{i_M}|.$\begin{enumerate}
\item If $\|x \|_1 < \beta/2,$ then $\mathbb S_{\beta, \varepsilon}^\infty(x) = 0.$
\item Suppose $\|x \|_1 > \beta/2.$ If $ |x_{i_2}|<|x_{i_1}| -\beta/2,$ then choose $n=1.$ Otherwise, let $n \in \{2, \ldots, M \} $ be the largest index satisfying
	\[
		|x_{i_n}| \geq \frac{1}{n-1} \left(\sum_{k=1}^{n-1} |x_{i_k}| - \frac \beta2 \right).
	\]
	\end{enumerate}
Then
\begin{eqnarray*}
	(\mathbb S_{\beta, \varepsilon}^\infty(x ))_{i_j}& =& \frac{\sgn(x_{i_j})}{n} \left(  \sum_{k=1}^{n} |x_{i_k}| - \frac \beta2 \right), \quad j=1,\ldots, n \\
	(\mathbb S_{\beta, \varepsilon}^\infty(x))_{i_j} & =&x_{i_j}, \quad j = n+1,\ldots, M.
\end{eqnarray*}

These results  cannot be in practice extended to the  infinite-dimensional case because one would need to perform the reordering of the infinite-dimensional vector in absolute values. However, in the case of infinite-dimensional sequences, i.e.,  $x \in \ell_2(\Lambda),$ which is our main interest in the theoretical part of the current manuscript, one can still use the results \cite{fora08} by employing at the first step an adaptive coarsening  approach described in \cite{CoDaDe03}. This approach allows us  to obtain an approximation of an  infinite-dimensional sequence  by its $N-$dimensional counterpart with optimal accuracy order. 

\end{remark}
\subsubsection{Minimizers of $J_u^s (u, v; a)$ for $a,~v$ fixed}

In this subsection, we want to derive an efficient method to compute $u^{(n)}$. In the special case $1 \leq p <2$  the iteration $u^{(n)}$ is given  by soft-thresholdings \cite{DDD04};  for $p=0$ the iteration $u^{(n)}$ is defined by hard-thresholding  \cite{BD08}.
For the sake of brevity, we limit our analysis below to the range $0<p<1,$ which requires a more careful adaptation of the techniques already included in \cite{BD08, DDD04}. The cases $p=0$ and $1 \leq p <2$ are actually minor modifications of our analysis and the one of    \cite{BD08, DDD04}

In order to derive the minimizers of the non-smooth and non-convex functional $J_u^s (u, v; a)$ for generic $a, v,$ we follow the similar approach as proposed in \cite{BL09}, where  a general framework for minimization of non-smooth and non-convex functionals based on a generalized gradient projection method has been considered.

\begin{propos}
	\label{lemma:thr_u}
	
	For $0<p<1$ the minimizer  (\ref{eq:min_surr_cw_u}) for generic $a,v$ can be computed by
	\begin{equation}
		\label{eq:minim_u}
			u^*_\lambda = H_\alpha^p (a_\lambda + [T^*(y-Tv-Ta)]_\lambda),\quad \lambda \in \Lambda,
	\end{equation}
	where the function $H_\alpha^p: \RR \rightarrow \RR$ obeys:
	\[
	H_\alpha^p (x)	= 
				\begin{cases}
					0, & |x| \leq \tau_\alpha, \\
					(F_\alpha^p)^{-1}(x),& |x|\geq\tau_\alpha,
				\end{cases}
	\]
	\[
			|H_\alpha^p(x)| \in \{0\} \cup \{x \geq \gamma_\alpha\}.
	\]
	Here, $(F_\alpha^p)^{-1}(x)$ is the inverse of the function $F_\alpha^p(t) =  t + \frac{\alpha p}{2} \sgn(t) |t|^{p-1},$ which is defined on $\RR_+,$  strictly convex and attains a minimum at $t_\alpha>0,$ and 
	\[
		\gamma_\alpha = (\alpha(1-p))^{1/(2-p)}, \quad \tau_\alpha = F_\alpha^p (\gamma_\alpha) =  \frac{2-p}{2-2p}(\alpha(1-p))^{1/(2-p)}. 
	\]
The thresholding function  $H_\alpha^p$  is continuous except at $|x| = \tau_\alpha,$ where it has a jump discontinuity.
\end{propos}
The proof of the proposition follows the similar arguments presented in Lemmas 3.10, 3.12 in \cite{BL09} and, thus, for the sake of brevity, it can be omitted here.
\begin{remark}
Since we  consider the case $p=0.5$ in our numerical experiments, we present here an explicit formulation of the thresholding function $H_\alpha^{1/2}$, which has been derived recently in~\cite{XuHalfThresholding}. It is given by
\[
	H_\alpha^{1/2} (x)	= 
				\begin{cases}
					0, & |x| \leq \frac{\sqrt[3]{54}}{4}(\alpha)^{2/3} , \\
					(F_\alpha^{1/2})^{-1}(x),& |x|\geq \frac{\sqrt[3]{54}}{4}(\alpha)^{2/3},
				\end{cases}
\]
where $$	
\left(F_{\alpha}^{1/2}\right)^{-1}(x) = \frac{2}{3}x \left( 1+\cos\left(\frac{2\pi}{3}-\frac{2}{3}\arccos\left(\frac{\alpha}{8}\left(\frac{|x|}{3}\right)^{-3/2}\right)\right)\right).
$$
\end{remark}

\begin{figure}[htp]
\begin{center}
\includegraphics[width=0.8\textwidth]{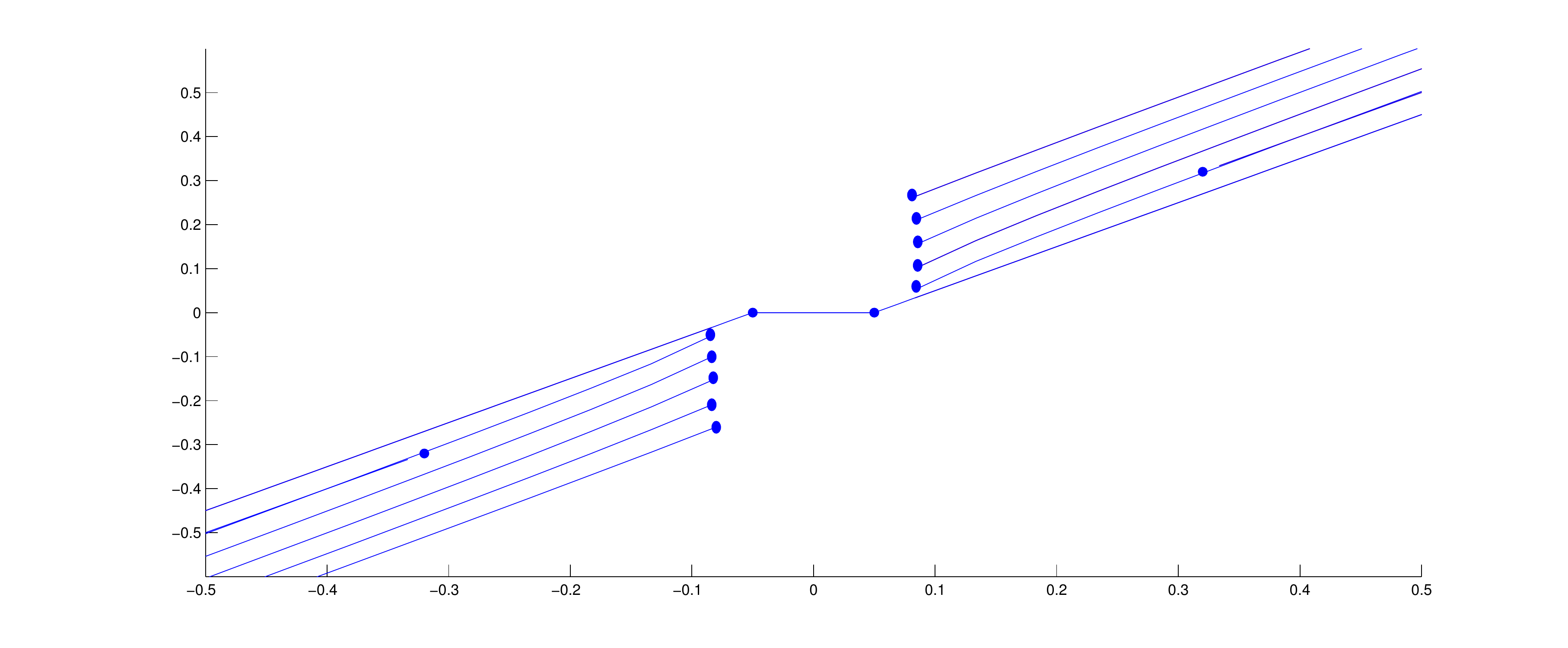}
\end{center}
\caption{The thresholding function $H_\alpha^p$ for $p=0, 0.15, 0.3, 0.45, 0.6, 0.9, 1$ and $\alpha=0.1$.}\label{fig:biribi1}
\end{figure}

For solving the low-level minimization problems of (\ref{algorithm_surrogate_func}), we can use an iterative thresholding algorithm induced by (\ref{eq:minim_v}) and (\ref{eq:minim_u}). Thus, we can reformulate the above algorithm in (\ref{algorithm_surrogate_func}) by
\begin{equation}
	\label{algorithm_thresholding_fun}
\left\{
	\begin{array}{l}
 			\left\{
				\begin{array}{l}
					u^{(n)} = u^{(n,L)} = u^{(n+1,0)}  \\
					u_\lambda^{(n+1,l+1)} =  H_\alpha^p (u_\lambda^{(n+1,l)} + [T^*(y-Tv^{(n,M)}-Tu^{(n+1,l)})]_\lambda), \quad l=0, \ldots, L-1 \\
				\end{array}
			\right.\\
			\left\{
				\begin{array}{l}
					v^{(n)} = v^{(n,M)} = v^{(n+1,0)}  \\
					v_\lambda^{(n+1,l+1)} =(1+\varepsilon)^{-1} S_{\beta,\varepsilon}^q(v_\lambda^{(n+1,l)} + [T^*(y-Tu^{(n+1)}-Tv^{(n+1,l)})]_\lambda), \quad l=0, \ldots, M-1 \\
				\end{array}
			\right.
	\end{array}.
\right.
\end{equation}

\subsubsection{Specifying the fixed points}
As the above algorithm may have multiple fixed points, it is important to analyze those in more detail. At first, we specify what we understand as a fixed point. 
Let us define
the functions
\begin{eqnarray}
	F_u(\bar u, \bar v) & = & \argmin_u J_u^s(u, \bar v; \bar u), 	\label{carrot_functions}\\
	F_v(\bar u, \bar v) & = & \argmin_v J_v^s(\bar u,  v; \bar v). 	\label{carrot_functions2}
\end{eqnarray}
Then we say that $(u^*,v^*)$ is a fixed point for the equations  (\ref{carrot_functions}) and (\ref{carrot_functions2}) if
\begin{equation}
	\label{carrot_fixed_points}
 			\left\{
				\begin{array}{l}
					u^{*} = F_u(u^*, v^*), \\
					v^{*} = F_v(u^*, v^*).
				\end{array}
			\right.\\
\end{equation}
We define $\mathcal F\mathit{ix}$ to be the set 	of fixed points for the equations  (\ref{carrot_functions}) and (\ref{carrot_functions2}).

\begin{lemma}
	\label{lemma:fixed_points}
	Let $(u^*,v^*) \in \mathcal F\mathit{ix}.$
	Define the sets $\Gamma_0 : = \{\lambda \in \Lambda: u^*_\lambda = 0\}$ and $\Gamma_1 : = \{\lambda \in \Lambda: |u^*_\lambda| \geq \gamma_\alpha \}.$ Then
	\begin{equation*}
		[T^*(y - Tv^* - Tu^*)]_\lambda  = \frac {\alpha p}2 \sgn(u^*_\lambda) |u^*_\lambda|^{p-1}, \quad \mbox{ if } \lambda \in \Gamma_1,
	\end{equation*}
		or
	\begin{equation*}
		| [T^*(y - Tv^* - Tu^*)]_\lambda| \leq \tau_\alpha, \quad \mbox{ if } \lambda \in \Gamma_0.
	\end{equation*}
\end{lemma} 
\begin{proof}
	By Proposition \ref{lemma:thr_u} 
		\[
		u^*_\lambda = H^p_\alpha (u^*_\lambda + [T^*(y-Tv^*-Tu^*)]_\lambda), ~\forall \lambda \in \Lambda.
	\]
If $u^*_\lambda = 0,$ this equality holds if and only if $| [T^*(y - Tv^* - Tu^*)]_\lambda| \leq \tau_\alpha.$ Similarly for $\lambda \in \Gamma_1$ we get
	\[
		u_\lambda^* = (F_\alpha^p)^{-1}(u^*_\lambda + [T^*(y-Tv^*-Tu^*)]_\lambda),
	\]
 and by definition of $F_\alpha^p$ we have $(F_\alpha^p)^{-1}(u^*_\lambda + \frac{\alpha p}{2} \sgn(u^*_\lambda)|u^*_\lambda|^{p-1}) = u_\lambda^*.$ Thus, the statement of the lemma follows.

\end{proof}



\subsubsection{Fixation of the index set $\Gamma_1$}

To ease notation, we define the operators $\HH^p_\alpha:~\ell_2(\Lambda) \rightarrow \ell_2(\Lambda)$ and $\mathbb S_\tau:~\ell_2(\Lambda) \rightarrow \ell_2(\Lambda)$ by their component-wise action
\begin{eqnarray*}
	(\HH^p_\alpha(u))_\lambda & := & H^p_\alpha(u_\lambda), \\
	( \mathbb S_\tau (v))_\lambda & := & S_{\beta, \varepsilon}^q (v_\lambda),
\end{eqnarray*}
here $\tau = \frac{\beta q}{2(1+\varepsilon)^{q-1}}.$

At the core of the proof of convergence stands the fixation of the \enquote{discontinuity set} during the iteration (\ref{eq:minim_u}), at which point the non-convex and non-smooth minimization with respect to $u$ in (\ref{algorithm_surrogate_func}) is transformed into a simpler  problem.
\begin{lemma}
	\label{lemma:disc_set}
Consider the iterations
\[
	u^{(n+1,l+1)}  = \HH^p_\alpha(u^{(n+1,l)} + T^*(y - Tv^{(n,M)} - Tu^{(n+1,l)}))
\]
and the partition of the index set $\Lambda$ into 
\begin{eqnarray*}
	\Gamma_1^{n,l} & = & \{ \lambda \in \Lambda: |u_\lambda^{(n,l)}| \geq \gamma_\alpha \}, \\
	\Gamma_0^{n,l} & = & \{ \lambda \in \Lambda: u_\lambda^{(n,l)} = 0 \},
\end{eqnarray*}
where $(\tau_\alpha, \gamma_\alpha)$ is the position of the jump-discontinuity of the thresholding function. For sufficiently large $N \in \NN$ (after a finite number of iterations), this partition fixes during the iterations, meaning there exists
$\Gamma_0$ such that for all $n \geq N, \forall ~l \leq L,~\Gamma_0^{n,l} = \Gamma_0$ and $\Gamma_1^{n,l} = \Gamma_1 : = \Lambda \setminus \Gamma_0.$ 
\end{lemma}
\begin{proof}
By discontinuity of the thresholding function $H^p_\alpha(x),$ each sequence component satisfies
\begin{itemize}
	\item $u_\lambda^{(n,l)}= 0$ if $\lambda \in \Gamma_0^{n,l};$
	\item $|u_\lambda^{(n,l)}| \geq \gamma_\alpha$ if $\lambda \in \Gamma_1^{n,l}.$
\end{itemize}

Thus, $| u_\lambda^{(n,l+1)} - u_\lambda^{(n,l)}| \geq \gamma_\alpha$ if $\lambda \in \Gamma_0^{n,l+1} \cap \Gamma_1^{n,l}$, or $\lambda \in \Gamma_0^{n,l} \cap \Gamma_1^{n,l+1}.$ At the same time, Lemma \ref{lemma1_general_surrigate} implies that 
\[
	|u_\lambda^{(n,l+1)} - u_\lambda^{(n,l)}| \leq \| u^{(n,l+1)} - u^{(n,l)} \|_2 \leq \epsilon,
\]
for sufficiently large $n \geq N(\epsilon).$ In particular, the last inequality implies that $\Gamma_0$ and $\Gamma_1$ must be fixed once $n \geq N(\epsilon).$
\end{proof}

Since $(u^{(n,l)}) \in  \ell_2,$ the set  $\Gamma_1$ is finite. 
Moreover, fixation of the index set $\Gamma_1$ implies that the sequence $(u^{(n)})$ can be considered  constrained to a subset of $\ell_2(\Lambda)$ on which the functionals $J_{p,q}(\cdot, v)$ and $J_u^s(\cdot,v;a)$ are differentiable.



\section{Convergence of the iterative algorithm}
\subsection{Weak convergence}
Given that $H^p_\alpha (u_\lambda^{(n)}) = (F^p_\alpha)^{-1}(u_\lambda^{(n)}),~\lambda \in \Gamma_1,$ after a finite number of iterations, we can use the tools from real analysis to prove that the sequence $(u^{(n)})$ converges to some stationary point. Notice that the convergence of the iterations of the type 
\begin{equation*}
				\begin{array}{l}
					u^{(n+1)} = F_u(u^{(n)}, \bar v), \\
					v^{(n+1)} = F_v(\bar u, v^{(n)}),
				\end{array}
\end{equation*}
to a fixed point $(\bar u^*, \bar v^*)$ for any $(\bar u, \bar v)$ and $F_u,F_v$ given as in (\ref{carrot_functions}) has been extensively discussed in the literature, e.g., \cite{BD08, BL09, DDD04, fora08}. 

\begin{theorem}
	\label{theorem:weak_conv}
	Assume $0 < p < 1$ and $2 \leq q < \infty.$
	The algorithm (\ref{algorithm_surrogate_func}) produces sequences $(u^{(n)}),~(v^{(n)})$ in $\ell_2$ whose weak accumulation points are fixed points of the equations (\ref{carrot_functions})-(\ref{carrot_functions2}).
\end{theorem}
\begin{proof}
By Lemma \ref{lemma:disc_set} the iteration step
\[
	u_\lambda^{(n+1,l+1)} =  H^p_\alpha (u_\lambda^{(n+1,l)} + [T^*(y-Tv^{(n,M)}-Tu^{(n+1,l)})]_\lambda)
\]
becomes equivalent to the step of the form 
\[
u_\lambda^{(n+1,l+1)} = (F^p_\alpha)^{-1}(u_\lambda^{(n+1,l)}+ [T ^*(y - T v^{(n,M)} - T u^{(n+1,l)})]_\lambda   ),~\lambda \in \Gamma_1,
\]
after a finite number of iterations and $u_{\lambda'}^{(n+1,l+1)} = 0,~\forall \lambda' \in \Lambda \setminus \Gamma_1 = \Gamma_0.$


	From (\ref{eq:decr_u}) and (\ref{eq:decr_v})  we have
	\[
		J_{p,q}(u^{(0)}, v^{(0)}) \geq  J_{p,q}(u^{(n)},v^{(n)}) \geq \alpha \| u^{(n)} \|_p^p \geq \alpha \| u^{(n)} \|_2^p ,
	\]
	and 
	\[
		J_{p,q}(u^{(0)}, v^{(0)}) \geq  J_{p,q}(u^{(n+1)},v^{(n)}) \geq \beta \| v^{(n)} \|^q_q + \varepsilon \| v^{(n)} \|^2_2\geq   \varepsilon \| v^{(n)} \|^2_2. 
	\]	
These mean that $(u^{(n)})$ and $(v^{(n)})$ are uniformly bounded in $\ell_2,$ hence there exist weakly convergent subsequences $(u^{(n_j)})$ and $(v^{(n_j)}).$ Let us denote by $u^\infty$ and $v^\infty$ the weak limits of the corresponding subsequences. For simplicity, we rename the corresponding subsequences $(u^{(n)})$ and $(v^{(n)}).$ Moreover, since the sequence  $J_{p,q}(u^{(n)}, v^{(n)})$ is monotonically decreasing and bounded from below by 0, it is also convergent.

First of all, let us recall that the weak convergence implies component wise convergence, so that $u^{(n)}_\lambda \rightarrow u^\infty_\lambda,~v^{(n)}_\lambda \rightarrow v^\infty_\lambda,$ and $[T^*Tu^{(n)}]_\lambda \rightarrow [T^*Tu^{(\infty)}]_\lambda,~[T^*Tv^{(n)}]_\lambda \rightarrow [T^*Tv^{(\infty)}]_\lambda.$

By definition of $u^{(n+1,L)}$ and $v^{(n+1,M)}$ in (\ref{algorithm_thresholding_fun}), we have for $n$ large enough
\begin{eqnarray}
	0 & = & [-2(u^{(n+1,L-1)} + T^* (y-Tv^{(n)}) - T^*Tu^{(n+1,L-1)})]_\lambda + 2 u_\lambda^{(n+1,L)} \nonumber\\
	&&+ \alpha p \sgn (u_\lambda^{(n+1,L)} ) |u_\lambda^{(n+1,L)}|^{p-1} , ~\lambda \in \Gamma_1,\label{FO_NOC_u}\\
	0 & = & [-2(v^{(n+1,M-1)} + T^* (y-Tu^{(n+1,L)}) - T^*Tv^{(n+1,M-1)})]_\lambda + (2 + \varepsilon) v_\lambda^{(n+1,M)} \nonumber \\
	&& + \beta q \sgn (v_\lambda^{(n+1,M)} ) |v_\lambda^{(n+1,M)}|^{q-1}, ~\lambda \in \Lambda.  \label{FO_NOC_v}
\end{eqnarray}
By taking now the limit for $n \rightarrow \infty$ in (\ref{FO_NOC_u}) and (\ref{FO_NOC_v}), and by using Lemma \ref{lemma1_general_surrigate} we obtain
\begin{eqnarray*}
	0 & = & [-2(u^{\infty} + T^* (y-Tv^{\infty}) - T^*Tu^{\infty})]_\lambda + 2 u_\lambda^{\infty}  + \alpha p \sgn (u_\lambda^{\infty} ) |u_\lambda^{\infty}|^{p-1},~\lambda \in \Gamma_1,\\
	0 & = & [-2(v^{\infty} + T^* (y-Tu^{\infty}) - T^*Tv^{\infty})]_\lambda + (2 + \varepsilon) v_\lambda^{\infty} + \beta q \sgn (v_\lambda^{\infty} ) |v_\lambda^{\infty}|^{q-1}, ~\lambda \in \Lambda. 
\end{eqnarray*}
An application of Lemma \ref{lemma:fixed_points} implies $ (u^*,v^*) = (u^\infty, v^\infty),$ i.e., 
\[
\begin{array}{l}
	u^\infty_\lambda =  H_\alpha^p(u_\lambda^\infty + [T^* (y-Tv^\infty-Tu^\infty)]_\lambda), ~\lambda \in \Gamma_1. \\
	v^\infty_\lambda  =  (1+\varepsilon)^{-1} S_{\beta, \varepsilon}^q (v_\lambda^\infty + [T^* (y-Tv^\infty-Tu^\infty)]_\lambda), ~\lambda \in \Lambda. 
\end{array}
\]
The  argumentation holds true for every subsequences of $(u^{(n)})$ and $(v^{(n)}).$
\end{proof}

\begin{remark}
The case $q=\infty$ would need a special treatment due to lack of differentiability. 
 For simplicity we further assume that $2 \leq q  < \infty.$
\end{remark}

\subsection{Minimizers of $J_{p,q}$}

In this section we explore the relationship between a limit point  $(u^*,~v^*)$ of the iterative thresholding algorithm  (\ref{algorithm_surrogate_func}) and minimizers of the functional (\ref{eq:funct_general}). 
We shall show that under the so-called {\it finite basis injectivity} (FBI) property \cite{BL09} the set of fixed points of the algorithm is a subset of the set of local minimizers. 
We note that here again we provide the proof  only for  the case $0<p<1,$ and we refer to   \cite{BD08} and \cite{DDD04} for the cases $p=0$ and $1\leq p<2$, which follow similarly after minor adaptations.

\begin{theorem}
	\label{theorem:minimizers}
Let $T$ have the FBI property, i.e., $T$ is injective whenever restricted to finitely many coefficients.
Let us denote $\mathcal L$ the set of local minimizers of $J_{p,q}$. Then we have the following inclusion
\[
\mathcal F\mathit{ix} \subset \mathcal L,	
\]
where $\mathcal F\mathit{ix}$ is the set	 of fixed points for equations (\ref{carrot_functions})-(\ref{carrot_functions2}).
\end{theorem}
The rest of the subsection addresses the proof of Theorem \ref{theorem:minimizers}. 
In fact, we first show that the choice of a sufficiently small $p \in (0,1)$ guarantees that an accumulation point $(u^*,v^*)$ is a local minimizer of the functional with respect to $u.$
A main ingredient in the proof of this fact is the FBI property. 
\begin{propos}
	\label{propos:min_u}
Let $T$ satisfy the FBI property. Then there exists $p^* \in (0,1)$ such that for every $0<p<p^*$ every accumulation point  $(u^*,v^*)$ is a  local minimizer of the functional $J_{p,q}$ with respect to $u,$ i.e.,
\begin{equation}
\label{eq:min_u_func}
			J_{p,q} (u^*+du,v^*) \geq  J_{p,q} (u^*,v^*) 
\end{equation}
for any $du \in \ell_2(\Lambda), ~\| du \|_2 \leq \epsilon_1$ for $\epsilon_1$ sufficiently small.
\end{propos}

\begin{proof}

In the following, we denote by  $J_{p,q}^{\Gamma_1}$  the restriction of the functional $J_{p,q}$  to $\ell_2(\Gamma_1),$ i.e., 
\begin{equation}
	\label{eq:funct_general_restricted}
	J^{\Gamma_1}_{p,q}(u,v) : = \| T(u+v)-y \|^2_{\mathcal H}+  \alpha\sum_{\lambda \in \Gamma_1} | u_\lambda |^{p} + \left (\beta  \| v\|_{q}^{q} + \varepsilon \| v\|_{2}^2 \right ),
\end{equation}
and by $J_{u,\Gamma_1}^s,~J_{v,\Gamma_1}^s$  the corresponding surrogate functionals restricted to $\ell_2(\Gamma_1)$. 

For the sake of simplicity, let us define
\[
	F(u)  = J_{p,q} (u,v), ~\quad F^{\Gamma_1}(u) =  J_{p,q}^{\Gamma_1} (u,v).
\]
We proceed with the proof of the lemma in two steps:
\begin{itemize}
	\item We show that an accumulation point  $(u^*,v^*)$ is a local minimizer of $F^{\Gamma_1}(u);$
	\item We show that $(u^*,v^*)$ is a local minimizer of $F(u).$
\end{itemize}

Let us for now consider the $u_\lambda^{(n)}$ for $\lambda \in \Gamma_1,$ i.e.,  $|u^{(n)}_\lambda| \geq \gamma_\alpha$.
Since $u^*$ is an accumulation point for $(u^{(n)}),$ by Theorem \ref{theorem:weak_conv} it is also a fixed point. Taking into account the restriction to the set $\Gamma_1,$
by Lemma \ref{lemma:fixed_points} we get
\[
			[T^*(Tv^* + Tu^* - y)]_\lambda  + \frac {\alpha p}2 \sgn(u^*_\lambda) |u^*_\lambda|^{p-1} = 0.
\]
As the functional $F^{\Gamma_1}(u)$ is differentiable on $\ell_2(\Gamma_1)$, we compute the Jacobian
\[
	\nabla F^{\Gamma_1}(u) = 2T^*(Tv + Tu - y)  +  {\alpha p}u |u|^{p-2}, 
\]
for which holds $ \nabla F^{\Gamma_1}(u^*) =0,~v=v^*$. Since the mapping is smooth for all $u_\lambda \neq 0,$ one can check additionally that the Hessian matrix
 \[
	 \nabla^2 F^{\Gamma_1}(u^*) = 2T^*T  -  {\alpha p} (1-p) \diag(|u^*|^{p-2}), 
\]
is actually positive definite for $p < p^*$: For $z$ with $\supp z \subset \supp u^*$ we  have the following estimate
\begin{eqnarray*} 
	\langle z,  \nabla^2 F^{\Gamma_1}(u^*) z \rangle  & = &  2\| Tz \|^2_{\mathcal H} -   {\alpha p}(1-p) \sum_{\lambda \in \Gamma_1} |u^*_\lambda|^{p-2}  z_\lambda^2  \\
								& \geq & (c -  {\alpha p} (1-p) \gamma_\alpha^{p-2}) \| z\|^2_2 = (c-p) \| z\|^2_2,
\end{eqnarray*}
where $c>0$ is the the smallest eigenvalue of $T^*T.$
Therefore, for all  $p \in (0, p^*),~ p^* = \min \{ 1, c \},$ the Hessian is positive definite and thus $u^*$ is a local minimizer of $F^{\Gamma_1}.$

Next we show that $u^*$ is a local minimizer of the functional $F(u)$ without the restriction on the support of $u^*.$	
For the sake of transparency, we shall write the restrictions $u_{\Gamma_1},~u^*_{\Gamma_1}$ and $du_{\Gamma_1}$ meaning that $u_{\Gamma_1},  u^*_{\Gamma_1}, du_{\Gamma_1} \in \ell_2(\Gamma_1)$, and $du_{\Gamma_0}$ meaning that $du_{\Gamma_0} \in \ell_2(\Gamma_0).$ 

The desired statement of the proposition follows if we can show that  $F^{\Gamma_1}(u^*_{\Gamma_1} + du_{\Gamma_1}) \leq F(u^* + du).$  At this point it is convenient to write the functional $F(u^* + du)$ with $\bar y : = y - Tv^*$ as 
\[
	F(u^*+ du) = \|T_{\Gamma_1} (u^*_{\Gamma_1} + du_{\Gamma_1}) + T_{\Gamma_0} du_{\Gamma_0} -\bar y \|^2_{\mathcal H} + \alpha\| u_{\Gamma_1} \|^p_{p} + \alpha \| du_{\Gamma_1} \|_p^p + \alpha \| du _{\Gamma_0} \|^p_p +\beta \| v \|_q^q + \varepsilon  \| v \|_2^2. 
\]
Moreover, the inequality  $F^{\Gamma_1}(u^*_{\Gamma_1} + du_{\Gamma_1}) \leq F(u^* + du)$  can be written as
\[
-\alpha \| du _{\Gamma_0} \|^p_p \leq \|T_{\Gamma_1} (u^*_{\Gamma_1} + du_{\Gamma_1}) + T_{\Gamma_0} du_{\Gamma_0} -\bar y \|^2_{\mathcal H}  -
 \|T_{\Gamma_1} (u^*_{\Gamma_1} + du_{\Gamma_1}) -\bar y \|^2_{\mathcal H}.
\]
By developing the squares, we obtain
\begin{eqnarray*}
 -  \alpha \| du _{\Gamma_0} \|^p_p &\leq&  2 \langle  T_{\Gamma_1} (u^*_{\Gamma_1} + du_{\Gamma_1}) -\bar y, T_{\Gamma_0}  du_{\Gamma_0} \rangle + \|  T_{\Gamma_0} du_{\Gamma_0} \|^2_{\mathcal H} \\
 &=& 2( \langle T_{\Gamma_1} (u^*_{\Gamma_1} + du_{\Gamma_1}), T_{\Gamma_0} du_{\Gamma_0}\rangle - \langle T_{\Gamma_0} du_{\Gamma_0} , \bar y \rangle)  +  \| T_{\Gamma_0} du_{\Gamma_0} \|^2_{\mathcal H},
\end{eqnarray*}
for $\| du _{\Gamma_0} \|_2$ sufficiently small. 
 One concludes by observing that for $p<1$ the term $\| du _{\Gamma_0} \|^p_p$ will always dominate the linear terms on the right-hand side of the above inequality.

\end{proof}

\begin{propos}
	\label{propos:min_v}
Every accumulation point  $(u^*,v^*)$ is a  local minimizer of the functional $J_{p,q}$ with respect to $v,$ i.e.,
\begin{equation}
\label{eq:min_v_func}
			J_{p,q} (u^*,v^*+dv) \geq  J_{p,q} (u^*,v^*) 
\end{equation}
for any $dv \in \ell_2(\Lambda), ~\| dv \|_2 \leq \epsilon_2$ for $\epsilon_2>0$ sufficiently small.
\end{propos}

\begin{proof}
First of all, we claim that $J_v^s(u^*,v^* + dv; v^*) - J_v^s(u^*,v^*; v^*)  \geq  \| dv\|^2_2.$
Indeed, a direct calculation shows that
	\begin{eqnarray*}
		J_v^s(u^*,v^* + dv; v^*) - J_v^s(u^*,v^*; v^*)  & = & \| T (u^* + v^* + dv) - y\|^2_{\mathcal H} + \beta \|ï¿½v^* + dv \|^q_q + \varepsilon \|ï¿½v^* + dv\|^2_2 \\
		&& + ï¿½ \| dv\|^2_2 - \| Tdv\|^2_{\mathcal H} 
		- \| T (u^*+ v^*) - yï¿½\|^2_{\mathcal H} - \beta \|ï¿½v^* \|^q_q - \varepsilon \| v^* \|^2_2  \\
		& = &\| dv\|^2_2  + \beta  \sum_{\lambda \in \Lambda} (|ï¿½v^*_\lambda + dv_\lambda |^q - |ï¿½v^*_\lambda |^q) + \varepsilon \sum_{\lambda \in \Lambda} (|ï¿½v^*_\lambda + dv_\lambda |^2 - |ï¿½v^*_\lambda |^2)\\
		& + &  \sum_{\lambda \in \Lambda} dv_\lambda [T^* (T(u^* + v^*) - y)]_\lambda \\
		& \geq & (1+\varepsilon)\| dv\|^2_2  + \sum_{\lambda \in \Lambda} dv_\lambda ( [T^* (T(u^* + v^*) - y)]_\lambda\\
		&  + & \beta  q \sgn(v_\lambda^*) |v_\lambda^*|^{q-1} + 2\varepsilon   v_\lambda^* ). \\
	\end{eqnarray*}

Since by (\ref{eq:min_surr_cw_v}) the term 
\[
	\langle [T^*(T(u^*+v^*) - y)]_\lambda + \beta q \sgn(v^*_\lambda)|v^*_\lambda|^{q-1} + 2 \varepsilon v^*_\lambda, t_\lambda - v^*_\lambda \rangle 
\]
vanishes,  the above claim follows. By using the above claim, we get that
\begin{eqnarray*}
		 J_{p,q}(u^*, v^*+dv) & = &  J_v^s (u^*, v^*+dv; v^*)  - \|dv \|^2_2 + \| Tdv \|^2_{\mathcal H} \\
					  & \geq &  J_v^s (u^*, v^*+dv; v^*) - \|dv^* \|^2_2  \\
					  & \geq &J_v^s (u^*, v^*; v^*)  =  J_{p,q}(u^*,v^*).
	\end{eqnarray*}

\end{proof}

%
%
With the obtained results we are now able to prove Theorem~\ref{theorem:minimizers}. In particular, we shall show that $J_{p,q} (u^*, v^*) \leq J_{p,q} (u^*, v^* + dv) \leq J_{p,q} (u^* + du, v^* + dv).$ The first inequality has been proven in Proposition \ref{propos:min_v}. We  only need to show the second inequality.
 
\begin{proof}[Proof of Theorem \ref{theorem:minimizers}]
Similarly as in Proposition \ref{propos:min_u} we proceed in two steps. First we prove that  $J^{\Gamma_1}_{p,q}(u^*, v^* + dv) \leq J_{p,q}^{\Gamma_1} (u^* + du, v^* + dv).$   Since the functional $J^{\Gamma_1}_{p,q}$ is differentiable, a Taylor expansion at $(u^*, v^* + dv)$ yields
\[
J_{p,q}^{\Gamma_1} (u^* + du, v^* + dv) = J^{\Gamma_1}_{p,q}(u^*, v^* + dv)  + \nabla J^{\Gamma_1}_{p,q}(u^*, v^* + dv)  du + \frac12 du \nabla^2 J^{\Gamma_1}_{p,q}(u^*, v^* + dv) du.
\] 
Due to Proposition \ref{propos:min_u},  $\nabla F^{\Gamma_1}(u^*) = \nabla J^{\Gamma_1}_{p,q}(u^*, v^*) = 0$ and the term $\nabla J^{\Gamma_1}_{p,q}(u^*, v^* + dv) = 2 T^*T dv \approx 0.$ Thus, 
\[
-2 \| T \|^2 \| dv \|_2 \| du\|_2 \leq \nabla J^{\Gamma_1}_{p,q}(u^*, v^* + dv)  du =  2 \langle T dv, Tdu  \rangle \leq 2 \| T \|^2 \| dv \|_2 \| du\|_2.
\]
Moreover, 
\[
		 \nabla^2 J^{\Gamma_1}_{p,q}(u^*, v^* + dv) = \nabla^2  J^{\Gamma_1}_{p,q}(u^*, v^*) + \xi(\| dv \|^2_2),
\]
where $\nabla^2  J^{\Gamma_1}_{p,q}(u^*, v^*) \geq 0$ due to the local convexity of the functional $J^{\Gamma_1}_{p,q}.$
Choosing $\eta \leq \frac{c-p}{2\|T\|^2}$ and $\| dv \|_2 = \eta \| du\|_2$, and combining the above estimates together, we get  
\begin{align*}
	\nabla J^{\Gamma_1}_{p,q}(u^*, v^* + dv)  du &+ \frac12 du \nabla^2 J^{\Gamma_1}_{p,q}(u^*, v^* + dv) du \geq -2 \| T \|^2 \| dv \|_2 \| du\|_2 + (c-p) \| du\|^2_2 \\
						&\geq [(c-p) - 2 \eta \|T\|^2] \|du\|^2_2 \geq 0,
\end{align*}
and thus, $J_{p,q}^{\Gamma_1} (u^* + du, v^* + dv)  \geq J_{p,q}^{\Gamma_1} (u^*, v^* + dv). $ 
The second part of the proof is concerned with the inequality $J_{p,q} (u^* + du, v^* + dv) \geq J_{p,q} (u^*, v^* + dv)$, which works in a completely analogous way to the second part of the proof of Proposition  \ref{propos:min_u}  and is therefore omitted here.

\end{proof}
\subsection{Strong convergence} 

In this subsection we show how the weak convergence established in the preceding subsections can be strengthened into norm convergence, also by a series of lemmas. Since the distinction between weak and strong convergence makes sense only when the index set $\Lambda$ is infinite, we shall prove the strong convergence only for the sequence $v^{(n)},$ since the iterates $u^{(n)}$ are constrained to the finite set after a finite number of iterations. 

For the sake of convenience, we introduce the following notation. 
\begin{align*}
\mu^{n+1} = v^{(n+1)} - v^*,  & \quad\quad \mu^{n+1/2} = v^{(n+1,M-1)} - v^*, & \\
\eta^{n+1} = u^{(n+1)} - u^*,  & \quad\quad  h = v^* + T^*(y - Tu^* - Tv^*), &
\end{align*}
where $v^* = \wlim_{n \rightarrow \infty} v^{(n)}$ and $u^* = \lim_{n \rightarrow \infty} u^{(n)}.$ Here and below, we use $\wlim$ as a shorthand for weak limit.
For the proof of strong convergence we need the following technical lemmas.
\begin{lemma}\cite{DDD04, fora08}
	The operator $\SS_\tau(v)$ is non-expansive, i.e., $\| \SS_\tau(u) - \SS_\tau(v) \|_2  \leq \|u - v \|_2. $
\end{lemma}
\begin{lemma}
	Assume $\| \mu^{n+1/2} \|_2 > \gamma$ for all $n$ and for a fixed $\gamma>0.$ Then $\| T \mu^{n+1/2} \|^2_{\mathcal H} \rightarrow 0$ as $n \rightarrow \infty.$
\end{lemma}
\begin{proof}
Since
\[
	\mu^{n+1}  - \mu^{n+1/2} = (1+\varepsilon)^{-1} [\mathbb S_\tau (h + (I-T^*T)\mu^{n+1/2} - T^*T \eta^{n+1}) - \mathbb S_\tau(h) ] -  \mu^{n+1/2},
\]
and
\[
	\|ï¿½\mu^{n+1}  - \mu^{n+1/2} \|_2 = \|ï¿½v^{(n+1,M)} -   v^{(n+1,M-1)} \|_2 \rightarrow 0,
\]
by Lemma \ref{lemma1_general_surrigate}, we get the following
\begin{equation}
		 \label{eq_carrot_StCo}
\begin{array}{c l}
	\| (1+\varepsilon)^{-1} & [\mathbb S_\tau (h + (I-T^*T)\mu^{n+1/2} - T^*T \eta^{n+1}) - \mathbb S_\tau(h) ] -  \mu^{n+1/2} \|  \\
	 & \geq | (1+\varepsilon)^{-1}   \| \mathbb S_\tau (h + (I-T^*T)\mu^{n+1/2} - T^*T \eta^{n+1}) - \mathbb S_\tau(h)  \|  - \| \mu^{n+1/2} \| | \\
	& \rightarrow 0.  
	\end{array}
	\end{equation}
By non-expansiveness of $\mathbb S_\tau,$ we have the estimate
\[
  \| \mathbb S_\tau (h + (I-T^*T)\mu^{n+1/2} - T^*T \eta^{n+1}) - \mathbb S_\tau(h)  \|_2 \leq \| (I-T^*T) \mu^{n+1/2} - T^*T\eta^{n+1}ï¿½\|_2.
\]
Consider now
\begin{align}
	\| (I-T^*T) \mu^{n+1/2} - &T^*T\eta^{n+1}ï¿½\|^2_2 \nonumber \\
	& \leq  \|ï¿½ (I-T^*T) \mu^{n+1/2} \|^2_2 + \| T^*T\eta^{n+1}ï¿½ \|^2_2  - 2 \langle (I-T^*T) \mu^{n+1/2}, T^*T\eta^{n+1} \rangle \nonumber \\
									 & \leq  \|ï¿½ (I-T^*T) \mu^{n+1/2} \|^2_2 + \| T^* T \eta^{n+1}\|_2 + 2 \| (I-T^*T) \mu^{n+1/2}\|_2 \| T^*T\eta^{n+1} \|_2 \nonumber \\
									  & \leq   \|ï¿½ (I-T^*T) \mu^{n+1/2} \|^2_2 + \delta + 2 C \delta \nonumber \\
									 & \leq   \|ï¿½  \mu^{n+1/2} \|^2_2 + \epsilon, 
									\label{eq:bound_diff_str_conv}
\end{align}
for large enough $n$  so that $\| u^{(n+1,L)} - u^* \|_2ï¿½\leq \delta.$ The constant $C>0$ is due to the boundedness of $\| \mu^{n+1/2} \|.$ 
Due to estimate (\ref{eq:bound_diff_str_conv}),  we  have
\begin{eqnarray*}
   \| \mathbb S_\tau (h + (I-T^*T)\mu^{n+1/2} - T^*T \eta^{n+1}) - \mathbb S_\tau(h)  \|_2 & \leq &
   	 \| (I-T^*T) \mu^{n+1/2} - T^*T\eta^{n+1}ï¿½\|_2 \\
	 &\leq &   \sqrt{\|ï¿½\mu^{(n + 1/2)} \|_2 + \epsilon} \\
	 &\leq &\|ï¿½\mu^{(n + 1/2)} \|_2 + \bar\epsilon.
\end{eqnarray*}

By assumption of the lemma there exists a subsequence $(\mu^{n_k+1/2})$ such that  $\| \mu^{n_k+1/2} \|_2 \geq \gamma $ for all $k.$ For simplicity, we rename such subsequence as $(\mu^{n+1/2})$ again. Then 
$$
   (1+\varepsilon)^{-1}  \| \mathbb S_\tau (h + (I-T^*T)\mu^{n+1/2} - T^*T \eta^{n+1}) - \mathbb S_\tau(h)  \|_2 \leq \frac{1}{1+\varepsilon}  \|ï¿½\mu^{(n + 1/2)} \|_2 +  \frac{1}{1+\varepsilon} \bar\epsilon.
$$
For $\bar\epsilon \leq \varepsilon \gamma$ we obtain 
\begin{align*}
   (1+\varepsilon)^{-1}  \| \mathbb S_\tau (h + (I-T^*T)\mu^{n+1/2} - & T^*T \eta^{n+1}) - \mathbb S_\tau(h)  \|_2 \\
    & \leq  \frac{1}{1+\varepsilon}  \|ï¿½\mu^{(n + 1/2)} \|_2+  \frac{1}{1+\varepsilon}(1+ \varepsilon - 1) \gamma \\
   & \leq   \frac{1}{1+\varepsilon}  \|ï¿½\mu^{(n + 1/2)} \|_2  + \left(1- \frac{1}{1+\varepsilon} \right) \|ï¿½\mu^{(n + 1/2)} \|_2 \\
	&\leq  \|ï¿½\mu^{(n + 1/2)} \|_2.
\end{align*}

Combining the  above inequalities, we get 
	\begin{align*}	\|\mu^{(n + 1/2)} \|^2_2  & -   \|ï¿½ (I-T^*T) \mu^{n+1/2} - T^*T \eta^{n+1} \|^2_2 \\
	  & \leq  \| \mu^{(n + 1/2)} \|^2_2  - (1+\varepsilon)^{-1}   \| \mathbb S_\tau (h + (I-T^*T)\mu^{n+1/2} - T^*T \eta^{n+1}) - \mathbb S_\tau(h)  \|^2_2. 
	\end{align*}
This implies from (\ref{eq_carrot_StCo}) that
\[
	\lim_{n \rightarrow \infty} [ \| \mu^{(n + 1/2)} \|^2_2 -   \| (I-T^*T)\mu^{n+1/2} - T^*T \eta^{n+1} \|^2_2 ] = 0.
\]
Using (\ref{eq:bound_diff_str_conv}) we get
	\begin{align*}
	  \| \mu^{(n + 1/2)} \|^2_2 & - \|ï¿½ (I-T^*T) \mu^{n+1/2} - T^*T \eta^{n+1} \|^2_2 \\
	   & \geq \| \mu^{(n + 1/2)} \|^2_2 - \|ï¿½ (I-T^*T) \mu^{n+1/2} \|^2_2- \epsilon \\
	  &  = 2 \| T \mu^{n+1/2} \|^2_{\mathcal H} - \|T^*T \mu^{n+1/2} \|^2_2 - \epsilon \\
	  & \geq \| T \mu^{n+1/2} \|^2_{\mathcal H}   - \epsilon.
	\end{align*}
This yields $\| T \mu^{n+1/2} \|^2_{\mathcal H} \rightarrow 0$ for $n \rightarrow \infty.$
\end{proof}

\begin{lemma}
	\label{lemma_str_conv_0}
	For $h = v^* + T^*(y - Tu^* - Tv^*),$ 
	\[
		\|ï¿½(1+\varepsilon)^{-1} [\mathbb S_\tau (h + \mu^{n+1/2}) - \mathbb S_\tau(h)] - \mu^{n+1/2} \|_2 \rightarrow 0.
	\]
\end{lemma}

\begin{proof}
\begin{align*}
	\|ï¿½(1+\varepsilon)^{-1}& [\mathbb S_\tau (h + \mu^{n+1/2}) - \mathbb S_\tau(h)] - \mu^{n+1/2} \|_2 \\
	& \leq  \|ï¿½(1+\varepsilon)^{-1} \mathbb S_\tau (h + \mu^{n+1/2} - T^*T \mu^{n+1/2} - T^* T \eta^{n+1}) -  (1+\varepsilon)^{-1} \mathbb S_\tau(h) - \mu^{n + 1/2} \|_2 \\
	& + \| (1+\varepsilon)^{-1} [\mathbb S_\tau (h + \mu^{n+1/2})  -  \mathbb S_\tau (h + \mu^{n+1/2} - T^*T \mu^{n+1/2} - T^* T \eta^{n+1})] \|_2 \\
	& \leq  \|ï¿½(1+\varepsilon)^{-1} [\mathbb S_\tau (h + (I - T^* T) \mu^{n+1/2} - T^*T\eta^{n+1} - \mathbb S_\tau(h)] -  \mu^{n + 1/2} \|_2  \\
	& +  (1 + \varepsilon)^{-1} \| T^*T(\mu^{n+1/2} + \eta^{n+1}) \|_2,
\end{align*}
where we used the non-expansivity of the operator. The result follows since both terms in the last bound tend to 0 for $n \rightarrow \infty$ because of the previous lemma and Theorem \ref{theorem:weak_conv}.
\end{proof}

\begin{lemma}
	\label{lemma:embarace}
If for some $a \in \ell_2$ and some sequence $(v^{(n)}),~\wlim_{n \rightarrow \infty} v^{(n)} = 0$ and $\lim_{n \rightarrow \infty} \| (1+\varepsilon)^{-1}[\mathbb S_\tau (a + v^{(n)}) - \mathbb S_\tau (a)] - v^{(n)} ï¿½\|_2 =0,$ then $\|ï¿½v^{(n)} \|_2 =  0$ for $n \rightarrow \infty.$
\end{lemma}

\begin{proof}
In the proof of the lemma we mainly follow the arguments in \cite{DDD04}. Since the sequence $(v^{(n)})$ is weakly convergent, it has to be bounded: there is a constant $K$ such that for all $n,~\|ï¿½v^{(n)}\|_2 \leq K.$ Reasoning component-wise we can write $|v_\lambda^{(n)}| < K$ for all $\lambda \in \Lambda.$

Let us define the set $\Gamma_0 = \{ \lambda \in \Lambda: |a_\lambda| \geq K \}$ and since $a \in \ell_2(\Lambda),$ this is a finite set. We then have $\forall \lambda \in \Gamma_1 = \Gamma \setminus \Gamma_0,$ that $|a_\lambda|$ and $|a_\lambda + v_\lambda^{(n)}|$ are bounded above by $2K.$ Recalling the definition of $S_{\beta, \varepsilon}^q = (F_{\beta, \varepsilon}^q )^{-1},$ we observe that for $q\geq 1$ and $|x| \leq 2K$,
\[
	(F_{\beta, \varepsilon}^q )'(x) = 1+ \frac{\beta q (q-1)}{2(1+\varepsilon)^{q-1}} |x|^{q-2}  \geq 1,
\]

and therefore 
\begin{align*}
|(1+\varepsilon)^{-1}[S_{\beta, \varepsilon}^q (a_\lambda + v^{(n)}_\lambda) - S_{\beta, \varepsilon}^q (a_\lambda)]|
	 & \leq (1+\varepsilon)^{-1} (\max_{x} |(S_{\beta, \varepsilon}^q)'(x)|) |v^{(n)}_\lambda| \\
	 & \leq (1+\varepsilon)^{-1} (\min_{x} |(F_{\beta, \varepsilon}^q)'(x)|)^{-1} |v^{(n)}_\lambda|  \leq  (1+\varepsilon)^{-1}|v^{(n)}_\lambda|.
\end{align*}
In the first inequality, we have used the mean value theorem and in the second inequality we  have used the lower bound for $(F_{\beta, \varepsilon}^q )'$ to upper bound the derivative $(S_{\beta, \varepsilon}^q)'$ since $S_{\beta, \varepsilon}^q = (F_{\beta, \varepsilon}^q )^{-1}$. By subtracting $|v^{(n)}_\lambda|$ from the upper inequality and rewriting $\left( 1 - (1+\varepsilon)^{-1}  \right) = C' \leq 1$, we have for all $\lambda \in \Gamma_1,$ that
\begin{align}
C' |v^{(n)}_\lambda| &\leq |v^{(n)}_\lambda| -  (1+\varepsilon)^{-1}|S_{\beta, \varepsilon}^q (a_\lambda + v^{(n)}_\lambda) - S_{\beta, \varepsilon}^q (a_\lambda)| \\
&\leq |v^{(n)}_\lambda -  (1+\varepsilon)^{-1}[S_{\beta, \varepsilon}^q (a_\lambda + v^{(n)}_\lambda) - S_{\beta, \varepsilon}^q (a_\lambda)]|,
\end{align}
by the triangle inequality which implies
\[
\sum_{\lambda \in \Gamma_1} |v_\lambda^n|^2 \leq \left(\frac{1}{C'}\right)^2\sum_{\lambda \in \Gamma_1}  |v_\lambda^n -  (1+\varepsilon)^{-1}[S_{\beta, \varepsilon}^q (a_\lambda + v^{(n)}_\lambda) - S_{\beta, \varepsilon}^q (a_\lambda)] | \rightarrow 0,~n \rightarrow \infty.
\]

On the other hand, since $\Gamma_0$ is a finite set, and $(v^{(n)})$ tends to 0 weakly as $n$ tends to $\infty,$ we also obtain
\[
	\sum_{\lambda \in \Lambda} |v_\lambda^{(n)}|^2 \rightarrow 0 \mbox{ as } n \rightarrow \infty.
\]
\end{proof}

\begin{theorem}
	The algorithm (\ref{algorithm_surrogate_func}) produces sequences $(u^{(n)})$ and $(v^{(n)})$ in $\ell_2(\Lambda)$ that converge strongly to the vectors $u^*,v^*$ respectively. In particular, the sets of strong accumulation points are non-empty.
\end{theorem}

\begin{proof}
	Let $u^*$ and $v^*$ be weak accumulation points and let $(u^{(n_j)})$ and $(v^{(n_j)})$  be subsequences weakly convergent to $u^*$ and $v^*$ respectively. Let us denote the latter subsequences $(u^{(n)})$ and $(v^{(n)})$ again. 

If $\mu^{n+1/2}$ is such that $\| \mu^{n+1/2}\|_2 \rightarrow 0,$ then the statement of the theorem follows from  Lemma \ref{lemma1_general_surrigate}.
If, instead, there exists a subsequence, denoted by the same index, that $\| \mu^{n+1/2} \|_2 \geq \gamma$, then by Lemma \ref{lemma_str_conv_0} we get that $\| T \mu^{n+1/2} \|_{\mathcal H} \rightarrow 0.$     Subsequently applying Lemma \ref{lemma_str_conv_0} and Lemma \ref{lemma:embarace}, we get $\| \mu^{n+1/2} \|_2 = 0,$ which yields a contradiction to the assumption. 
Thus, by  Lemma \ref{lemma1_general_surrigate} we have that $(v^{(n)})$ converges to $v^{*}$ strongly. The strong convergence of $(u^{(n)})$ is already guaranteed by Theorem  \ref{theorem:weak_conv} .
\end{proof}


\section{Numerical realization and testing}

In this section, we continue the discussion started in the introduction on the geometry of the solution sets for fixed $p,q$  and regularization parameters chosen from the prescribed grids. 
However, we not only extend our preliminary geometrical observations on the sets of computed solutions  to the high-dimensional case, but also provide an a priori goal-oriented parameter choice rule. 

We also compare results for multi-penalty regularization with $0\leq p\leq1,~2 \leq q\leq \infty$ and the corresponding one-penalty regularization schemes. This specific comparison has been motivated by encouraging results obtained in  \cite{fonapeXX, NP13} for the Hilbert space setting, where the authors have shown the superiority and robustness of the multi-penalty regularization scheme compared to the \enquote{classical} one-parameter regularization methods.

\subsection{Problem formulation and experiment data set}
In our numerical experiments we consider the model problem of the type
\[
	y= T (u^\dag + v^\dag),
\]
where $T \in \RR^{m \times N}$ is an i.i.d Gaussian matrix, $u^\dag$ is a sparse vector and $v^\dag$ is a noise vector. The choice of $T$ corresponds to compressed sensing measurements \cite{FoRa13}.
In the experiments, we consider 20 problems of this type with $u^\dag$ randomly generated with values on $[-3,3]$ and $\#\supp(u^\dag) = 7,$ and $v^\dag$ is a random vector whose components are uniformly distributed on $[-1,1]$, and normalised such that $\|v^\dag \|_2 =0.7,$ corresponding to  a signal to noise ratio of ca. 10 \%. In our numerical experiments, we are keeping such a noise level fixed.

In order to create an experimental data set, we were considering for each of the problems the minimization of the functional \eqref{eq:funct_general} for $p\in\{0,0.3,0.5,0.8,1\}$ and $q\in\{2,4,10,\infty\}$. The regularization parameters $\alpha$ and $\beta$ were chosen from the grid $Q_{\alpha_0}^k\times Q_{\beta_0}^k$, where $Q_{\alpha_0}^k:=\{\alpha= \alpha_i = \alpha_0 k^i\, ,\alpha_0 = 0.0009,k=1.25,i=0,\ldots,30\}$, and $Q_{\beta_0}^k:=\{\beta = \beta_i = \beta_0 k^i,\beta_0 = 0.0005,k=1.25,i=0,\ldots,30\}$. For all possible combinations of $p$ and $q$ and $(\alpha,\beta)$   we run algorithm \eqref{algorithm_thresholding_fun} with number  $L=M=20$ of inner loop iterations and starting values $u^{(0)}=v^{(0)}=0$. 
Furthermore, we set $\varepsilon=0$ since the additional term $\varepsilon \| v\|_2^2$ is necessary for coercivity  only in the infinite-dimensional setting. Due to the fact that the thresholding functions for $p\in\{0.3,0.8\}$ are not given explicitly, we, at first, precomputed them on a grid of points in $[0,5]$ and interpolated in between, taking also into consideration the jump discontinuity. Respectively, we did the same precomputations for $q\in\{4,10\}$ on a grid of points in $[0,1]$.

\subsection{Clustering of solutions}
As we have seen in the introduction of this paper, the 2D experiments revealed certain regions of computed solutions for $u^{\dagger}$ and $v^{\dagger}$ with very particular shapes, depending on the parameters $p$ and $q$. We question if similar clustering of the solutions can also be found for problems in high dimension. To this end, the challenge is the proper geometrical representation  of the computed high dimensional solutions, which can preserve the geometrical structure in terms of mutual distances. We consider the set of the computed solutions  for fixed $p$ and $q$ in the grid $Q_{\alpha_0}^k\times Q_{\beta_0}^k$ as  point clouds which we investigate independently with respect to the components $u^\dagger$ and $v^\dagger$ respectively. As the solutions are depending on the two scalar parameters $(\alpha, \beta),$ it is legitimate to assume that they form a 2-dimensional manifold embedded in the higher-dimensional space. Therefore, we expect to be able 
to visualize the point clouds and analyze their clustering  by employing suitable dimensionality reduction techniques. 
A broad and nearly complete overview although not extended in its details on existing dimensionality reduction techniques as well as a MATLAB toolbox is provided in~\cite{maaten1,maaten2,maatenSW}. \\

For our purposes, we have chosen the Principal Component Analysis (PCA) technique because we want to verify that calculated minimizers $u^*$ and $v^*$ form clusters around the original solutions.
In the rest of the subsection, we only consider one fixed problem from the previously generated data set. In the following figures, we report the estimated regions of the solutions $u^*$ and $v^*$, as well as the corresponding regularization parameters chosen from the grids $Q_{\alpha_0}^k\times Q_{\beta_0}^k$. We only present \emph{feasible} solutions, i.e., the ones that satisfy the discrepancy condition
\begin{equation}
\label{feaible_sol_carrot}
\#\supp(u^*)\leq\#\supp(u^{\dagger}) \text{, and } \|T(u^*+ v^*) - y\|_2 <0.1.
\end{equation}
 In Figure \ref{fig:PCA05} we consider the cases $p=0.5$ and $q \in \{ 2,4 \}$, and in Figure  \ref{fig:PCA03}  the corresponding results for  $p=0.3$ and $q \in \{ 2,4 \}$ are displayed. \\%

First of all, we observe that the set of solutions $u^*$ forms certain structures, visible here as one-dimensional manifolds, as we also observed in the 2D experiments of the introduction. Likewise, the set of solutions $v^*$ are more unstructured, but still clustered. The effect of modifying $q$ from 2 to 4 increases the number of feasible solutions according to  (\ref{feaible_sol_carrot}). Concerning the parameter $p$, by modifying it from $0.5$ to $0.3$, the range of $\alpha'$s which  provide feasible solutions is growing. Since it is still hard from this geometrical analysis on a single problem to extract any qualitative information concerning the accuracy of the reconstruction, we defer the discussion on multiple problems to the following subsection. 

\begin{figure}[hpt]
\includegraphics[width=0.32\textwidth]{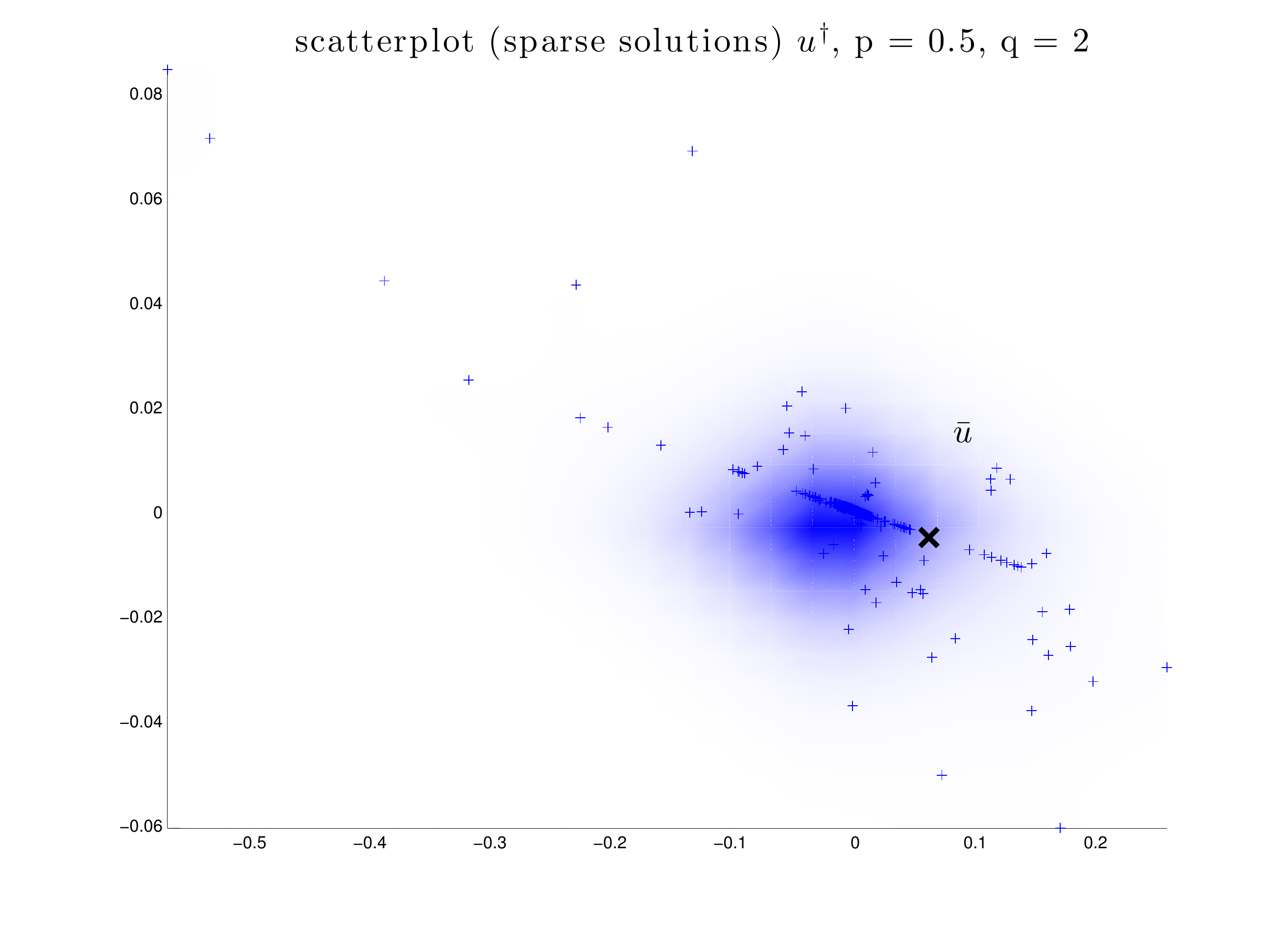}
\includegraphics[width=0.32\textwidth]{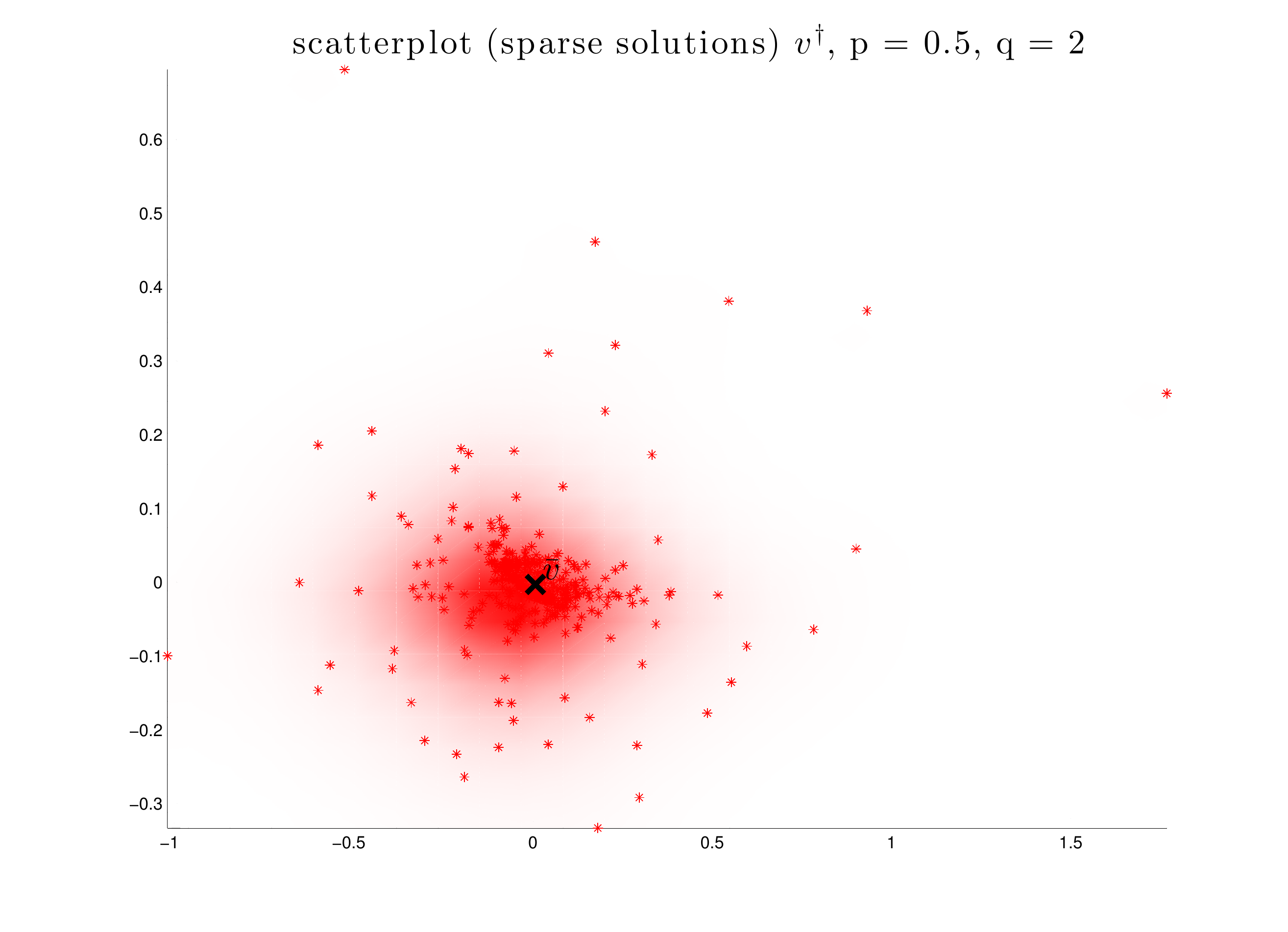}
\includegraphics[width=0.32\textwidth]{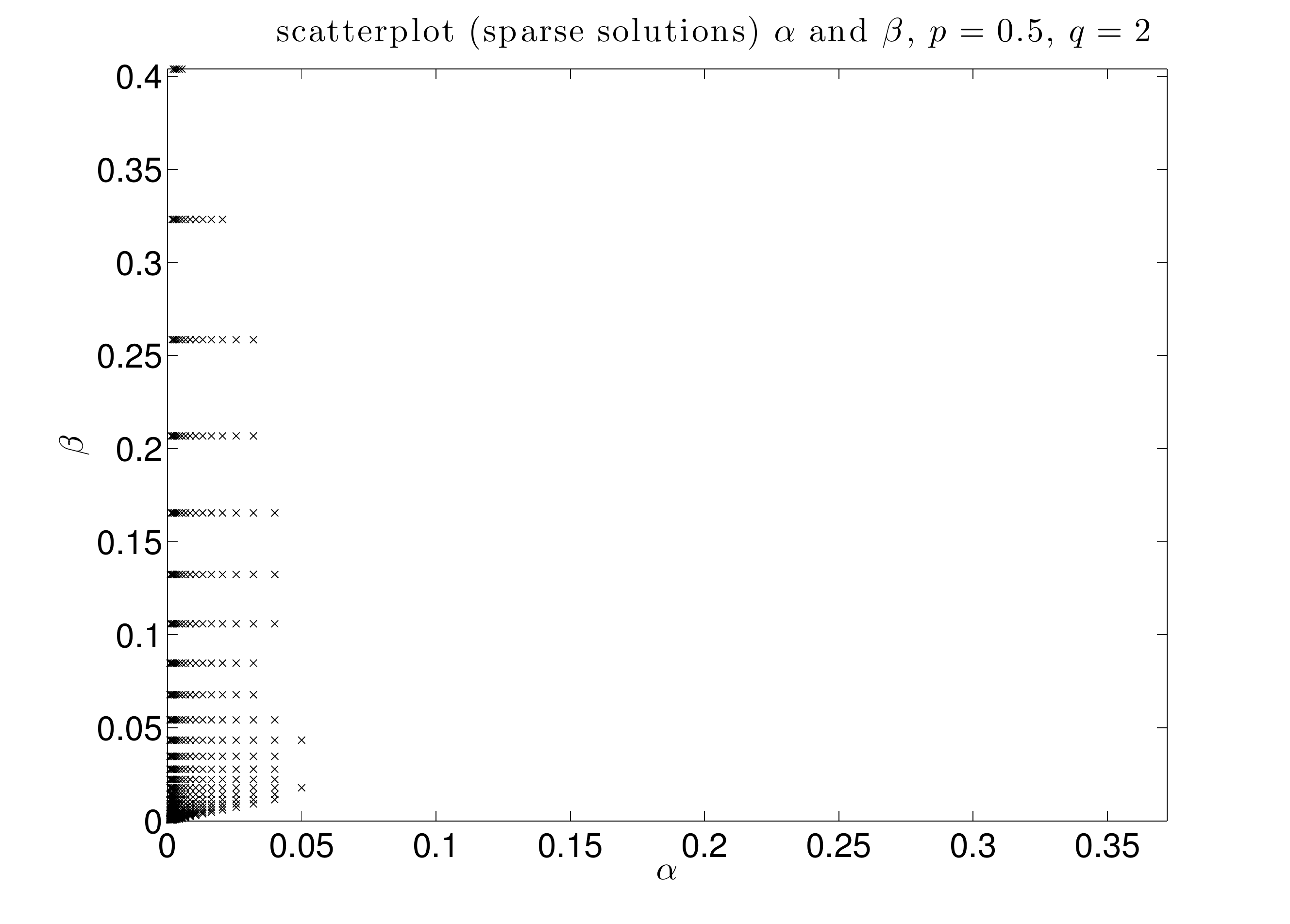}
\includegraphics[width=0.32\textwidth]{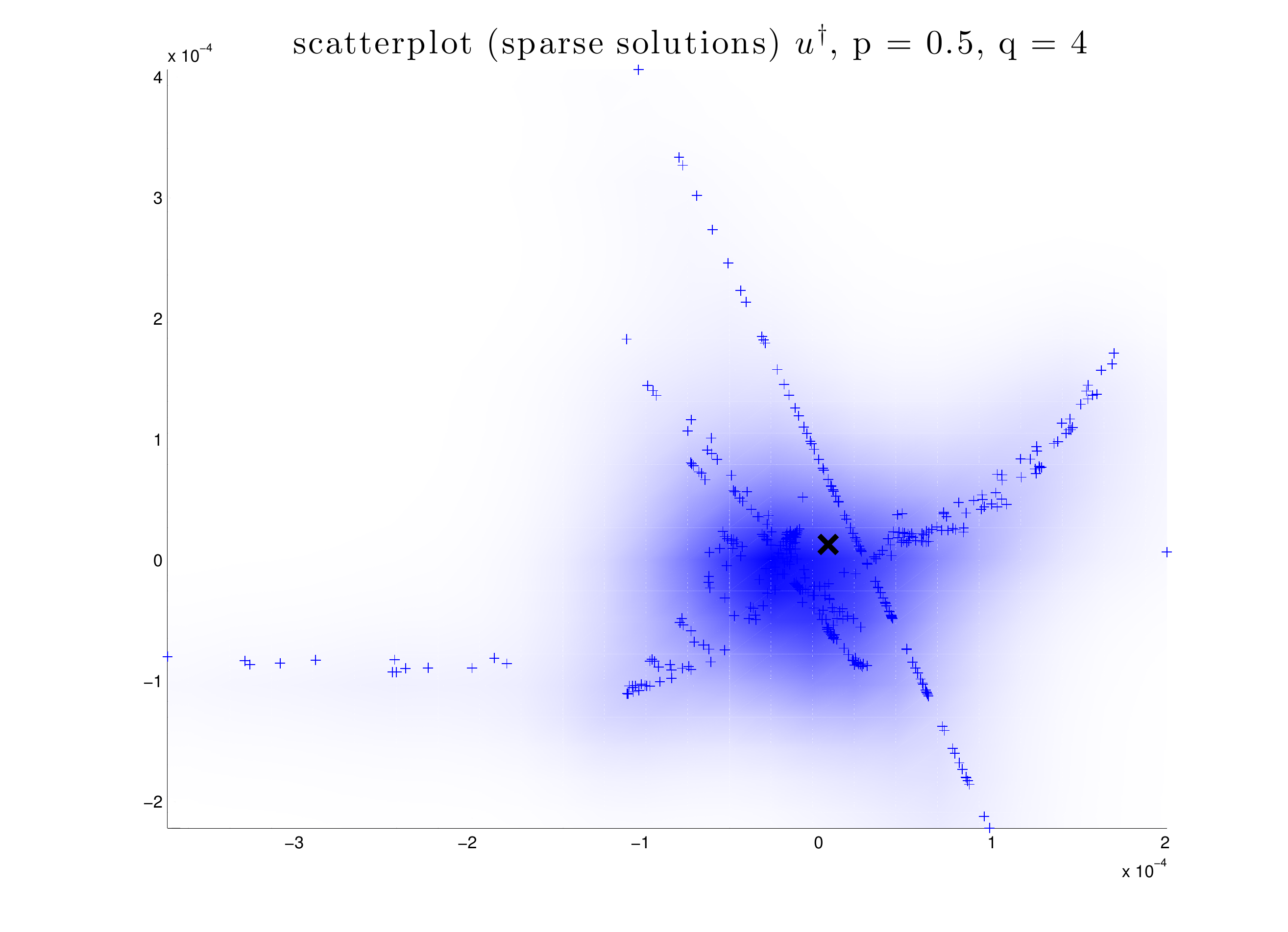}
\includegraphics[width=0.32\textwidth]{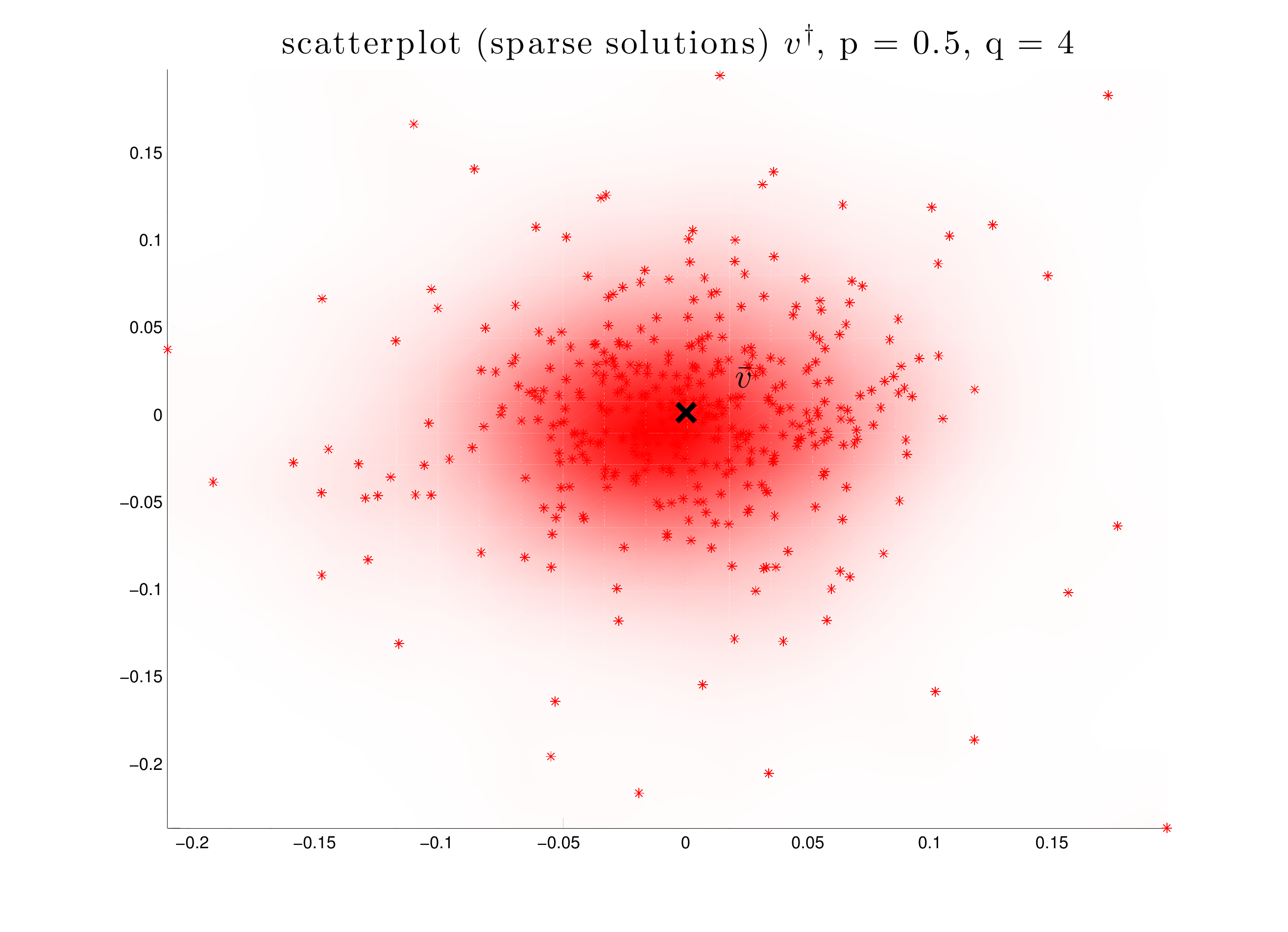}
\includegraphics[width=0.32\textwidth]{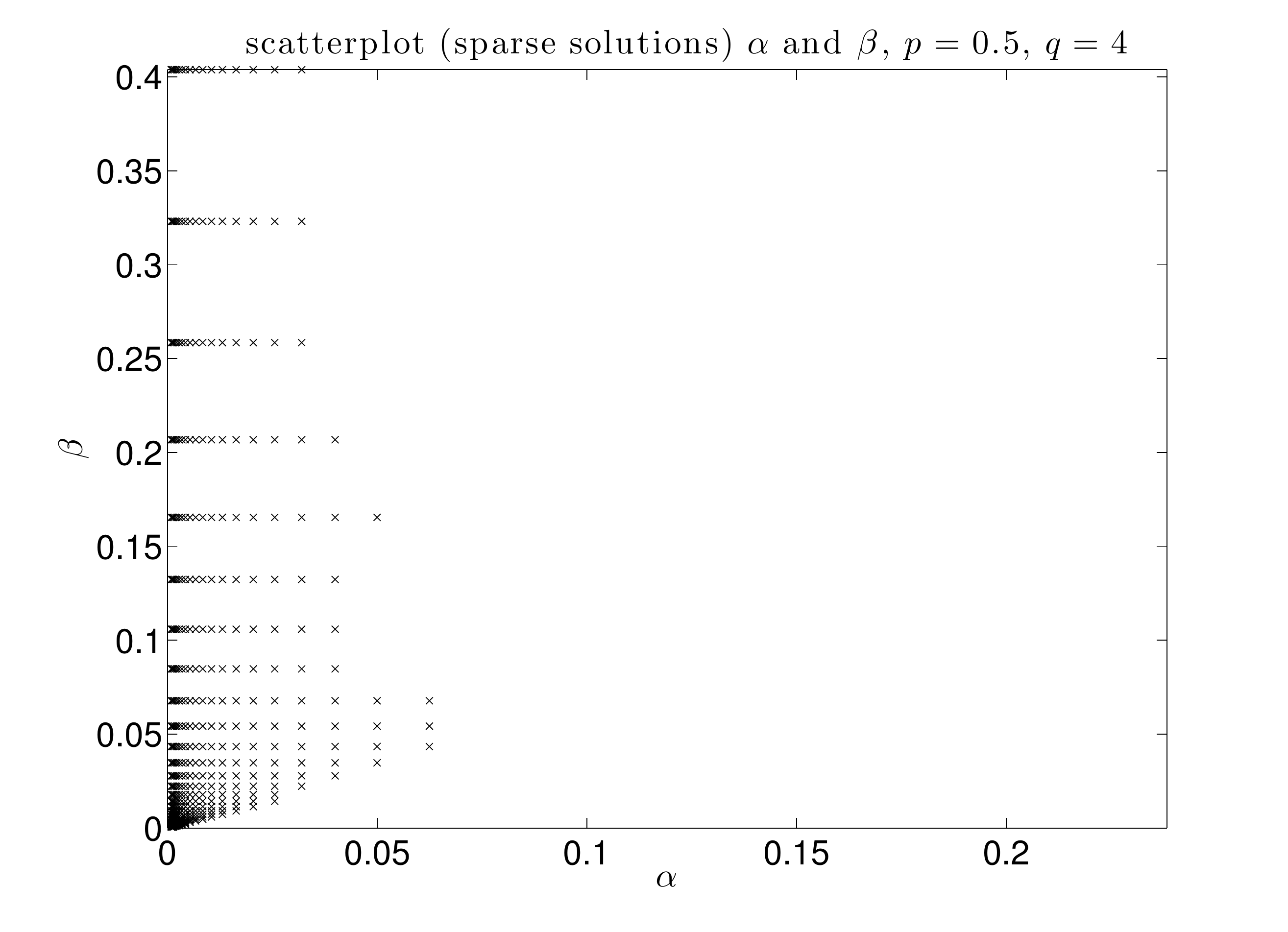}
\caption{Estimated regions of the regularization parameters (right panel) and the corresponding solution $u^*$ (left panel) and $v^*$ (middle panel) for $p=0.5$, and $q=2$ (top), and $q=4$ (bottom) repectively using PCA. The black crosses indicate the real solutions.} \label{fig:PCA05}
\end{figure}

\begin{figure}[hpt]
\includegraphics[width=0.32\textwidth]{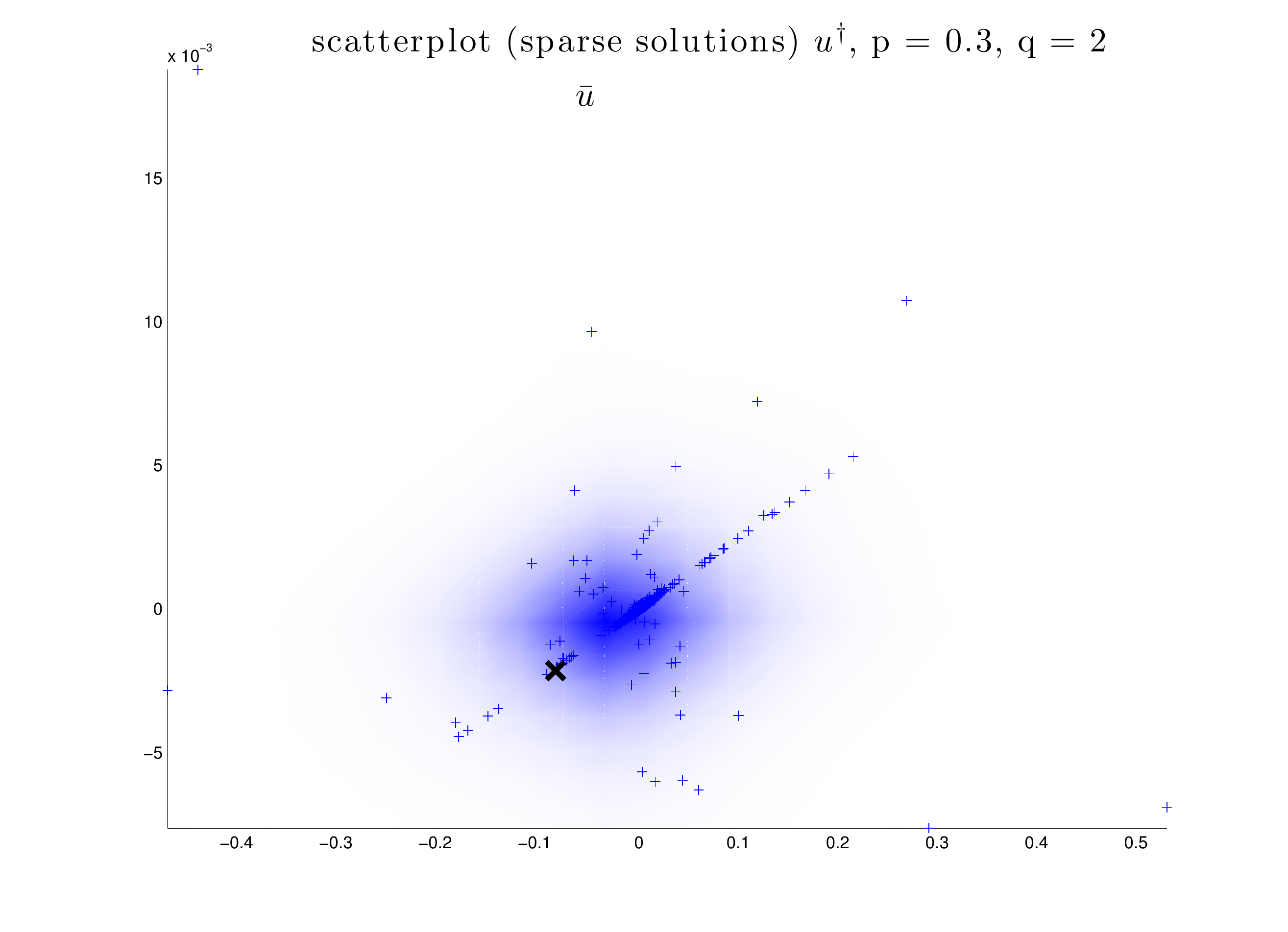}
\includegraphics[width=0.32\textwidth]{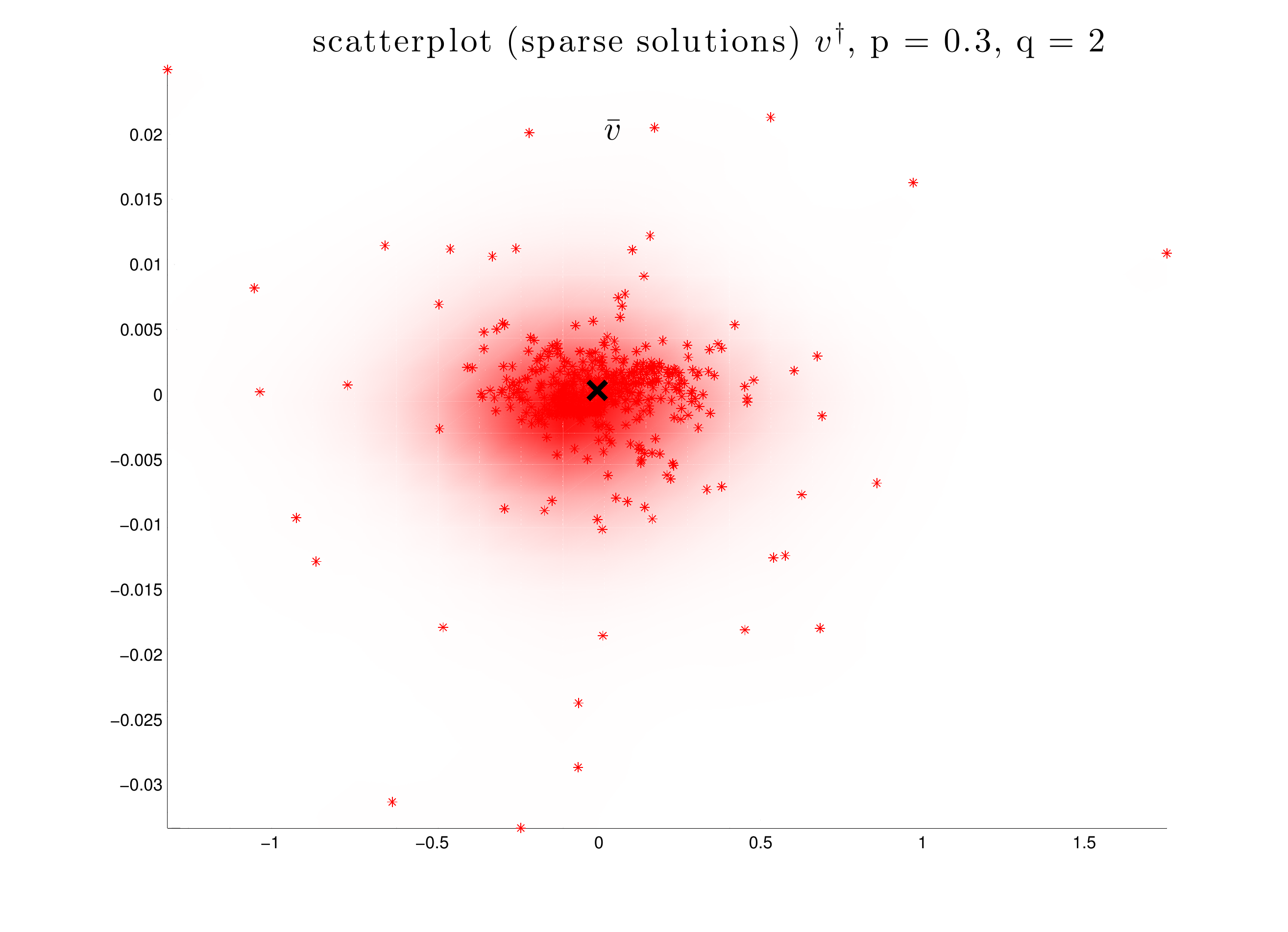}
\includegraphics[width=0.32\textwidth]{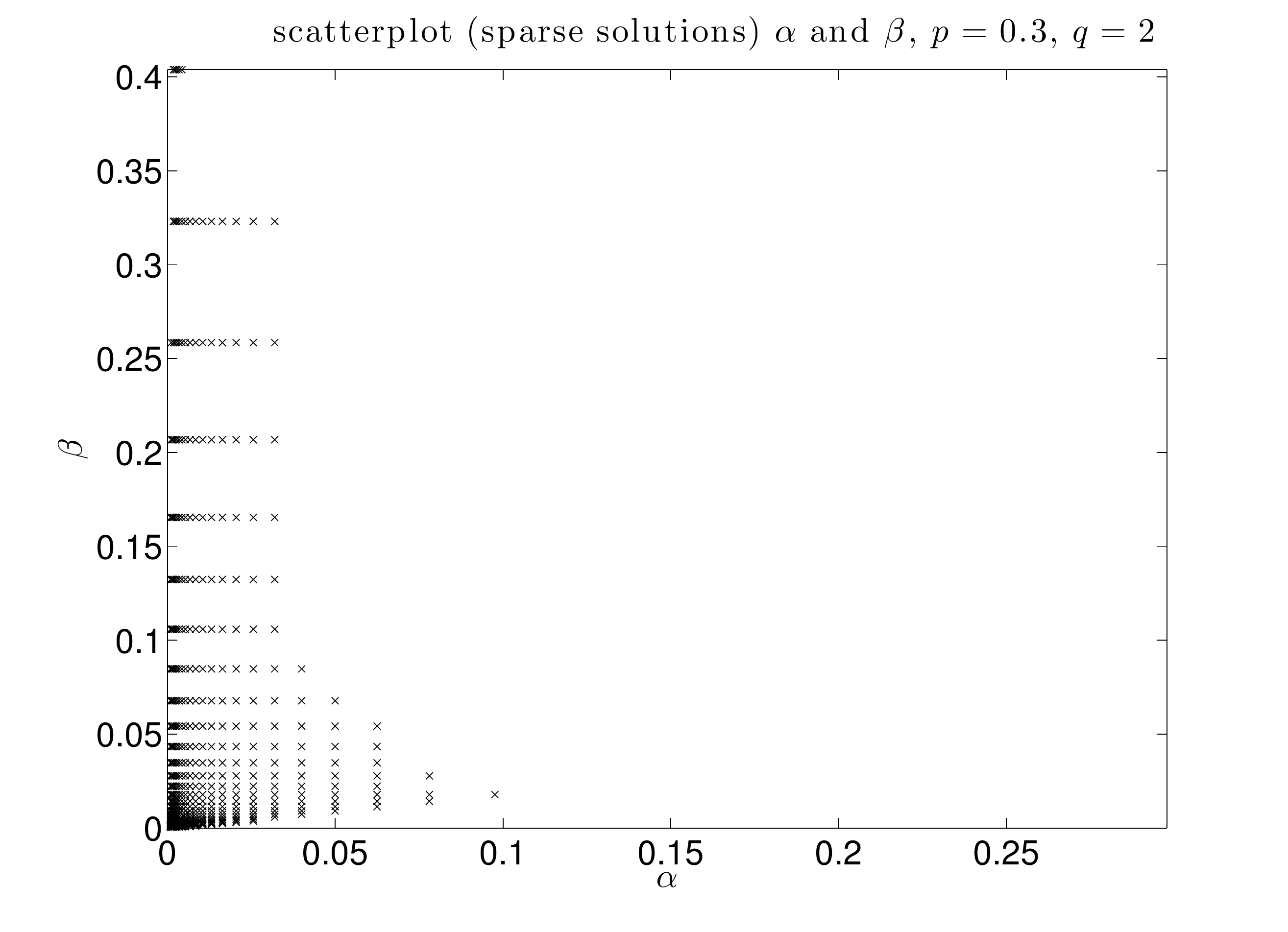}
\includegraphics[width=0.32\textwidth]{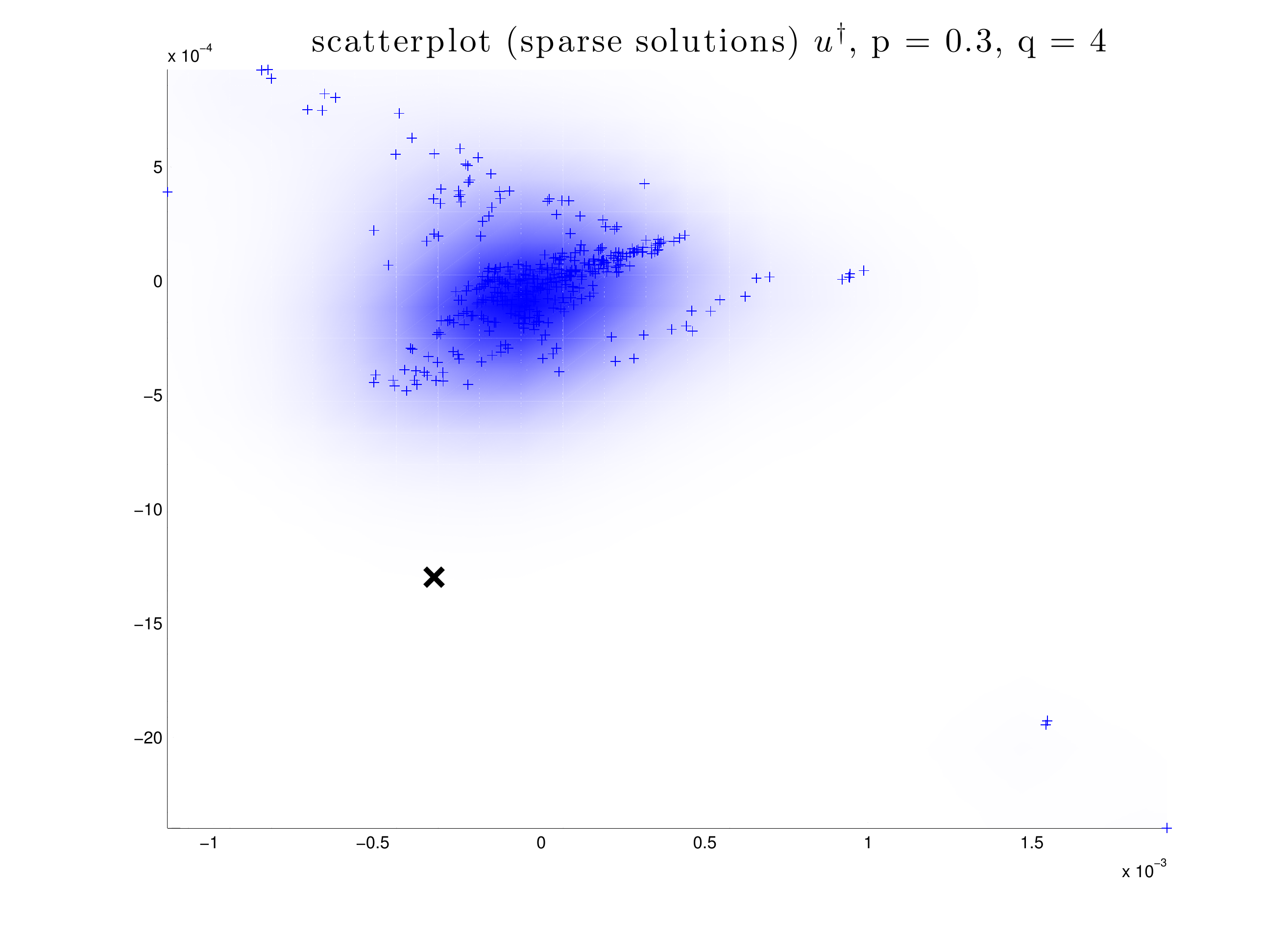}
\includegraphics[width=0.32\textwidth]{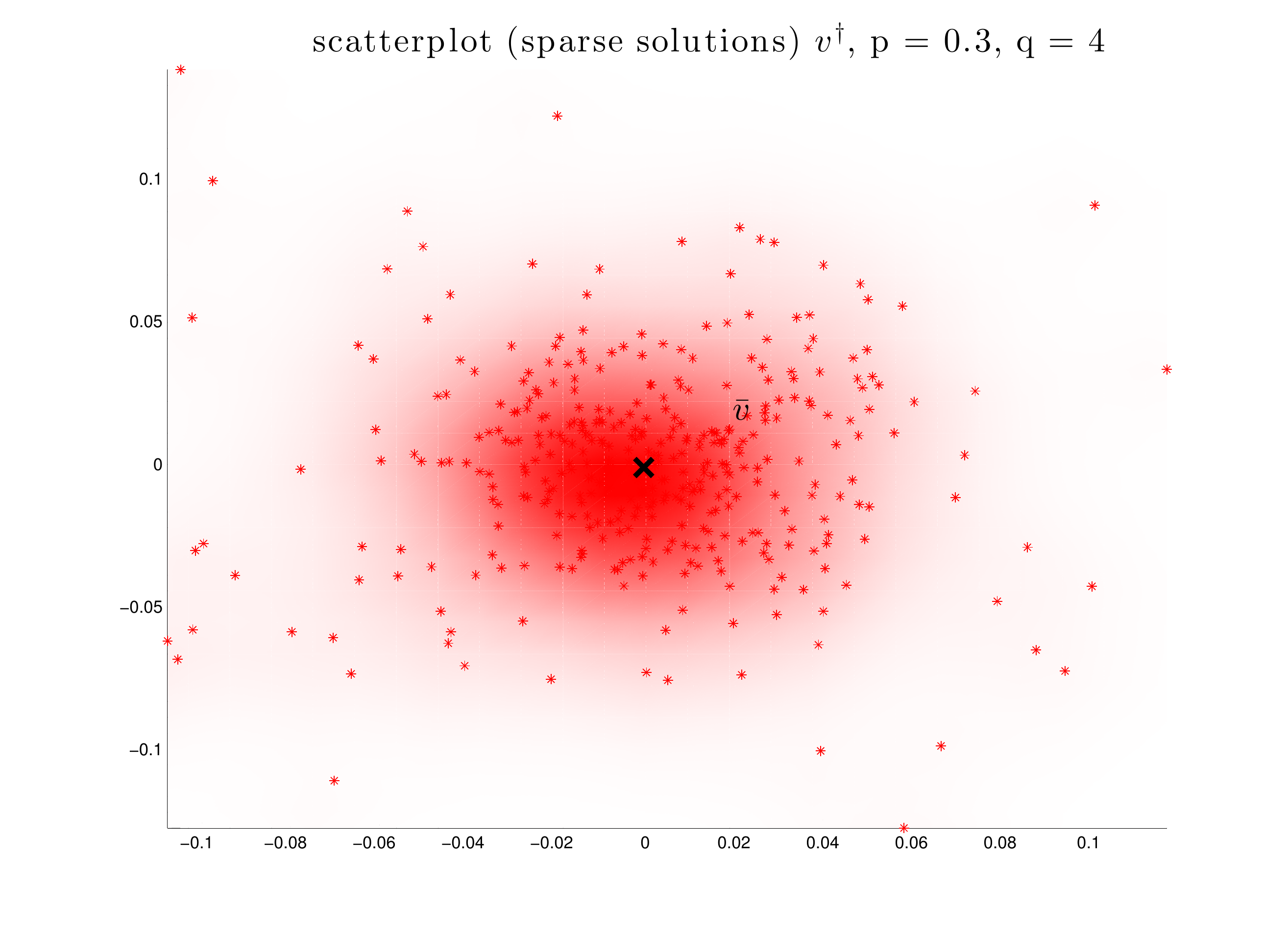}
\includegraphics[width=0.32\textwidth]{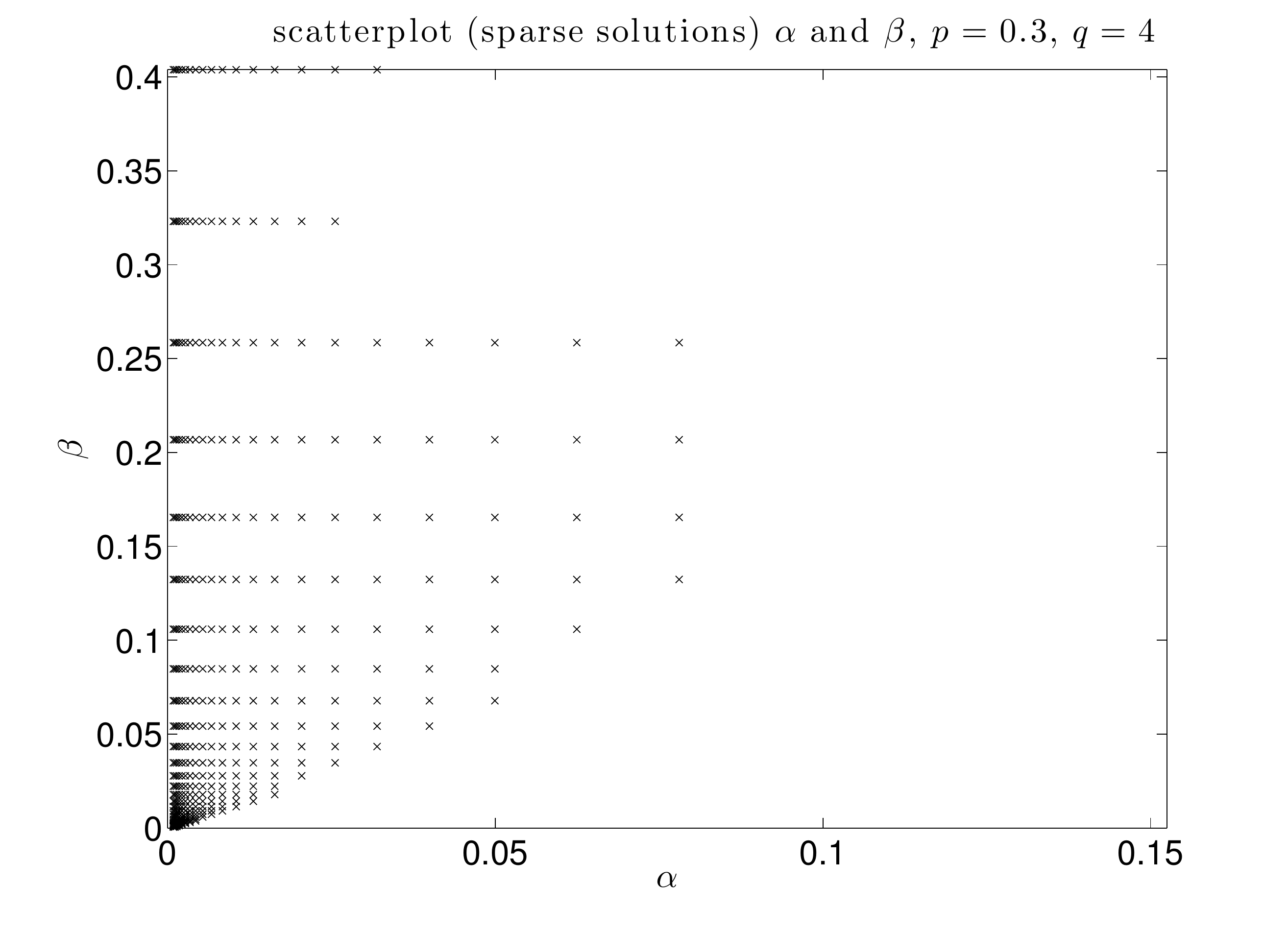}
\caption{Estimated regions of the regularization parameters (right panel) and the corresponding solution $u^*$ (left panel) and $v^*$ (middle panel) for $p=0.3$, and $q=2$ (top), and $q=4$ (bottom) repectively using PCA. The black crosses indicate the real solutions.} \label{fig:PCA03}
\end{figure}

\subsection{Comparison with the one-parameter counterpart}
Motivated by some positive results  \cite{NP13, fonapeXX} showing the superiority of multi-penalty regularization against classical single-parameter regularization schemes in Hilbert spaces, in this subsection we compare the performance of multi-penalty regularization and its one-parameter counterpart, to which we further refer as \emph{mono-penalty minimization}.

It is now well-established that a sparse solution can be reconstructed by minimizing the functional of the type
\begin{equation}
\label{eq:one_param}
	J_p (u) : = \|T u - y\|^2_{2} + \alpha \| u \|_{p}^p,
\end{equation}
with $p \in [0,1].$ A local minimizer $u_{\alpha,p}^*$ of this functional can be computed by the iterations
\[
	u^{(n+1)}_\lambda =  H_\alpha^p (u^{(n)}_\lambda + [T^* (y-Tu^{(n)}) ]_\lambda), \quad n\geq 0,
\] 
where $H_\alpha^p$ is the thresholding operator, defined as in Proposition 1. \\

In order to assess the obtained results, we compare the performance of the considered regularization schemes. We measure the approximation error (AE) by $\| u - u^\dagger \|_{2}$, as well as the number of elements in the symmetric difference (SD) by $\#(\supp(u) \Delta \supp(u^\dagger))$. The SD is defined as follows: $\lambda \in \supp(u) \Delta \supp(u^\dagger)$ if and only if  either $\lambda \notin \supp(u)$ and $\lambda\in\supp(u^\dagger)$ or $\lambda\in\supp(u)$ and $\lambda\notin\supp(u^\dagger)$. \\

For each problem and each $(p,q)-$combination we compute the best multi-penalty solution $u^* = u^*(\alpha,\beta),$ meaning that no  other pairs $(\alpha,\beta)\in Q_{\alpha_0}^k\times Q_{\beta_0}^k$ can improve the accuracy of the algorithm. 
Simultaneously, for each value of $p\in \{0, 0.3, 0.5, 0.8, 1\}$ and each problem from our data set we compute the best mono-penalty solution.
Then,  for each pair of $p$ and $q$ (multi-penalty) and for each $p$ (mono-penalty), we compute the mean value of the AE and SD. The respective results are shown in Figure~\ref{fig:accuracy_support} on the left. We observe from these results, that
\begin{itemize}
\item comparing only the multi-penalty schemes, the choice of $q=2$ allows to achieve the smallest AE, in general, followed by $q=\infty$ and $q=4$; in SD the best choice is also $q=2$, but followed first by $q=4$ and then $q=\infty$;
\item AE for the mono-penalty approach is always larger than for the best multi-penalty scheme, and this negative gap is even more significant in SD.
\end{itemize}
Since mean value statistics can be corrupted by single large outliers, we provide an additional more individual comparison of multi-penalty and mono-penalty minimization on the right panel of Figure~\ref{fig:accuracy_support}. We compare problem-wise the values of AE and SD for each multi-penalty solution with $p$ and $q$ and the respective mono-penalty solution. The colored bar shows the relative number of problems for which multi-penalty minimization performed better or equal than mono-penalty minimization. The results confirm the above stated interpretation of the panel on the left. In particular multi-penalty minimization appears to be particularly superior in terms of support recovery.

\begin{figure}[hpt]
\includegraphics[width=0.48\textwidth]{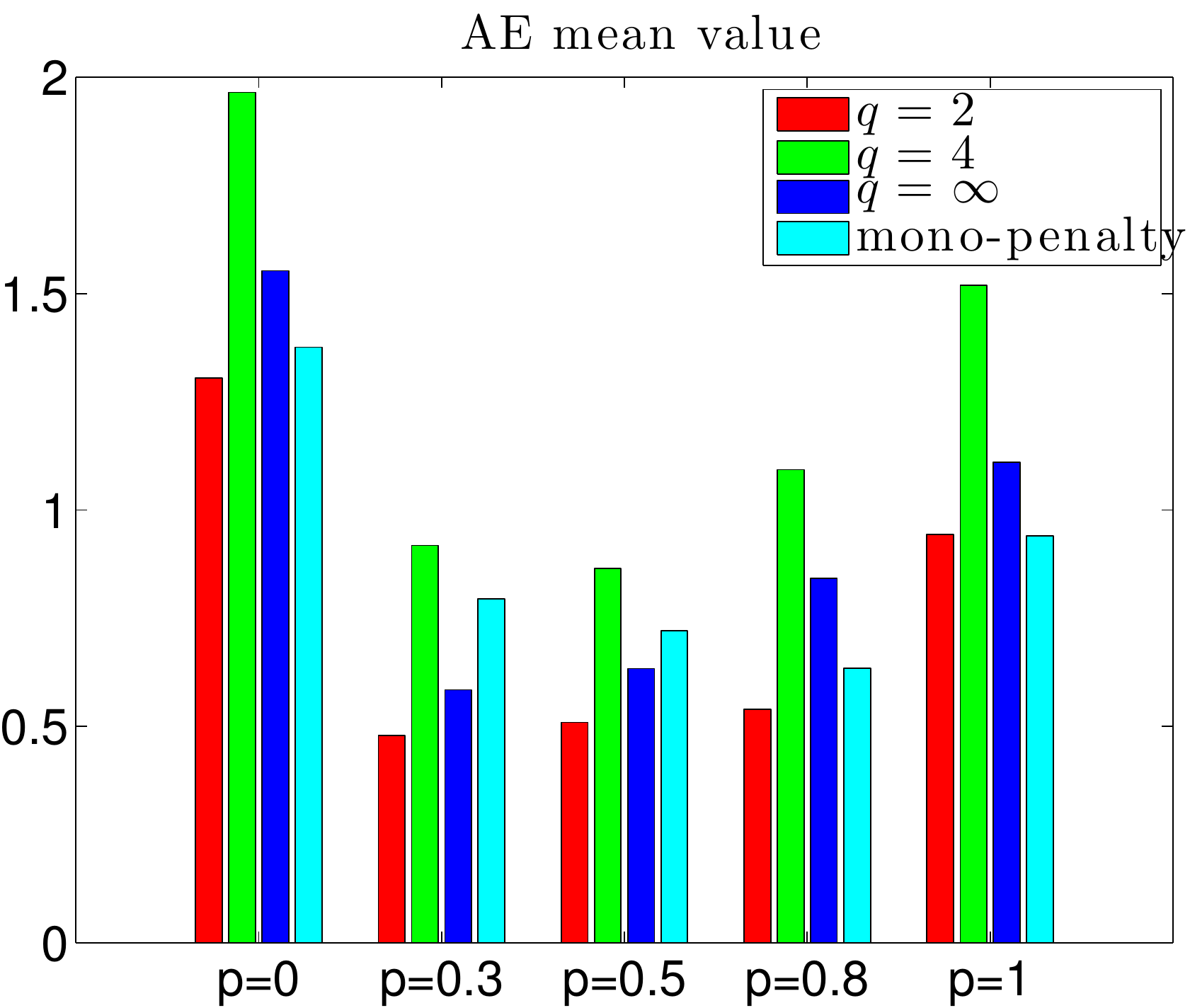}
\includegraphics[width=0.48\textwidth]{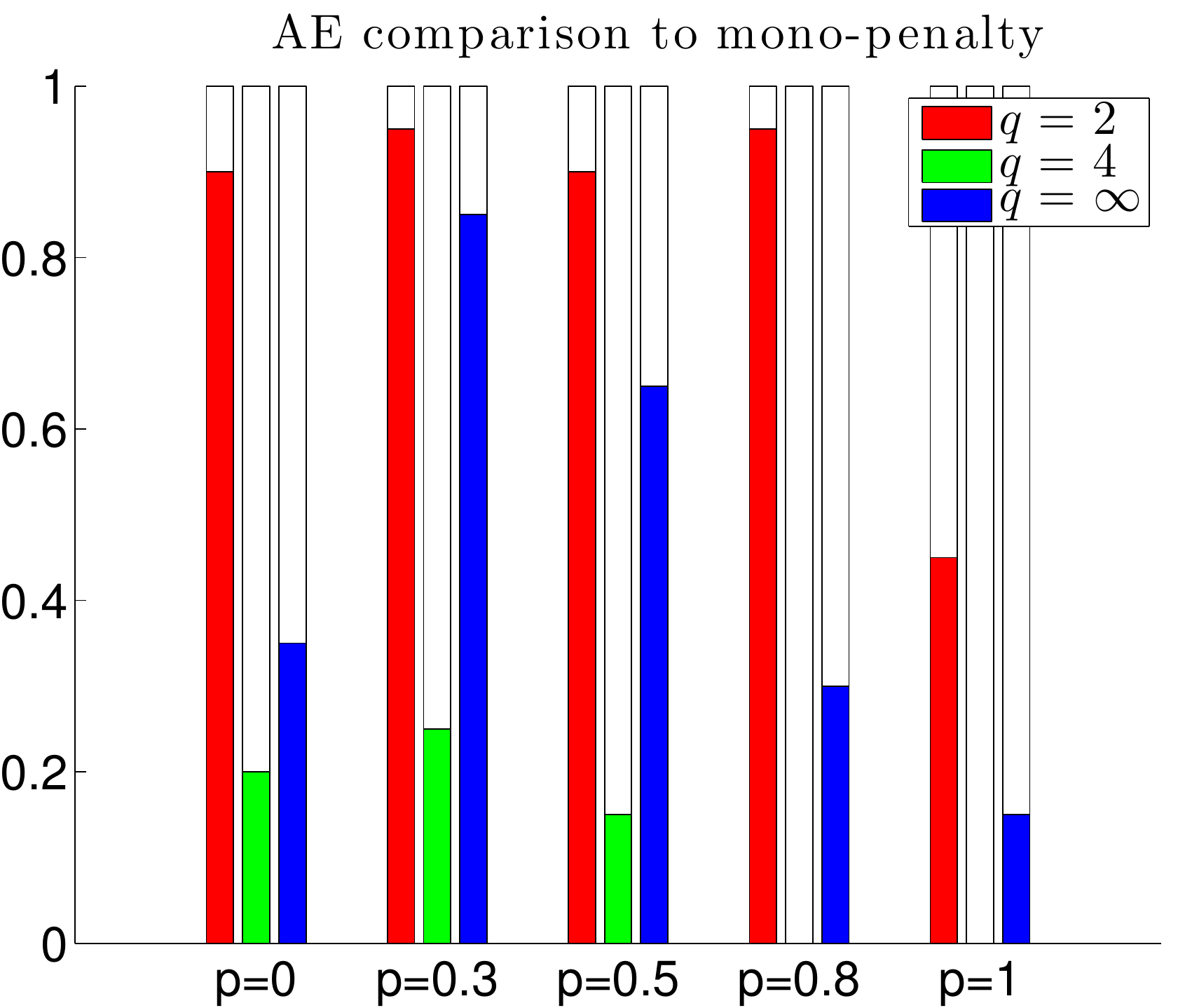}
\includegraphics[width=0.48\textwidth]{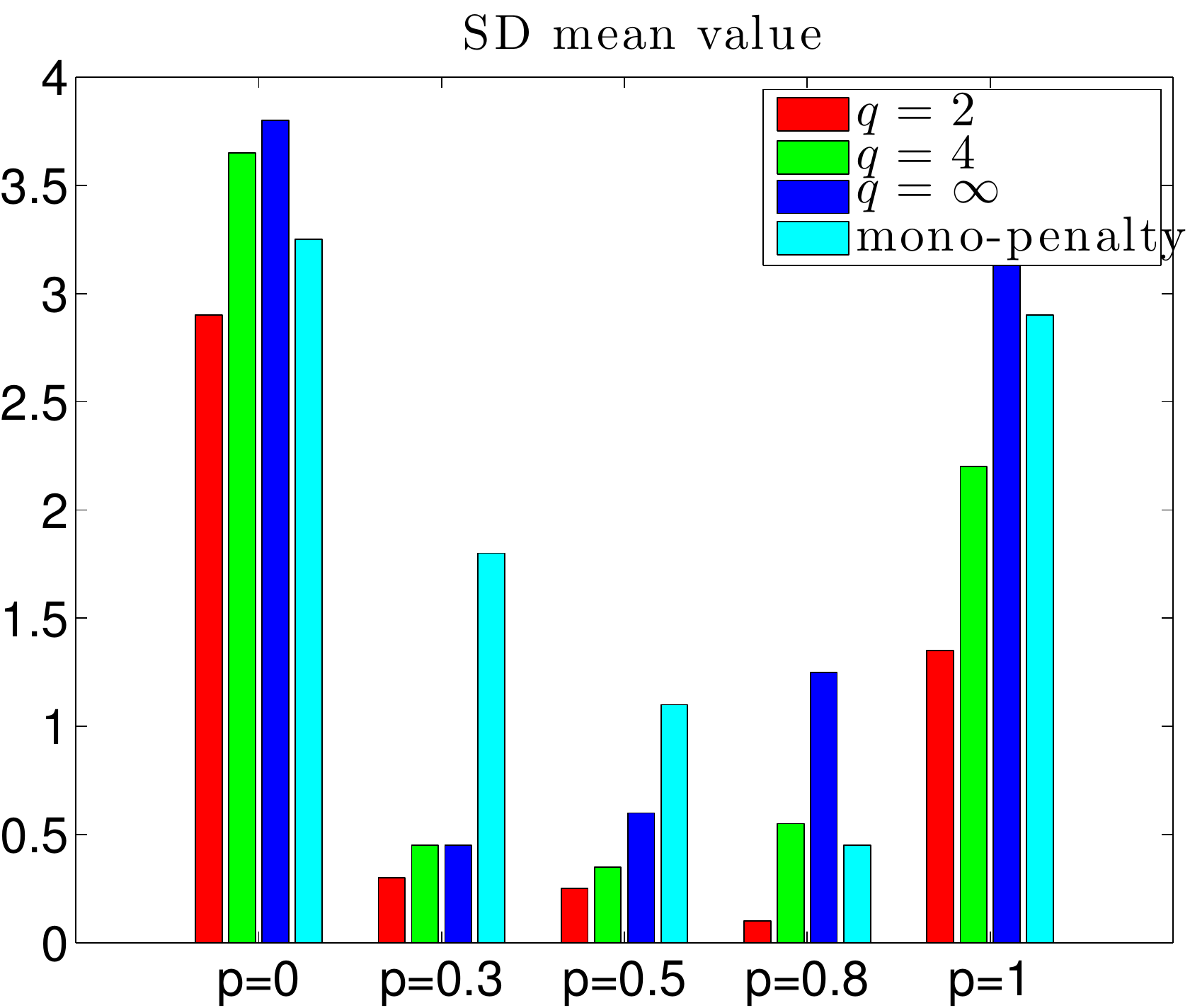}
\includegraphics[width=0.48\textwidth]{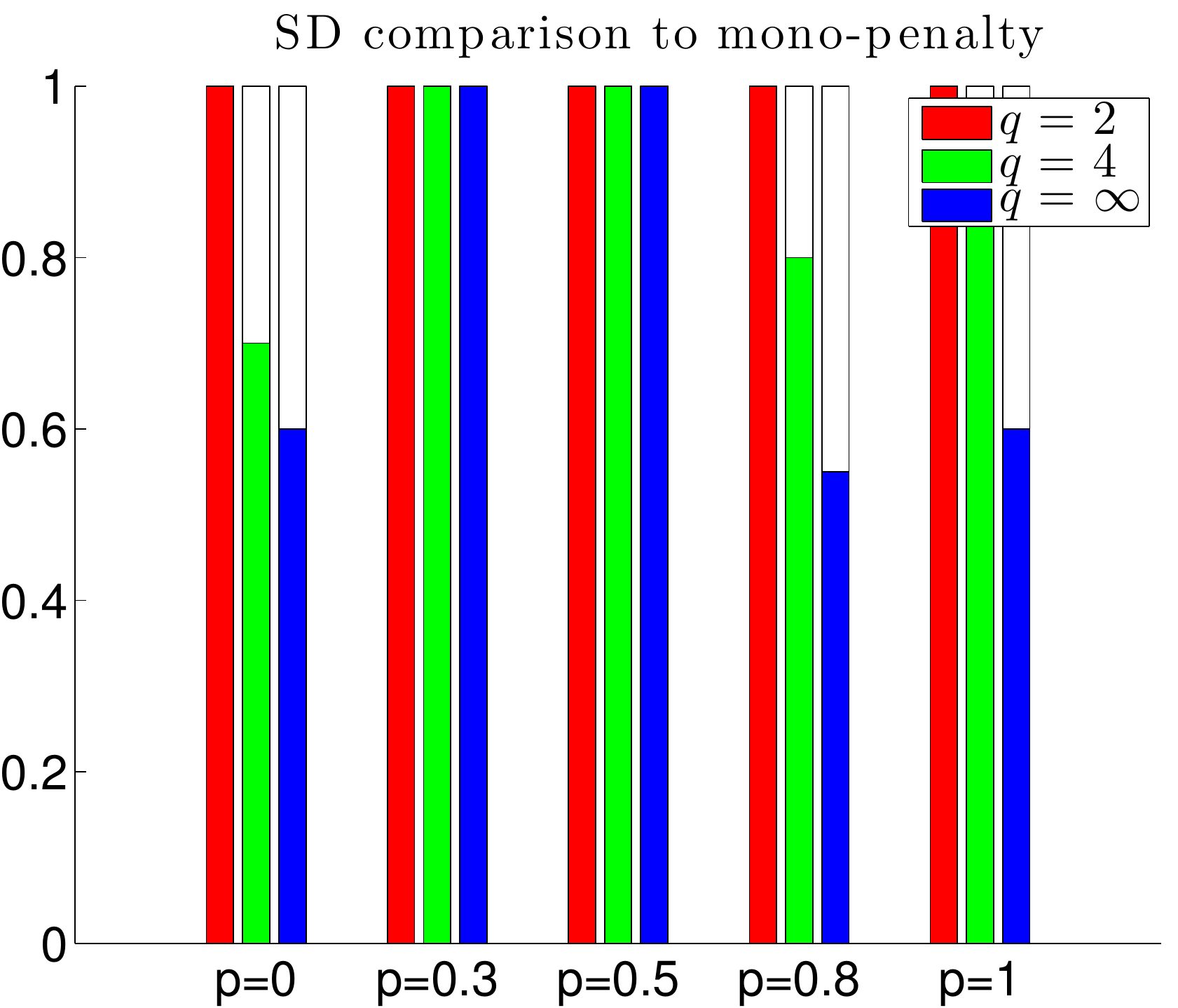}
\caption{The left panel presents for each $p\in\{0,0.3,0.5,0.8,1\}$ the mean of the AE (top) as well as the SD (bottom) for the solution $u^*$ for 20 problems for different parameter values $q\in\{2,4,\infty\}$ as well as for the mono-penalty solution $u_{\alpha,p}$. For each of the 20 problems and each pair $(p,q)$, the best individual parameter pair $(\alpha,\beta)$ was chosen for comparison. On the right panel a coloured bar indicates the empirical probability of better performance by the multi-penalty approach in terms of AE (top) and SD (bottom) with respect to the mono-penalty approach.} \label{fig:accuracy_support}
\end{figure}

\subsection{Goal-oriented parameter choice}

In the previous sections, we provided a very detailed description of the iterative thresholding algorithms for multi-penalty regularization with non-convex and non-smooth terms. However, till now we have not touched upon the topic on the parameters choice for the concrete implementation of the presented algorithm. At the same time,  from the extensive numerical testing presented above, we have observed that, firstly, some $(p,q)-$combinations allow the best performance according to the prescribed goal, and, secondly, due to the clustering of the solutions the regularization parameters $(\alpha, \beta)$ can be chosen a priori for any fixed $p,q$ (see again Figures \ref{fig:PCA05} and \ref{fig:PCA03}). 

As it can be observed in the quality measures of  Figure \ref{fig:accuracy_support}  the performance of the method depends on $p$ and $q$: in fact, one can conclude that the  minimization of (\ref{eq:funct_general}) by means of algorithm (\ref{algorithm_thresholding_fun}) performs  unsatisfactory for the \enquote{classical} choices $p=0$ and $p=1$ with respect to both AE and SD. Our most relevant regularization is that  multi-penalty regularization with $p \in (0,1)$ and $q=2$ performs best with respect to both the reconstruction accuracy and the support reconstruction. The support reconstruction for $p \in (0,1)$ and $q=4$ is not significantly worse than for $q=2.$ However, one can observe a very poor reconstruction accuracy for $q=4$ and this is even more relevant when $p=0$ or $p=1.$
Multi-penalty regularization with $p\in (0,1)$ and $q= \infty$ performs reasonably but  worse  in terms of accuracy and support reconstruction than $q=2$.

Very surprisingly for a fixed pair $(p,q)$ there exist a large set of regularization parameters  $(\alpha, \beta)$ which perform equally good in terms of AE and SD. In Figures \ref{fig:parameters0} and \ref{fig:parameters} we present for each of the 20 problems from our data set the twenty best pairs of regularization parameters $(\alpha, \beta)$ that allows for the best AE and SD with $p=0.3$ and $q \in \{2,4, \infty\}.$ The visual analysis presented in the figures show the geometrical properties of the sets of the best parameters, namely they lay within a cone, whose width is governed by $p,q$. Clearly, these observations furnish a guideline for choosing in a stable way effective regularization parameters.

\begin{figure}[hpt]
\includegraphics[width=0.32\textwidth]{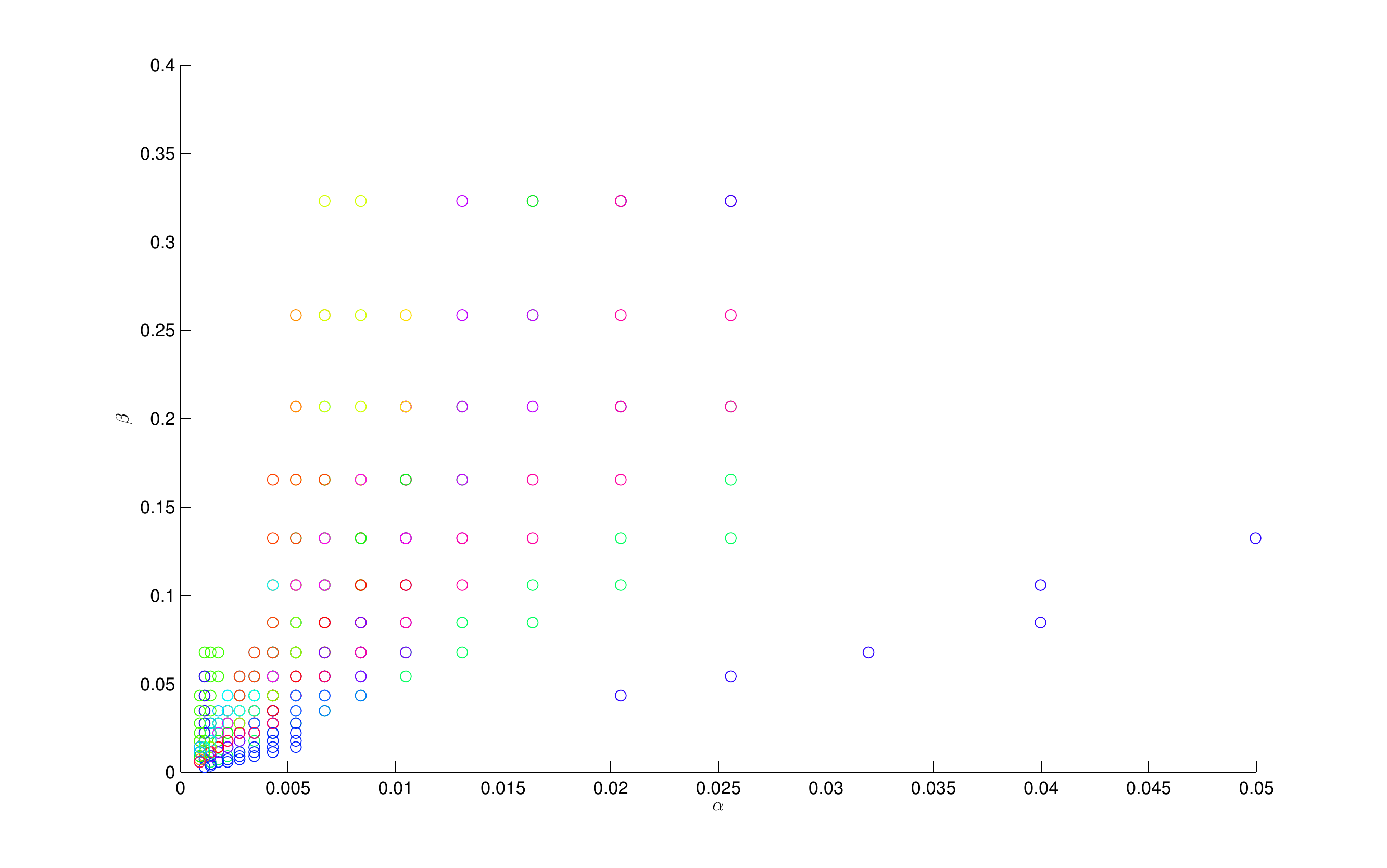}
\includegraphics[width=0.32\textwidth]{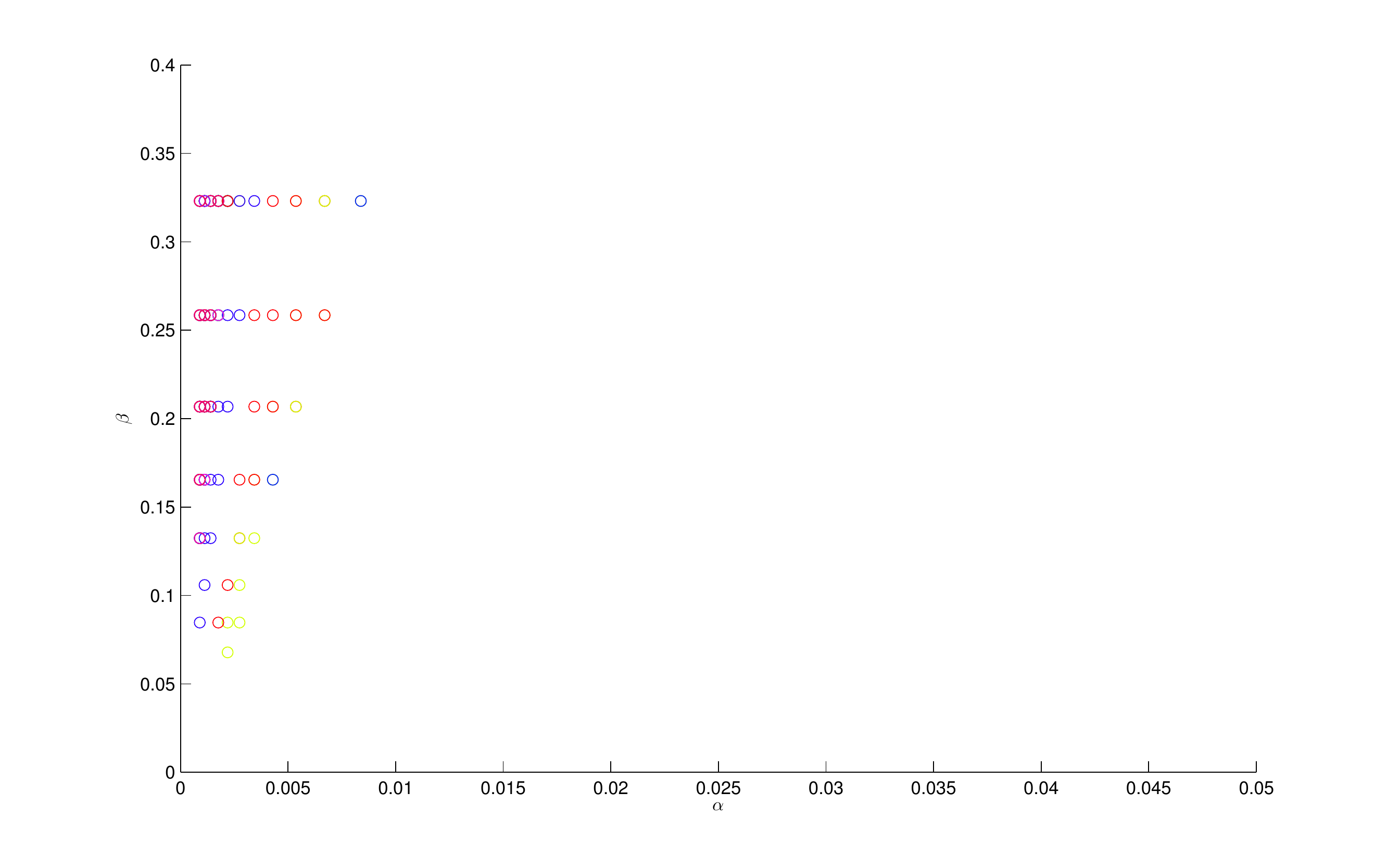}
\includegraphics[width=0.32\textwidth]{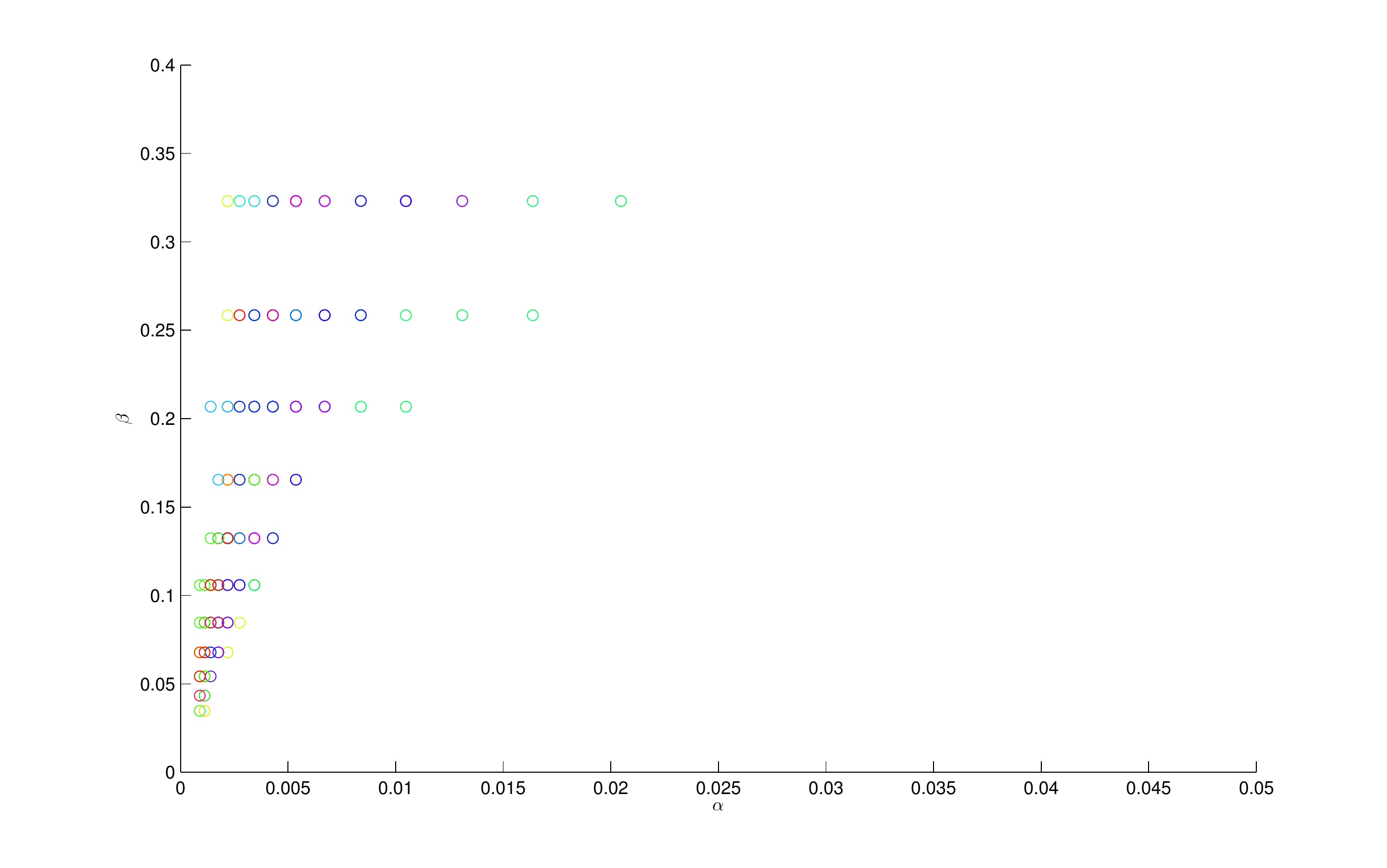}
\caption{The twenty best pairs of regularization parameters for each problem from the data set that for $p=0.3$ and $q \in \{2,4, \infty\}$ allows for the best AE.}
\label{fig:parameters0}
\end{figure}

\begin{figure}[hpt]
\includegraphics[width=0.32\textwidth]{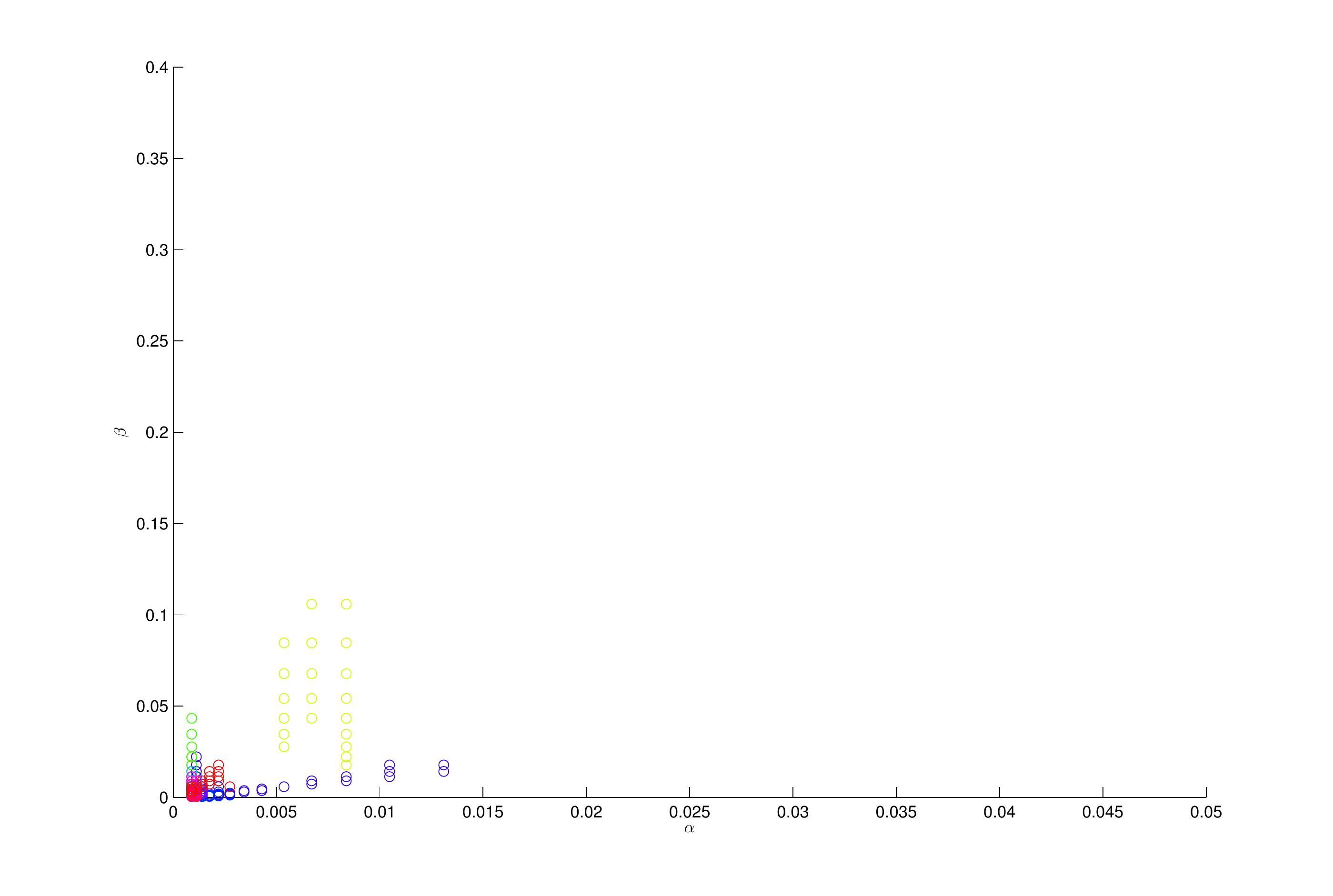}
\includegraphics[width=0.32\textwidth]{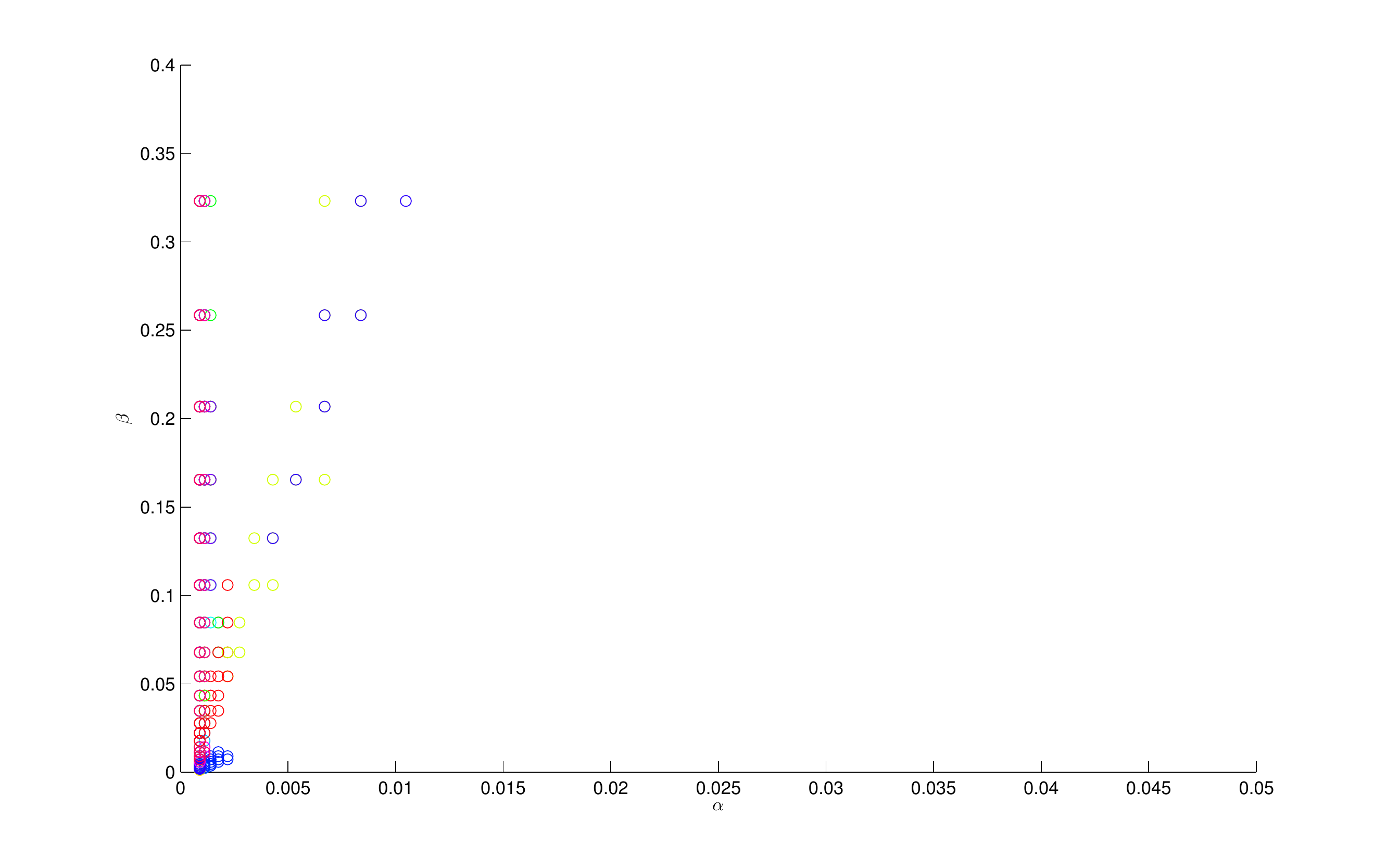}
\includegraphics[width=0.32\textwidth]{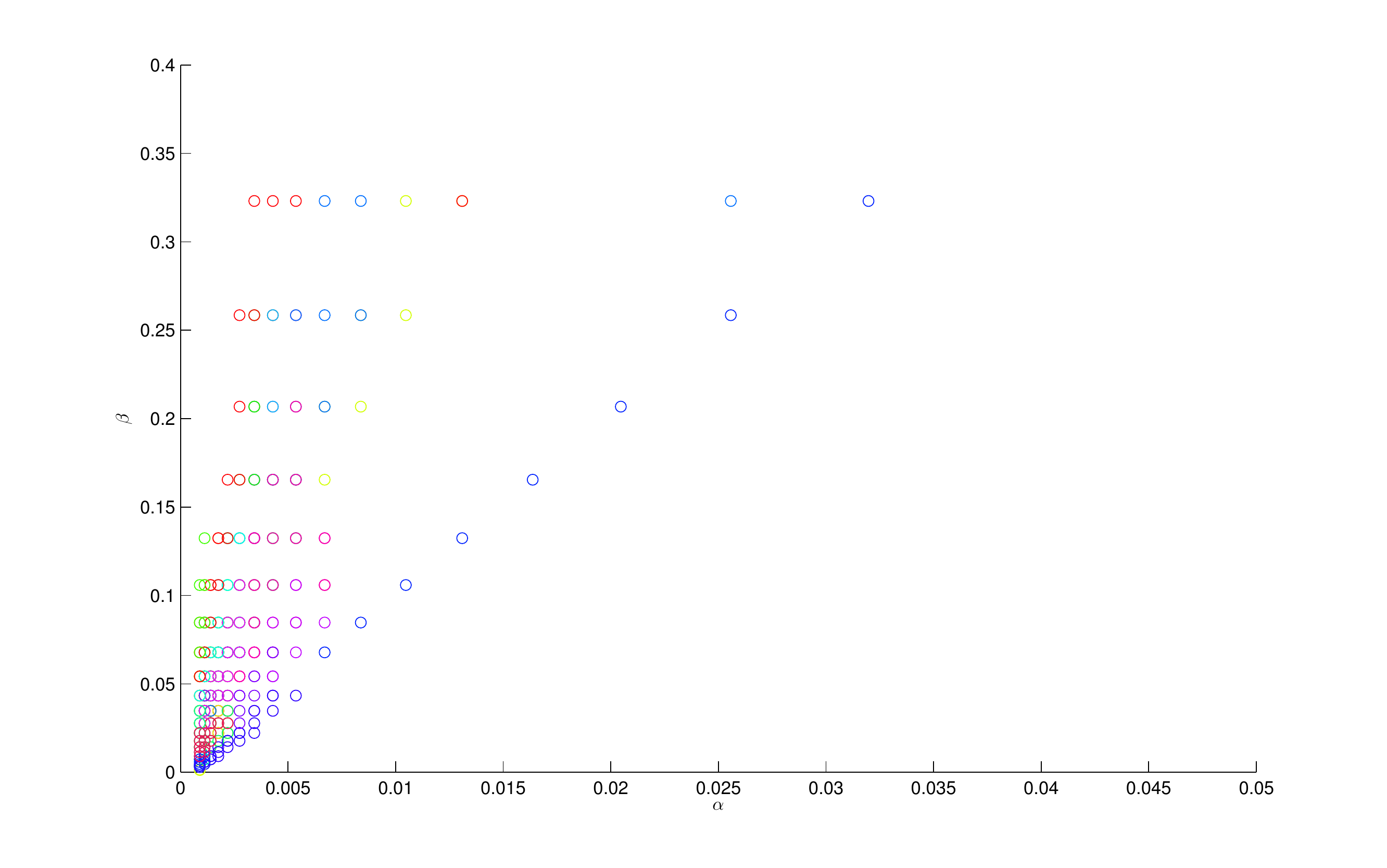}
\caption{The twenty best pairs of regularization parameters for each problem from the data set that for $p=0.3$ and $q \in \{2,4, \infty\}$ allows for the best SD.}
\label{fig:parameters}
\end{figure}

\subsection{Proposed guideline}
The analysis presented here suggests to perform the following best practice, when it comes to separate sparse solutions from additive noise in compressed sensing problems:

\begin{enumerate}
	\item Pick $p \in (0,1)$ and $q=2$;
	\item select regularization parameters $(\alpha, \beta)$ from a feasible region of parameters for a specific $(p,q)$ pair (see Figures \ref{fig:parameters0} and \ref{fig:parameters});
	\item iterate (\ref{algorithm_thresholding_fun}) for $u^{(0)} = v^{(0)}=0$ and $L,~M$ chosen at will until the desired accuracy is reached or the number of iterations exceeds some prescribed value. 
\end{enumerate}

\appendix

\section{Calculation for 2D example}
\label{app:A}
For representative purposes, we will show $u^{(n)}\in U_1^1$, for all $n\in\mathbb{N}$. Without loss of generality, we assume $y_2>y_1$ and prove the above statement by induction. By definition $u^{(1)}\in U_1^1$. It remains to show the induction step $u^{(n)} \in U_1^1 \Rightarrow u^{(n+1)}\in U_1^1$. Then, the repeated application of the induction step yields the statement. 

If $y-v^{(n)}\in U^1_1$, then there exists an $\hat{\alpha}$ such that $y-v^{(n)} = \mathbb S_{\hat{\alpha}}^1(y)$ and by a simple case-by-case analysis, one verifies $u^{(n+1)}=\mathbb S^1_{\alpha}(\mathbb S^1_{\bar{\alpha}}(y)) = \mathbb S^1_{\alpha+\bar{\alpha}}(y)\in U_1^1$. Thus, it remains to show $y-v^{(n)} \in U^1_1$. 

We know that by definition 
\begin{equation}
\label{eq:yminusv}
y-v^{(n)} = y- \mathbb S_{\beta}^{\infty}(y-u^{(n)}).
\end{equation}
Since, by induction hypothesis, $u^{(n)}\in U^1_1$, there exists a $\gamma$ such that $u^{(n)} = \mathbb S_{\gamma}^1(y)$. We choose an equivalent but more practical representation for elements in $U_1^1$ by employing an additional parameterization: There exist two cases:
	\begin{addmargin}[1cm]{1cm}

\begin{compactenum}[(A)]
 \item $u^{(n)} = \left(\begin{matrix}0 \\ \check{\gamma}\end{matrix}\right)$, for $\check{\gamma}\in [0,y_2-y_1]$;
 \item $u^{(n)} = \left(\begin{matrix}\hat{\gamma} \\ y_2-y_1 + \hat{\gamma}\end{matrix}\right)$, for $\hat{\gamma}\in [0,y_1]$.
 \end{compactenum}
\end{addmargin}

Each of these two cases has to be subdivided into sub-cases related to $\hat{\gamma}$ and $\check{\gamma}$. In the following table we summarize all sub-cases, an equivalent formulation in terms of the definition of $\mathbb S_{\beta}^{\infty}(y-u^{(n)})$, and the result of $y-v^{(n)}$.

\begin{tabular}{l|l|l|l}
& case & equivalent formulation & $y-v^n$ (by \eqref{eq:yminusv}) \\
\hline
A.1 & $\check{\gamma} > y_2 - \beta / 2 + y_1$ & $|y_1-u^{(n)}_1| + |y_2-u^{(n)}_2| < \beta / 2$ & $\left(\begin{matrix}y_1 \\ y_2 \end{matrix}\right)$ \\
\hline
A.2 & $\check{\gamma} < y_2 - \beta / 2 - y_1$ & $|y_1-u^{(n)}_1| < |y_2-u^{(n)}_2| - \beta / 2$ & $\left(\begin{matrix}0 \\ \check{\gamma} + \beta / 2 \end{matrix}\right)$ \\
\hline
A.3 & else & else & $\left(\begin{matrix} \frac{y_1-y_2+\beta / 2 +\check{\gamma}}{2} \\ (y_2-y_1 ) + \frac{y_1-y_2+\beta / 2 +\check{\gamma}}{2} \end{matrix}\right)$\\
\hline
B.1 & $\hat{\gamma} > y_1 - \beta / 4$ & $|y_1-u^{(n)}_1| + |y_2-u^{(n)}_2| < \beta / 2$ & $\left(\begin{matrix}y_1 \\ y_2 \end{matrix}\right)$ \\
\hline
B.2 & else & else & $\left(\begin{matrix} \hat{\gamma} + \beta / 4\\ (y_2-y_1 ) + \hat{\gamma} + \beta / 4 \end{matrix}\right)$
\end{tabular}

It remains to check for each case if the result of $y-v^{(n)}$ is an element of  $U_1^1$. Obviously in the cases A.1 and B.1, it is true. For the other cases, we check if the result can be expressed in terms of the above given practical representation of elements in $U_1^1$. In case A.2, by definition we have $0 \leq \check{\gamma} + \beta /2< y_2-y_1$. In  case A.3, it holds $y_2 - \beta / 2 - y_1 \leq \check{\gamma} \leq y_2 - \beta / 2 + y_1$ and thus obtain by adding $-y_2 + \beta / 2 + y_1$ and division by 2 that $ 0 \leq \frac{y_1-y_2+\beta / 2 +\check{\gamma}}{2} \leq y_1$. In case B.2, we immediately get  $\hat{\gamma} \leq y_1 - \beta / 4$ that $0 \leq \hat{\gamma} + \beta / 4 \leq y_1$. Thus, we have shown the statement for all cases.

\section*{Acknowledgements}

The  part of this work has been prepared, when Valeriya Naumova was staying at RICAM as a PostDoc. She gratefully acknowledges the partial support by the Austrian Fonds zur F\"orderung der Wissenschaftlichen Forschung (FWF), grant P25424 \enquote{Data-driven and problem-oriented choice of the regularization space}, and of the START-Project \enquote{Sparse Approximation and Optimization in High-Dimensions}. Steffen Peter acknowledges the support of the Project \enquote{SparsEO: Exploiting the Sparsity in Remote Sensing for Earth Observation} funded by Munich Aerospace.

\bibliographystyle{plain}
\bibliography{MinMPF_AIT_OPC}

\end{document}